\documentclass[10pt,a4paper,leqno]{amsart}

\usepackage[utf8]{inputenc}
\usepackage[T1]{fontenc}
\usepackage[english]{babel}
\usepackage{amsthm}
\usepackage{amsfonts}         
\usepackage{amsmath}
\usepackage{amssymb}
\usepackage{cases}
\usepackage{hyphenat}
\usepackage[psdextra, hidelinks]{hyperref}

\usepackage{mathrsfs}
\usepackage{bm}

\usepackage{xcolor}%

\def\XXint#1#2#3{{\setbox0=\hbox{$#1{#2#3}{\int}$ }
\vcenter{\hbox{$#2#3$ }}\kern-.6\wd0}}

\newcommand{\overbar}[1]{\mkern 1.5mu\overline{\mkern-1.5mu#1\mkern-1.5mu}\mkern 1.5mu}
\newcommand{\seq}[1]{{ \bar{#1} }} 

\newcommand{\sseqi}[3]{{#1}_{\seq{#2}^{#3}}}
\newcommand{\sseq}[2]{{\sseqi{#1}{#2}{}}}

\newcommand{\Bseq}[1]{{ \overbar{#1} }}
\newcommand{\seqm}{\Bseq{m}} 
\newcommand{\colM}{\{1, 2, \dots , m\}}

\newcommand{\wt}[1]{{\widetilde{#1}}}

\newcommand{\inv}[1]{\frac{1}{#1}}

\newcommand{\pinv}[1]{\inv{#1}} 

\newcommand{\avgfu}[1]{\frac{1_{#1}}{|#1|}}
\newcommand{\avgfr}[1]{\frac{1_{#1}}{|#1|^\pinv{2}}}

\newcommand{\sulk}[1]{\left(#1\right)}

\newcommand{\sulku}[1]{( #1 )}

\newcommand{\Bsulk}[1]{\Big(#1\Big)}
\newcommand{\biggsulk}[1]{\bigg(#1\bigg)}

\newcommand{\ksulk}[1]{\left\lbrace #1 \right\rbrace}

\newcommand{\ksulku}[1]{\lbrace #1 \rbrace}

\newcommand{\pair}[1]{\langle #1 \rangle}
\newcommand{\bpair}[1]{\big\langle #1 \big\rangle}
\newcommand{\Bpair}[1]{\Big\langle #1 \Big\rangle}

\newcommand{\abs}[1]{\left| #1 \right|}
\newcommand{\sabs}[1]{| #1 |}
\newcommand{\babs}[1]{\big| #1 \big|}
\newcommand{\Babs}[1]{\Big| #1 \Big|}

\newcommand{\bbrac}[1]{\big[ #1\big]}
\newcommand{\Bbrac}[1]{\Big[ #1 \Big]}
\newcommand{\bigbrac}[1]{\bigg[ #1\bigg]}

\newcommand{\norm}[2]{{\| #1 \|}_{#2}}
\newcommand{\Bnorm}[2]{\Big\| #1 \Big\|_{#2}}

\newcommand{\normBMO}[3]{\norm{#1}{\BMOW{#2}{#3}}}

\newcommand{\normbmo}[3]{\norm{#1}{\bmoW{#2}{#3}}}

\newcommand{\sqsumm}[2]{\Big(\sum_{#2} {#1}^2 \otimes \frac{1_{#2}}{\left| #2\right|} \Big)^\pinv{2}}

\newcommand{\BMOW}[2]{\BMO^{#1} ({#2})}
\newcommand{\bmoW}[2]{\bmo^{#1}({#2})}

\newcommand{\dd}{\mathrm d}

\newcommand{\bD}{\mathcal{D}}

\newcommand{\bS}{\mathcal{S}}

\newcommand{\bI}{\mathcal{I}}
\newcommand{\bJ}{\mathcal{J}}

\newcommand{\BMO}{\operatorname{BMO}}
\newcommand{\bmo}{\operatorname{bmo}}
\newcommand{\spt}{\operatorname{spt}}

\newcommand{\ch}{\operatorname{ch}}

\newcommand{\esssup}[1]{\mathop{\mathrm{ess\,sup}}_{#1}}

\newcommand{\R}{\mathbb{R}}
\newcommand{\C}{\mathbb{C}}
\newcommand{\Q}{\mathbb{Q}}

\newcommand{\E}{\mathbb{E}}

\swapnumbers
\theoremstyle{plain}
\newtheorem{thm}[equation]{Theorem}
\newtheorem{lem}[equation]{Lemma}
\newtheorem{prop}[equation]{Proposition}
\newtheorem{cor}[equation]{Corollary}

\theoremstyle{definition}

\theoremstyle{remark}
\newtheorem{rem}[equation]{Remark}

\numberwithin{equation}{section}

\author{Emil Airta}
\address[E.A.]{Department of Mathematics and Statistics, University of Helsinki, P.O.B. 68, FI-00014 University of Helsinki, Finland}
\email{emil.airta@helsinki.fi}

\title[Two-weight commutator estimates]{Two-weight commutator estimates: general multi-parameter framework}

\makeatletter
\@namedef{subjclassname@2010}{%
  \textup{2010} Mathematics Subject Classification}
\makeatother


\begin{document}

\begin{abstract}
We provide an explicit technical framework for proving very general two-weight commutator estimates in arbitrary parameters. The aim is to both
clarify existing literature, which often explicitly focuses on two parameters only, and to extend very recent results to the full generality of
arbitrary parameters.
More specifically, we study two-weight commutator estimates -- Bloom type estimates -- in the multi-parameter setting
involving weighted product BMO and little BMO spaces, and their combinations. 
\end{abstract}

\allowdisplaybreaks 
\maketitle
\section{Introduction}
Singular integral operators (SIOs) $T$ have the general form
$$
  Tf(x)=\int_{\R^d}K(x,y)f(y)\dd y.
$$
Varying the assumptions on the underlying kernel $K$ gives us many fundamental linear transformations arising naturally in pure and applied analysis.
One-parameter kernels are singular when $x=y$, while the {\em multi-parameter} theory deals with kernels with singularities on all hyperplanes of the form $x_i=y_i$, where $x,y\in\R^d$ are written in the form $x=(x_i)_{i=1}^m\in\R^{d_1}\times\cdots\times\R^{d_m}$ for a given partition $d=d_1+\ldots+d_m$. Compare, for example, the one-parameter Cauchy kernel $1/(x-y)^2$ to the bi-parameter kernel
$$\frac{1}{(x_1-y_1)(x_2-y_2)},$$ which is the product of Hilbert kernels in both coordinate directions of $\R \times \R = \R^2 = \C$. 
General multi-parameter kernels do not need to be of the product or convolution form, however. Fefferman--Stein \cite{Fefferman1982} deals with the convolution case, while Journ\'e \cite{Journe1985} develops more general theory.
However, we will be relying on the much more recent dyadic-probabilistic methodology -- see Martikainen \cite{Martikainen2012} for the original bi-parameter theory and Ou \cite{Ou2017} for the multi-parameter extensions.

Commutator estimates are a key part of modern harmonic analysis. Coifman--Rochberg--Weiss \cite{Coifman1976}
showed that
\begin{equation}\label{eq:commutatorbb}
\|b\|_{\BMO} \lesssim \|[b,T]\|_{L^p \to L^p} \lesssim \|b\|_{\BMO}, \textup{ where } [b,T]f := bTf - T(bf),
\end{equation}
for $p \in (1,\infty)$ and for some non-degenerate enough one-parameter SIOs $T$. In general, commutator estimates e.g. yield by duality
factorizations for Hardy functions, imply various div-curl lemmas relevant for compensated compactness, and have
connections to recent developments of the Jacobian problem $Ju = f$ in $L^p$ -- for the latter see Hyt\"onen \cite{Hytonen2018a}.
The field of multi-parameter commutator estimates has recently also been very active.
For evidence of the activity, see, for example, the paper
Duong--Li--Ou--Pipher--Wick \cite{Duong2019}, which studies the commutators of multi-parameter flag singular integrals. We get
to other recent multi-parameter commutator estimates momentarily.

Let $\mu$ and $\lambda$ be two general Radon measures in $\R^d$.
A two-weight problem asks for a characterisation of the boundedness $T \colon L^p(\mu) \to L^p(\lambda)$, where
$T$ can e.g. be an SIO.
For the two-weight characterisation for the Hilbert transform $T =H$, where $K(x,y) = 1/(x-y)$, see
Lacey \cite{Lacey2014} and Lacey, Sawyer, Uriarte-Tuero and Shen \cite{Lacey2014a} (see also Hyt\"onen \cite{Hytonen2018}).
The general higher dimensional theory has serious challenges, and there is no characterisation yet in the Riesz transform case.
However, recently the corresponding two-weight question in the commutator setting has seen a lot of attention and progress.
In these so-called Bloom type variants of the two-weight question we require that $\mu$ and $\lambda$ are Muckenhoupt $A_p$ weights and that the problem
involves a function $b$. The theory then concerns the triple $(\mu, \lambda, b)$, and the function $b$ will lie
in some appropriate weighted $\BMO$ space $\BMO(\nu)$ formed using the Bloom weight $\nu := \mu^{1/p}\lambda^{-1/p} \in A_2$.
Therefore, this means that for an operator $A^b$, depending naturally on some function $b$,
the Bloom type questions concern the estimate
$$
\|A^b\|_{L^p(\mu) \to L^p(\lambda)} \lesssim_{[\mu]_{A_p}, [\lambda]_{A_p}} \|b\|_{\BMO(\nu)}.
$$
In the natural commutator setting the corresponding lower bound
$$
\|b\|_{\BMO(\nu)} \lesssim_{[\mu]_{A_p}, [\lambda]_{A_p}} \|[b,T]\|_{L^p(\mu) \to L^p(\lambda)}
$$
is also of interest.
For the Hilbert transform $T = H$ Bloom \cite{Bloom1985} proved such a two-sided estimate -- hence the name of the theory.

In the much more recent works of
Holmes--Lacey--Wick \cite{Holmes2016, Holmes2017} Bloom's upper bound was proved for general bounded SIOs in all dimensions $\R^d$.
The lower bound was proved in the Riesz case.
Lerner--Ombrosi--Rivera-R\'ios \cite{Lerner2017} refined these results -- this time the proofs employed
sparse domination methods. An iterated commutator of the form $[b, [b,T]]$ is studied by Holmes--Wick \cite{Holmes2018}, when $b \in \BMO \cap \BMO(\nu)$. This iterated case also follows from the so-called Cauchy integral trick of Coifman--Rochberg--Weiss \cite{Coifman1976}, see
Hyt\"onen \cite{Hytonen2016}. This trick only works, though, as it is assumed that $b \in \BMO$. 
However, this assumption is not valid in the optimal case -- a fundamentally improved iterated case is by Lerner--Ombrosi--Rivera-R\'ios \cite{Lerner2019}, where $b \in \BMO(\nu^{1/2}) \supsetneq \BMO \cap \BMO(\nu)$. This is optimal: a lower bound is also proved in \cite{Lerner2019}. In the already mentioned paper \cite{Hytonen2018a} by Hyt\"onen lower bounds with very weak non-degeneracy assumptions were shown. 
Multilinear Bloom type inequalities are studied in the paper Kunwar--Ou \cite{Kunwar2018}.

We now get into bi-parameter and multi-parameter theory. Here the recent progress is most often based on the so-called representation theorems as sparse domination methods essentially currently work in one-parameter only (although see Barron--Pipher \cite{Barron2017}). In fact, Barron--Conde-Alonso--Ou--Rey \cite{Barron2019} show that one of the simplest bi-parameter model operators -- the dyadic bi-parameter maximal function -- cannot satisfy the most natural or useful candidate for bi-parameter sparse domination.
A representation theorem represents SIOs by some dyadic model operators (DMOs).  
To understand the upcoming discussion, we need to discuss some details regarding this.
The proofs of representation theorems are based on very careful refinements
of various $T1$ theorems (for the original one see David and Journ\'e \cite{DJ}) and dyadic--probabilistic methods (see Nazarov--Treil--Volberg \cite{Nazarov2003}).
Indeed, $T1$ theorems essentially exhibit
a decomposition of a standard SIO into its cancellative part and the so-called paraproducts. 
The one-parameter dyadic representation theorem of Hyt\"onen \cite{Hytonen2012} (extending e.g. Petermichl \cite{Petermichl2000}) then provides a further decomposition of the cancellative part into so-called {\em dyadic shifts}, which are generalisations of the Haar multipliers
\begin{equation*}\label{eq:HaarMult}
  f=\sum_{Q\in\mathcal D}\pair{f,h_Q}h_Q\mapsto \sum_{Q\in\mathcal D}\lambda_Q\pair{f,h_Q}h_Q, \qquad |\lambda_Q| \lesssim 1.
\end{equation*}
On the other hand, a paraproduct refers to an expression obtained by expanding both factors of the usual pointwise product $b\cdot f$ in some resolution of the identity, and dropping some of the terms in the resulting double expansion (so that it is not the full product). The $T1 \in \BMO$ assumptions in $T1$ theorems specifically deal with these paraproducts. The $T1$ theorems follow from representation theorems, but the real point is that the structural information
of representation theorems is key for proofs of many other results.

In bi-parameter we have paraproducts and cancellative shifts, but also their hybrid combinations. The latter are new in this setting, and are called partial paraproducts due to their hybrid nature. The pure bi-parameter
paraproducts are called full paraproducts. This leads to the following terminology: free of paraproducts (all paraproducts vanish) and free
of full paraproducts (the partial paraproducts need not vanish but the full paraproducts do). These can all be phrased with checkable $T1 = 0$ type conditions. Such conditions always hold
in the convolution case, and in some works these types of assumptions are made if the technology to handle the various paraproducts is not yet in place.
See Martikainen \cite{Martikainen2012} for the bi-parameter representation and Ou \cite{Ou2017} for the multi-parameter extension. The following terminology is also convenient: the term
SIO refers just to the kernel structure of our operators, while a Calder\'on--Zygmund operator (CZO) is an SIO satisfying appropriate $T1$ type conditions (and is thus bounded)

We are now ready to start our discussion of bi-parameter and multi-parameter commutators.
If $T$ is a bi-parameter CZO in $\R^{d_1 + d_2}$ the right thing for $[b,T]$ is that $b \in \bmo(\R^{d_1 + d_2})$ -- this means that $b(\cdot, x_2)$ and $b(x_1, \cdot)$ are uniformly in the usual BMO (this is one of the many equivalent ways to state this).  This so-called little $\BMO$ is a certain type of 
bounded mean oscillation space in bi-parameter, and it arises in commutators of this type, but in many other cases the so-called product $\BMO$ (denoted e.g. by $\BMO_{\textup{prod}}$) of Chang and Fefferman \cite{Chang1980, CF1985} involving general open sets is more fundamental.
If $T_{d_1}$ and $T_{d_2}$ are linear one-parameter CZOs in $\R^{d_1}$ and $\R^{d_2}$, respectively, then for
$$
[T_{d_1}^1, [b, T_{d_2}^2]]f = T_{d_1}^1(b T_{d_2}^2 f) - T_{d_1}^1T_{d_2}^2(bf) - b T_{d_2}^2 T_{d_1}^1 f + T_{d_2}^2(bT_{d_1}^1f),
$$
where $T_{d_1}^1 f(x) = T_{d_1}(f(\cdot, x_2))(x_1)$, the right object is $b \in \BMO_{\textup{prod}}(\R^{d_1+ d_2})$. 
In the Hilbert transform $T = H$ case references for these commutators include Ferguson--Sadosky \cite{Ferguson2000} and Ferguson--Lacey \cite{Ferguson2002}.
We note that Ferguson--Lacey \cite{Ferguson2002} contains the deep lower bound 
$$
\|b\|_{\BMO_{\textup{prod}}(\R^{2})} \lesssim \| [H^1, [b, H^2]]\|_{L^2 \to L^2}.
$$
See also Lacey--Petermichl--Pipher--Wick \cite{Lacey2009, Lacey2010, Lacey2012} for the higher dimensional Riesz setting and div-curl lemmas.

By bounding commutators of bi-parameter shifts Ou, Petermichl and Strouse \cite{Ou2016} proved that
 \begin{equation}\label{eq:OPS1}
 \|[b,T]\|_{L^2 \to L^2} \lesssim \|b\|_{\bmo},
 \end{equation}
when $T$ is a general bi-parameter CZO as in \cite{Martikainen2012} and is free of paraproducts. This is a very special case of their theorem -- we get to the full case later.
Holmes--Petermichl--Wick \cite{Holmes2018a} removed the paraproduct free assumption of \cite{Ou2016} and proved the first bi-parameter Bloom type estimate
$$
\|[b,T]\|_{L^p(\mu) \to L^p(\lambda)} \lesssim_{[\mu]_{A_p}, [\lambda]_{A_p}} \|b\|_{\bmo(\nu)}.
$$
Here $A_p$ stands for bi-parameter weights (replace cubes by rectangles in the usual definition) and $\bmo(\nu)$ is the weighted little $\BMO$ space defined using the norm
$$
\|b\|_{\bmo(\nu)} := \sup_{R} \frac{1}{\nu(R)} \int_R |b - \langle b \rangle_R|,
$$
where the supremum is over all rectangles $R = I_1 \times I_2 \subset \R^{d_1} \times \R^{d_2}$, $\langle b \rangle_R = \frac{1}{|R|} \int_R b$ and $\nu(R) = \int_R \nu$.

Recently, Li--Martikainen--Vuorinen \cite{Li2018Bloom} reproved the result of \cite{Holmes2018a} using a short proof based on some improved bi-parameter commutator decompositions from
their bilinear bi-parameter theory \cite{Li2018prodArxiv}. Importantly, the new proof also allowed them to handle the iterated little BMO commutator by showing that
$$
\| [b_k,\cdots[b_2, [b_1, T]]\cdots]\|_{L^p(\mu) \to L^p(\lambda)} \lesssim_{[\mu]_{A_p}, [\lambda]_{A_p}} \prod_{i=1}^k\|b_i\|_{\bmo(\nu^{1/k})}.
$$
They also recently showed the corresponding lower bound in \cite{Li2018prod} using the median method.
In \cite{Dalenc2016} Dalenc and Ou extended \cite{Fefferman1982} by proving that for all one-parameter CZOs
$$
\| [T_{d_1}^1, [b, T_{d_2}^2]] \|_{L^p \to L^p} \lesssim \|b\|_{\BMO_{\textup{prod}}}. 
$$
The two-weight version of this was recently proved in \cite{Li2018prod}:
\begin{equation}\label{eq:BloomProdBMO}
\| [T_{d_1}^1, [b, T_{d_2}^2]] \|_{L^p(\mu) \to L^p(\lambda)} \lesssim_{[\mu]_{A_p}, [\lambda]_{A_p}} \|b\|_{\BMO_{\textup{prod}}(\nu)}.
\end{equation}
This is the first two-weight Bloom estimate involving the most delicate (and important) bi-parameter $\BMO$ space -- the product $\BMO$.
In the weighted setting it can be defined by using the norm
\begin{equation}\label{eq:CFprodBMO}
\|b\|_{\BMO_{\textup{prod}}(\nu, \bD)} := 
\sup_{\Omega} \Big( \frac{1}{\nu(\Omega)} \mathop{\sum_{R \in \bD}}_{R \subset \Omega} |\langle b, h_R\rangle|^2 \langle \nu \rangle_{R}^{-1} \Big)^{1/2},
\end{equation}
where $\Omega$ is an open set, $\bD$ is a given cartesian product of some dyadic grids in $\R^{d_1}$ and $\R^{d_2}$, respectively, and the non-dyadic variant is a supremum over all such norms.

Our goal in this paper is to provide a careful proof of the analog of the estimate \eqref{eq:BloomProdBMO} in the case that the appearing singular integrals are multi-parameter, and in the case that we allow more singular integrals in the iteration. We want to provide an explicit proof in the multi-parameter setting, as they are very rare in the literature -- often bi-parameter results are proved and the corresponding multi-parameter results
are implicitly or explicitly claimed. This practice makes those result available only for a very few experts as often the details of the multi-paramter extensions
are actually very challenging -- both technically and notationally. The underlying general philosophies can be hard to understand from just the bi-parameter results.
So our focus is both on the explicit methodology unveiling the general principles, and also on extending the recent result \eqref{eq:BloomProdBMO} as much as we possibly can.

In \cite{Ou2016} the estimate \eqref{eq:OPS1} is used implicitly as a base case for more complicated multi-parameter commutator estimates.
For example, suppose that $T_1$ and $T_2$ are paraproduct free linear bi-parameter singular integrals satisfying the assumptions of the representation theorem \cite{Martikainen2012} in $\R^{d_1} \times \R^{d_2}$ and $\R^{d_3} \times \R^{d_4}$ respectively.
Then according to \cite{Ou2016} we have the estimate
\begin{align*}
\|[T_1, [b, T_2]]f&\|_{L^2(\prod_{i=1}^4 \R^{d_i})} \lesssim \max\big( \sup_{x_2, x_4} \|b(\cdot, x_2, \cdot, x_4)\|_{\BMO_{\operatorname{prod}}}, \\
 &\sup_{x_2, x_3} \|b(\cdot, x_2, x_3, \cdot)\|_{\BMO_{\operatorname{prod}}}, 
  \sup_{x_1, x_4} \|b(x_1, \cdot, \cdot, x_4)\|_{\BMO_{\operatorname{prod}}}, \\
  &\sup_{x_1, x_3} \|b(x_1, \cdot, x_3, \cdot)\|_{\BMO_{\operatorname{prod}}} \big) \|f\|_{L^2(\prod_{i=1}^4 \R^{d_i})}
\end{align*}
involving both the product $\BMO$ and little $\BMO$ philosophies.
The paraproduct free assumption can be removed according to \cite{Holmes2018a}. We prove results of this type in the two-weight Bloom case generalising \cite{Li2018prod} and \eqref{eq:BloomProdBMO}.
As a byproduct, we get explicit proof of unweighted multi-parameter estimates of \cite{Ou2016}.
The full methodology is included, which is key. 

\subsection*{Statement of the main results}
A small restriction in our theorems concerns the fact that the way we handle the hybrid paraproducts (partial paraproducts) requires sparse domination methods in one-parameter. This requires that
when our CZOs are not paraproduct free, they are at most bi-parameter -- otherwise the partial paraproducts would not be amenable to
sparse domination methods. This restriction comes from the methods of \cite{Li2018Bloom, Li2018prod}, which we adapt here. We have not found a way to estimate certain terms
without relying on these methods. However, if our CZOs are paraproduct free, they can be CZOs of arbitrarily many parameters. This recovers multipliers, convolution form CZOs and others.

\begin{thm}\label{thm:main}
We work in $\R^d$ with $m$-parameters, i.e., $d = d_1 + \ldots + d_m$. 
For a given $k \le m$ we want to have $k$ multi-parameter CZOs so that their parameters add up exactly to $m$.
Thus, let $\bI  = \ksulku{\bI_i}_{i=1}^k$ be a partition of $\{1, \ldots, m\}$, and for each $i = 1, \ldots, k$ let us be given
an $\#\bI_i$-parameter CZO $T_i$ in $\prod_{j \in \bI_i} \R^{d_j}$. Suppose, in addition, that for all $i = 1, \ldots, k$ at least one of the following conditions holds:
\begin{enumerate}
\item the CZO $T_i$ is paraproduct free, or
\item $\#\bI_i \le 2$.
\end{enumerate}

Let $\nu = \mu^\pinv{p} \lambda^{-\pinv{p}},$ where $\mu, \lambda \in A_p(\R^{d_1} \times \cdots \times \R^{d_m})$ ($m$-parameter $A_p$ weights in $\R^d$) and $1<p<\infty.$
Then for $b \colon \R^d \to \C$ we have
$$
\norm{[T_{1},[T_2,\dots [b,T_k]]]}{L^p(\mu) \rightarrow L^{p}(\lambda)}
\lesssim \normbmo{b}{\bI}{\nu},
$$
where we understand that in this formula $T_i$ acts on the whole space $\R^d$ -- i.e, $T_i  = T_i^{\bI_i}$ (see Section \ref{sec:basicnot} for this notation). 
Moreover, $\normbmo{b}{\bI}{\nu}$ is the suitable little product BMO -- see \eqref{eq:littleproductBMO}.
\end{thm} 

Structure of this paper is the following. In the begining of Section \ref{sec:defs}, we give the notation which we are going to use in the entire paper. Then we give the definitions and recall some standard estimates.

In Section \ref{sec:expansion}, we introduce expansions of function products and paraproduct operators. The main result of this section is to prove Bloom type upper bound for these multi-parameter paraproduct operators.

Then we split the study of our main theorem \ref{thm:main}. In Section \ref{sec:shift}, we consider the paraproduct free CZOs by first proving the results for multi-parameter shifts. Then using the representation theorem we get the result for paraproduct free CZOs. 

In Section \ref{sec:journe}, we begin with four parameter product space. We prove the case of the main theorem with two bi-parameter CZOs. The strategy of the proof is to use representation theorem such that it is enough to study commutators with DMOs. We illustrate how to prove Bloom type upper bound for these commutators by a careful study of a certain special case.
Then by iterating previous result and combining with the result of Section \ref{sec:shift} we get our main theorem \ref{thm:main}.

\subsection*{Acknowledgements}
E.A. was supported by the Academy of Finland through the grant 306901, and is a member of the Finnish Centre of Excellence in Analysis and Dynamics Research.
The author wishes to thank Kangwei Li for helpful discussions. This work is a part of the PhD thesis of the author
supervised by Henri Martikainen.
\section{Definitions and preliminaries}\label{sec:defs}
\subsection{Basic notation}\label{sec:basicnot}
We are working with the multi-parameter setting  in $\R^{d_1} \times \dots \times \R^{d_m}.$ We set $d_1 + d_2 + \dots + d_m = d.$
To avoid confusion we need consistent notation. 
For example, every $x \in \R^{d_1} \times \dots \times \R^{d_m}$ is a tuple $(x_1, x_2, \dots, x_m),$  where $x_i  \in \R^{d_i}.$ Similarly, every rectangle $I_1 \times I_2 \times \dots \times I_m \subset \R^{d_1} \times \dots \times \R^{d_m}$ consists of cubes $I_i \subset \R^{d_i}.$  Rather than writing each cube separately, we let $I_\seqm,$ where $\seqm$ denotes the vector $ (1,2,\dots, m),$ be a rectangle in $ \R^{d} = \R^{d_1} \times \dots \times \R^{d_m}$
and also for functions we can write $1_{I_1} \otimes 1_{I_2} \otimes \dots \otimes 1_{I_m} = 1_{I_1 \times \dots \times I_m} =: 1_{I_{\seqm}}.$  
Since it really does not matter what order tensor form functions are written, we do not distinguish between $1_{I_1} \otimes 1_{I_2}$ and $ 1_{I_2} \otimes  1_{I_1}.$ 

We often need operators to be defined only for some of the variables -- e.g.  for $f\colon\R^{d_1 +\dots + d_m}
\to \C$ and for some operator $U$ in $\R^{d_2},$  $U^2 f$ is defined as 
$$
U^{2} f (x)= (U f (x_1, \cdot, x_3, \dots, x_m) )(x_2).
$$
Notice that for example, for $g\colon \R^{d_1 + d_2} \rightarrow \C$ we would also have
$$
U^{2} g (x_1, x_2) = (U g (x_1, \cdot))(x_2).
$$
Since it is clear from the context, we do not make notational difference between these two. Additionally,  we always write  to the supscript the parameters where the operator is defined, i.e. for example, an operator $U^1$ is defined in $\R^{d_1}.$

Similarly, for integral pairings: 
$$
\pair{f, g}_1 (z) := \pair{f(\cdot, z), g}_{1},
$$
 where $f\colon  \R^{d_1 + d_2 + \dots + d_m} \rightarrow \C,$ $g\colon \R^{d_1} \rightarrow \C$ and $z \in \R^{d_2 + d_3 + \dots + d_m}.$ Although, for $\pair{f,g},$ where $f\colon  \R^{d} \rightarrow \C$ and $g\colon \R^{d} \rightarrow \C,$ it makes sense to leave out the parameters since in this case, the output of the pairing is a constant. Additionally, for example, for $f\colon \R^{d_1 + d_2 + \dots + d_m} \to \C$ we allow the notation $\pair{f, 1_{I_2}/|I_2| \otimes 1_{I_1}/|I_1|}_{2,1}$ and  understand it as $\pair{f,1_{I_1}/|I_1| \otimes 1_{I_2}/|I_2|}_{1,2}.$

 In addition, if $U^{v_1,v_2,\dots, v_n} $ is an operator in $\R^{d_{v_1} + d_{v_2} + \dots + d_{v_n} }$ for some subsequence $\seq{v} = (v_i)_{i=1}^n$ of $  \sulk{1,2,\dots, m},$ then we simply write  $U^{\seq{v}}f = U^{v_1,v_2,\dots, v_n} f.$  For brevity and clarity reasons,  for example, the operator $U^{2,1}$ is understood as the operator $U^{1,2}$ defined in $\R^{d_1 + d_2}.$

 Moreover, assume that  $\bI = \ksulku{\bI_1, \bI_2}$ is a partition of $\ksulku{1,2,\dots, m},$ that is, $\bigcup_{i} \bI_i = \ksulku{1,2,\dots, m}$ and $\bI_i$ are mutually disjoint, and $\bI_1 \neq \emptyset \neq \bI_2.$ Let $\seq{v} = (i)_{i \in \bI_1},$ $x_{\seq{v}} \in \prod_{i \in \bI_1} \R^{d_i},$ and let  $\varphi$ be a function $\R^d.$ Then we define that  $\varphi_{x_{\seq{v}}}$ is the obvious function defined on $\prod_{i \in \bI_2} \R^{d_i},$ where $x_{\seq{v}}$ has been fixed. For example, let $f\colon \R^{d_1 + d_2} \rightarrow \C$ and fix $ (x_1, x_2) \in \R^{d_1 +d_2},$ then $f_{x_1} (y_2) = f(x_1, y_2)$ and $f_{x_2} (y_1) = f(y_1, x_2)$ for all $(y_1, y_2) \in \R^{d_1 + d_2}.$

We denote dyadic grids in $\R^{d_i}$ by $\bD^{d_i}$ and $\bD^{\seq{d}}= \bD^{d_1,d_2,\dots,d_m} := \bD^{d_1} \times \bD^{d_2} \times \dots \times \bD^{d_m}.$ 
If $I_i \in \bD^{d_i}$, then $I_i^{(k_i)}$ denotes the
unique dyadic cube $Q_i \in \bD^{d_i}$ so that $I_i \subset Q_i$ and $\ell(Q_i) = 2^{k_i} \ell(I_i),$ where $\ell(J_i)$ denotes side length of a cube $J_i$. Similarly, for rectangles: if $I_{\seqm} \in \bD^{\seq{d}},$ then $I^{(\seq{k})}_{\seqm} = I_1^{(k_1)} \times I_2^{(k_2)} \times \dots \times I_m^{(k_m)}.$ In addition, for $I_i \in \bD^{d_i}$ we define  $\ch(I_i) = \ksulku{Q_i \in \bD^{d_i}: I_i = Q_i^{(1)}}.$

For $I_i \in \bD^{d_i}$ we denote by $h_{I_i}$ a cancellative $L^2$ normalised Haar function. Here we suppressed the presence of  $\eta \in \ksulku{0,1}^{d_i} \setminus \ksulku{0}.$ In particular, when we write $h_{I_i} h_{I_i}$ it can really stand for $h_{I_i}^{\eta_1} h_{I_i}^{\eta_2}$ for two different
$\eta_1, \eta_2$ -- however, this causes no problems as we only ever use the following size property $|h_{I_i} h_{I_i}| = 1_{I_i}/|{I_i}|.$ We recall some basic properties:
$\int h_{I_i} = 0,$ 
$\pair{h_{I_i}^{\eta_1}, h_{J_i}^{\eta_2}} = \delta_{I_i, J_i} \delta_{\eta_1, \eta_2},$
and
$ h_{I_i}^0 = 1_{I_i}/|I_i|^\pinv{2} = |h_{I_i}|.$

\subsection*{Martingale representation}

Let $\bD^{d_i}$ be some dyadic grid in $\R^{d_i}$ and suppose that $f$ is an appropriate function defined on $\R^d.$ Let $x = (x_1, \dots, x_m) \in \R^d$ and $\seq{u}_i = (j)_{j \in \colM \setminus \ksulku{i}}$ hence $x_{\seq{u}_i} = (x_j)_{j \in \ksulku{1,2,\dots , m} \setminus \{i \}}\in \prod_{j \in  \ksulku{1,2,\dots , m} \setminus \{i \}} \R^{d_j}.$ 

For $\varphi$ defined on $\R^{d_{\seq{v}}}$ and rectangle $Q_{\seq{v}} \subset \R^{d_{\seq{v}}}$ we denote  the integral average 
$$ \inv{|Q_\seq{v}|} \int_{Q_\seq{v}} \varphi $$
by $\pair{\varphi}_{Q_\seq{v}},$ where $\seq{v}$ is a subsequence of $(1,2,\dots, m).$ In addition, let $\ksulku{\bI_1, \bI_2}$ be a partition of $\ksulk{1,2,\dotsm m}$ such that $\bI_1 \neq \emptyset \neq \bI_2.$ We define $\pair{f}_{Q_{\seq{v}}, \seq{v}} \colon \R^{d_{\seq{v}'}} \rightarrow \C$ as 
$$
\pair{f}_{Q_{\seq{v}}, \seq{v}}(y_{\seq{v}'}) = \inv{|{Q_\seq{v}}|}\int_{Q_{\seq{v}}} f_{y_{\seq{v}'}}(y_{\seq{v}}) \dd y_\seq{v},
$$ where $y_{\seq{v}'} \in \R^{d_{\seq{v}'}},$ $\seq{v} = (i)_{i \in \bI_1}$ and  $\seq{v}' = (i)_{i\in \bI_2}.$

For all $i = 1,2, \dots, m,$ and $I_i \in \bD^{d_i}$ define the one-parameter martingale difference
\begin{align*}
\Delta_{I_i}^i f (x) &= (\Delta_{I_i}^i f_{x_{\seq{u}_i}})(x_i) = \sum_{Q_i \in \ch{(I_i)}} (\pair{f_{x_{\seq{u}_i}}}_{Q_i} - \pair{f_{x_{\seq{u}_i}}}_{I_i})1_{Q_i}(x_i).
\end{align*} 
The multi-parameter martingale difference is defined as iterated one-parameter martingale differences
$$
\Delta_{I_\seq{v}}^\seq{v} f (x) =  \Delta_{I_ {v_i}}^{v_i}  ( \Delta_{I_{\seq{v}_i'}}^{\seq{v}_i'} f)(x) = \dots = \Delta_{I_{v_i}}^{v_i}(\dots (\Delta_{I_{v_j}}^{v_j} f))(x),
$$
where $\seq{v} = (v_i)_{i = 1}^n$ is a subsequence of $\seqm,$ $I_\seq{v} = I_{v_1} \times \dots \times I_{v_n} \in \bD^{d_{\seq{v}}},$ $\seq{v}_i'$ is the sequence without the parameter $i$ and  order of the one-parameter martingale differences is arbitrary.  
Notice that  we have the following equality
$$
 \Delta_{I_i}^i f  = \sum_{\eta \in \ksulku{0,1}^{d_i}\setminus \ksulku{0}} \pair{f, h_{I_i}^\eta}_i \otimes h_{I_i}^\eta = \pair{f, h_{I_i}}_i \otimes h_{I_i},
$$
where $I_i \in \bD^{d_i}.$ 
Naturally, in the multi-parameter situation, we have
\begin{align*}
\Delta_{I_{\seq{v}}}^{\seq{v}} f &=   \pair{\Delta_{I_{\seq{v}'_i}}^{\seq{v}'_i}f, h_{I_{v_i}}}_{v_i}  \otimes h_{I_{v_i}} = \dots = \pair{f, h_{I_{v_1}} \otimes \dots \otimes h_{I_{v_n}}} h_{I_{v_1}} \otimes \dots h_{I_{v_n}} \\
&= \pair{f, h_{I_{\seq{v}}}}_{\seq{v}}  \otimes h_{I_\seq{v}},
\end{align*}
where $\seq{v} = (v_i)_{i = 1}^n$ is a subsequence of $\seqm,$ and  $\seq{v}_i'$ is the sequence without the parameter $i.$
Notice, if $\seq{v} = \seqm,$ then  
$$
\pair{f, h_{I_{\seq{v}}}}_{\seq{v}}  \otimes h_{I_\seq{v}} = \pair{f, h_{I_{\seqm}}} h_{I_\seqm}.
$$

Define the one-parameter martingale block
$$
\Delta_{K_i, k_i}^{i} f = \sum_{\substack{I_i \in \bD^{d_i}\\ I_i^{(k_i)} = K_i}} \Delta_{I_i}^i f$$
and the multi-parameter martingale block 
$$
\Delta_{K_{\seq{v}}, k_\seq{v}}^{\seq{v}} f  =  \sum_{\substack{I_\seq{v} \in \bD^{d_\seq{v}}\\ I_\seq{v}^{(k_\seq{v})} = K_\seq{v}}} \Delta_{I_\seq{v}}^\seq{v} f,
$$
where $\seq{v}$ is a subsequence of $\seqm.$

\subsection*{Square functions}
Define the one-parameter square function
$$
S_{\bD^{d_i}}^i f  = \Bsulk{\sum_{I_i \in \bD^{d_i}} |\Delta_{I_i}^i f|^2 }^\pinv{2} = \Bsulk{\sum_{I_i \in \bD^{d_i}} |\pair{ f, h_{I_i}}_{i}|^2 \otimes \avgfu{I_i} }^\pinv{2}.
$$ 
Since we will work on fixed dyadic grid, we abbreviate $S_{\bD^{d_i}}^i$ by $S^i.$ In addition, by the same reasoning we abbreviate 
$
\sum_{I_i \in \bD_i} 
$ by $\sum_{I_i},$ if there is no reason to emphasize  the dyadic grid, where the sum is taken over. 

Let $\seq{v}$ be a subsequence of $(1,2,\dots, m).$
Define the multi-parameter square function 
$$
S^{\seq{v}}_{\bD^{d_{\seq{v}}}} f =  \Bsulk{\sum_{I_\seq{v} \in \bD^{d_{\seq{v}}}} |\Delta_{I_\seq{v}}^\seq{v} f|^2 }^\pinv{2}.
$$
If $\seq{v}$ is a genuine subsequence, then we have $$ \Bsulk{\sum_{I_\seq{v} \in \bD^{d_{\seq{v}}}} |\Delta_{I_\seq{v}}^\seq{v} f|^2 }^\pinv{2} = \Bsulk{\sum_{I_\seq{v} \in \bD^{d_{\seq{v}}}} |\pair{ f, h_{I_\seq{v}}}_\seq{v}|^2 \otimes \avgfu{I_\seq{v}} }^\pinv{2}$$ and obviously,  for  $\seq{v} = \seqm$ we have
 $$ \Bsulk{\sum_{I_\seqm \in \bD^{d_{\seqm}}} |\Delta_{I_\seqm}^\seqm f|^2 }^\pinv{2} = \Bsulk{\sum_{I_\seqm \in \bD^{d_{\seqm}}} |\pair{ f, h_{I_\seqm}}|^2 \avgfu{I_\seqm} }^\pinv{2}.$$

\subsection*{Maximal functions}

Define the one-parameter dyadic maximal function
\begin{align*}
M^i_{\bD^{d_i}} f (x) = M^i_{\bD^{d_i}} f_{x_{\seq{u}_i}} (x_i) &= \sup_{\substack{Q_i \in \bD_{d_i} \\ Q_i \ni x_i}} \inv{|Q_i|} \int_{Q_{i}} \abs{f_{x_{\seq{u}_i}} (y_i)} \dd y_i  \\
&= \sup_{\substack{Q_i \in \bD^{d_i}}} 1_{Q_i} (x_i) \pair{|f_{x_{\seq{u}_i}}|}_{Q_i}.
\end{align*}

Similarly, as with the square functions, we suppress the dyadic grid, if there is no reason to specify it.

Let $\bI = \ksulk{\bI_1, \bI_2}$ be a partition of $\colM$ and we require that $\bI_1 \neq \emptyset.$ Let $\seq{v} = (i)_{i\in \bI_1}$ and $\seq{u} = (i)_{i \in \bI_2}.$  
Define the strong multi-parameter dyadic maximal function
\begin{align*}
M^{\seq{v}} f (x) = M^{\seq{v}} f_{x_{\seq{u}}} (x_{\seq{v}}) &= \sup_{Q_{\seq{v}}: Q_{\seq{v}} \ni x_{\seq{v}}} \inv{\sabs{Q_{\seq{v}}}}\int_{Q_\seq{v}} |f_{x_\seq{u}}(y_{\seq{v}})| \dd y_\seq{v}\\
&= \sup_{Q_{\seq{v}}} 1_{Q_\seq{v}} (x_{\seq{v}}) \pair{\sabs{ f_{x_{\seq{u}}}}}_{Q_{\seq{v}}},
\end{align*}
where the supremum is taken over the dyadic rectangles in $\bD^{d_\seq{v}}.$
If $\bI_2 = \emptyset,$ then $f_{x_{\seq{u}}}$ is understood as just $f.$

Observe that the strong maximal function is dominated by the iterated one-parameter maximal functions. For example, in the bi-parameter case  we have
\begin{align*}
\inv{|I_1 \times I_2|}\int_{I_1 \times I_2} |f(y_1,y_2)| \dd y_1 \dd y_2 &= \inv{| I_2|} \int_{I_2} \pair{\sabs{f}}_{I_1, 1} (y_2) \dd y_2 \\
&\leq  \inv{| I_2|} \int_{I_2} M^1 f (x_1, y_2) \dd y_2 \\
&\leq  M^2 M^1 f (x_1, x_2)
\end{align*}
for all $(x_1, x_2) \in I_1 \times I_2.$
Hence, the boundedness of the strong maximal function follows directly from the boundedness of the one-parameter maximal function.

\subsection*{Weights}
A weight $w$ ($w \in L^1_{loc}(\R^{d})$ and $0< w(x)< \infty$ a.e.) belongs to multi-parameter $A_p(\R^{d_1} \times \dots \times \R^{d_m}), 1 < p < \infty,$ if
 $$
 [w]_{A_p} := \sup_{I} \pair{w}_{I} \pair{w^{1-p'}}_{I}^{p - 1} < \infty,
 $$
 where the supremum is taken over $I = I_1 \times I_2 \times \dots \times I_m,$ where $I_i \subset \R^{d_i}$ are cubes with sides parallel to the axes.
We have
$w \in A_p(\R^{d_1} \times \dots \times \R^{d_m})$ if and only if for all $i = 1,2,\dots,m$ 
$$
 \esssup{x_{\seq{u}_i} } [w_{x_{\seq{u}_i}}]_{A_p( \R^{d_i})} < \infty,
$$
where the supremum  is taken over $x_{\seq{u}_i} = (x_j)_{j \in \ksulku{1,2,\dots , m} \setminus \{i \}} \in \prod_{j \in  \ksulku{1,2,\dots , m} \setminus \{i \}} \R^{d_j},$ and furthermore, we have
$$
\max_{i = 1,2, \dots,m}  \esssup{x_{\seq{u}_i} } [w_{x_{\seq{u}_i}}]_{A_p( \R^{d_i})} \leq [w]_{A_p(\R^d)}.
$$

We say that a weight $w$ belongs to one-parameter $A_\infty(\R^{d_i})$ if 
$$
[w]_{A_\infty(\R^{d_i})} = \sup_{I_i} \Bsulk{\inv{|I_i|} \int_{I_i} w }\exp \Bsulk{\inv{|I_i|}\int_{I_i} \log w^{-1}} < \infty, 
$$
where the supremum is taken over the cubes in $\R^{d_i}.$
Recall that in the one-parameter setting a weight $w$ belongs to $A_\infty(\R^{d_i})$ if $w \in A_p(\R^{d_i})$ for some $p < \infty.$

Hence, we say that a weight $w$ belongs to multi-parameter $A_\infty(\R^d)$ if $w_{x_{\seq{u}_i}}$ belongs to $A_{\infty}(\R^{d_i})$ uniformly for every parameter $i \in \ksulk{1,2,\dots, m}.$
\subsection{Standard estimates}
We record some standard estimates. These estimates and some estimates that follow from these are used implicitly in this paper.

First, we record  $A_\infty$-extrapolation result from Cruz-Uribe-Martell-Pérez \cite{Cruz-Uribe2004}.
\begin{lem}\label{lem:extrapol}
Let $(f, g)$ be a pair of positive functions defined on $\R^d$. Suppose that there exists some $0 < p_0 < \infty$ such that for every $w \in A_\infty(\R^d)$ we have
\begin{equation}\label{eq:extrapol}
\int_{\R^d} f^{p_0} w 	\leq C([w]_{A_\infty}, p_0) \int_{\R^d} g^{p_0} w.
\end{equation}

Then,
 for all $0 < p < \infty$ and $w \in A_\infty (\R^d)$ we have
$$
\int_{\R^d} f^{p} w 	\leq C([w]_{A_\infty}, p) \int_{\R^d} g^{p} w.
$$

In addition, let $\{(f_j,g_j)\}_j$ be a sequence of  pairs of positive functions defined on $\R^d.$ Suppose that for some $0 < p_0 < \infty$  pair $(f_j, g_j)$ satisfies inequality \eqref{eq:extrapol} for every $j.$
Then, for all $0 < p,q < \infty$ and $w \in A_\infty (\R^d)$ we have
$$
\Bnorm{\Bsulk{\sum_j (f_j)^q}^\pinv{q}}{L^p(w)} \lesssim_{[w]_{A_\infty}} \Bnorm{\Bsulk{\sum_j (g_j)^q}^\pinv{q}}{L^p(w)},
$$
where $\{(f_j,g_j)\}_j$ is a sequence of  pairs of positive functions defined on $\R^d.$
\end{lem}

\begin{lem}\label{lem:squareSim}
For $p \in (1,\infty)$ and $w \in A_p(\R^d)$ we have
$$
\norm{f}{L^p(w)} \sim_{[w]_{A_p}} \norm{S^{\seq{v}} f }{L^p(w)},
$$
where $\seq{v}$ is any subsequence of $(1,2,\dots,m).$
\end{lem}

\begin{rem}
In what follows we often assume that the appearing functions are nice -- a specific choice that works throughout the paper is that this is understood to mean bounded and compactly supported functions.
\end{rem}
 
\begin{lem}\label{lem:squareLower}
Let $\seq{v}$ be a subsequence of $(1,2,\dots,m)$ and $\seq{u}$ be a subsequence of $\seq{v}.$
Let $p \in (0,\infty)$ and $w \in A_\infty(\R^d).$ Then for nice function $f$ defined on $\R^d$ we have
$$
\norm{f}{L^p(w)} \lesssim_{[w]_{A_\infty}} \norm{S^{\seq{u}} f }{L^p(w)} \lesssim_{[w]_{A_\infty}}\norm{S^{\seq{v}} f }{L^p(w)}.
$$
\end{lem}
For  completeness we record the proof.
\begin{proof}
Suppose $\seq{v} = (v_1, v_2)$ and $\seq{u} = (v_1).$ Let $(f_j)_j$ be a sequence nice functions defined on $\R^d.$
Recall  the one-parameter result
\begin{equation*}
 \norm{f_{x_{\seq{u}'}}}{L^2(w_{x_{\seq{u}'}})} \lesssim_{[w]_{A_\infty}} \norm{S^{v_1} f_{x_{\seq{u}'}} }{L^2(w_{x_{\seq{u}'}})}, 
\end{equation*}
where  $x_{\seq{u}'} = (x_i)_{i \in \{1,2,\dots,m\} \setminus \{v_1\}}.$
Using this result  we get
\begin{align*}
\int_{\R^{d_{v_1}}} \Bsulk{\sum_{j} |f_j|^2}^{\pinv{2} \cdot 2}w &= \sum_{j} \int_{\R^{d_{v_1}}} |f_j|^2 w \\
&\lesssim\sum_j \sum_{I_{v_1}} \int_{\R^{d_{v_1}}} |\Delta_{I_{v_1}}^{v_1}f_j|^2 w \\
&= \int_{\R^{d_{v_1}}} \Bsulk{\sum_{j} \sum_{I_{v_1}} |\Delta_{I_{v_1}}^{v_1} f_j|^2}^{\pinv{2} \cdot 2}w
\end{align*}
for all $w \in A_\infty(\R^d).$ Here we abbreviated the fixed $x_{\seq{u}'}.$

Now, by $A_\infty$-extrapolation Lemma \ref{lem:extrapol} we have that 
$$
\Bnorm{\Bsulk{\sum_{j} |f_j|^2}^\pinv{2}}{L^p(w)} \lesssim_{[w]_{A_\infty}} \Bnorm{\Bsulk{\sum_{j}\sum_{I_{v_1}} |\Delta_{I_{v_1}}^{v_1} f_j|^2}^\pinv{2}}{L^p(w)} 
$$
for all $w \in A_\infty(\R^d)$ and $0 < p < \infty.$ Using this for $f_j = \Delta_{I_{v_2}}^{v_2} f$ we have
$$
\norm{f}{L^p (w)} \lesssim_{[w]_{A_\infty}} \Bnorm{\Bsulk{\sum_{I_{v_2}} |\Delta_{I_{v_2}}^{v_2} f|^2}^\pinv{2}}{L^p(w)} \lesssim_{[w]_{A_\infty}} \Bnorm{\Bsulk{\sum_{I_{v_1}, I_{v_2}} |\Delta_{I_{v_1} \times I_{v_2}}^{v_1, v_2}f|^2}^\pinv{2}}{L^p(w)} 
$$
for all $w \in A_\infty(\R^d)$ and $0 < p < \infty.$
It is clear that we can iterate the previous estimations for any number of parameters.
\end{proof}

Next, we record the multi-parameter Fefferman-Stein inequality, which follows from the classical one-parameter Fefferman-Stein inequality \cite{Fefferman1971} combined by the fact that strong maximal functions can be bounded with iterated one-parameter maximal functions.
\begin{lem}[Fefferman-Stein inequality]\label{lem:FS}
Let $\seq{v}$ be a subsequence of $(1,2,\dots,m).$
For $p, q \in (1,\infty)$ we have
$$
\Bnorm{\Bsulk{\sum_j | M^{\seq{v}} f_j|^q}^\pinv{q}}{L^p(w)} \lesssim_{[w]_{A_p}} \Bnorm{\Bsulk{\sum_j | f_j|^q}^\pinv{q}}{L^p(w)}.
$$
\end{lem}

Combining previous results we get the following result:
\begin{lem}\label{lem:FSgen}
Let $\bI = \ksulku{\bI_1, \bI_2}$ be a partition of $\ksulku{1,2,\dots, m}$ such that $\bI_1 \neq \emptyset \neq \bI_2.$ 
Let $\seq{v} = (i)_{i \in \bI_1}$ and $\seq{u} = (i)_{i \in \bI_2}.$
For $p \in (1,\infty)$ and $w \in A_p(\R^d)$ we have 
\begin{align*}
\Bnorm{\sum_{K_{\seq{v}}} M^{\seq{u}'}\bbrac{\pair{f, h_{K_{\seq{v}}}}_{\seq{v}}} \otimes h_{K_{\seq{v}}}}{L^p(w)} &\sim \Bnorm{\sqsumm{M^{\seq{u}'}\bbrac{\pair{f,h_{K_{\seq{v}}}}_{\seq{v}}}}{K_{\seq{v}}}}{L^p(w)} \\
&\lesssim_{[w]_{A_p}} \norm{f}{L^p(w)},
\end{align*}
where $\seq{u}'$ is a subsequence of $\seq{u}.$ 
\end{lem}

As explained earlier, we do not specify in the notation  the underlying space of the function that $M^{\seq{u}'}$ is operating. 
\begin{proof}

Notice that 
$$
S^{\seq{v}} \Bsulk{\sum_{K_{\seq{v}}} M^{\seq{u}'}\bbrac{\pair{f, h_{K_{\seq{v}}}}_{\seq{v}}} \otimes h_{K_{\seq{v}}}} =  \sqsumm{M^{\seq{u}'}\bbrac{\pair{f,h_{K_{\seq{v}}}}_{\seq{v}}}}{K_{\seq{v}}}.
$$
Hence, by Lemma \ref{lem:squareSim} we get the first conclusion
$$
\Bnorm{\sum_{K_{\seq{v}}} M^{\seq{u}'}\bbrac{\pair{f, h_{K_{\seq{v}}}}_{\seq{v}}} \otimes h_{K_{\seq{v}}}}{L^p(w)} \sim \Bnorm{\sqsumm{M^{\seq{u}'}\bbrac{\pair{f,h_{K_{\seq{v}}}}_{\seq{v}}}}{K_{\seq{v}}}}{L^p(w)}.
$$

Then using Lemma \ref{lem:FS} to the right-hand side of the previous estimate we get
 \begin{align*}
 \Bnorm{\sqsumm{M^{\seq{u}'}\bbrac{\pair{f,h_{K_{\seq{v}}}}_{\seq{v}}}}{K_{\seq{v}}}}{L^p(w)} &\lesssim \Bnorm{\sqsumm{\sabs{\pair{f,h_{K_{\seq{v}}}}_{\seq{v}}}}{K_{\seq{v}}}}{L^p(w)} \\
 &= \norm{S^{\seq{v}} f}{L^p(w)} \\
 &\lesssim \norm{f}{L^p(w)},
 \end{align*}
 where in the last step we use again Lemma \ref{lem:squareSim}.
\end{proof}

\subsection{Product BMO\texorpdfstring{$^{\bm{\seq{v}}}$}{v }  spaces}
Let  $\seq{v} = (v_i)_{i = 1}^n$ be a subsequence of $(1,2,\dots, m),$ let $w$ be a multi-parameter $A_\infty$ weight on  $\R^{d}  $ and $b \in L^1_{loc}(\R^d).$ Also let $\bD^{\seq{v}} = \prod_{i = 1}^n \bD^{v_i}$ be a product of dyadic grids.
We say that $b \in \BMOW{\seq{v}}{w; \bD^{\seq{v}}}$ if for  all nice functions $\varphi$ such that 
$
\norm{S^\seq{v}_{\bD^{\seq{v}}} \varphi }{L^1(w)} < \infty
$ we have 
$$
|\pair{b, \varphi}| \leq C_b \norm{S^\seq{v}_{\bD^{\seq{v}}} \varphi }{L^1(w)}.
$$
Then we denote the best constant $C_b$ by $\normBMO{b}{\seq{v}}{w; \bD^{\seq{v}}}.$

In addition, if $b \in \BMOW{\seq{v}}{w; \bD^{\seq{v}}} $ for all $\bD^{\seq{v}},$ then we say that $b \in \BMOW{\seq{v}}{w}.$

Furthermore, let us define the \emph{little product BMO}. 
Let $k \leq m$ and let $\bI = \ksulku{\bI_i: i\leq k}$ be a partition of $\ksulku{1,2, \dots, m}$ such that $\bI_i \neq \emptyset$ for $i =1,2\dots,k.$ We say that $b \in \bmoW{\bI}{w}$ if for all $\seq{v} = (v_j)_{j =1}^k$ such that $v_j \in \bI_j,$ we have $b\in \BMOW{\seq{v}}{w}.$ Then we set
\begin{equation}\label{eq:littleproductBMO}
\norm{b}{\bmoW{\bI}{w}} := \max_{\seq{v}} \normBMO{b}{\seq{v}}{w},
\end{equation}
where the maximum is taken over $\seq{v} = (v_j)_{j =1}^k$ such that $v_j \in \bI_j.$

For example, let $m = 3,$ $w = 1,$ $\bI_1 = \ksulku{1,2}$ and $\bI_2 = \ksulku{3}.$ Then $b \in \bmo^{\ksulk{\ksulk{1,2}, \ksulku{3}}},$ if $b \in \BMO^{(1,3)}$ and $b \in \BMO^{(2,3)}.$

\begin{rem}
We prefer this more direct square function definition over the typical square sum definition as in the introduction \eqref{eq:CFprodBMO}.
\end{rem}

\begin{rem}
If $k = 1,$ we have the standard multi-parameter little BMO space. For more details in the bi-parameter framework see e.g. \cite{Holmes2018a, Li2018Bloom}.
\end{rem}

\begin{prop}\label{prop:uniformBMO}
Let $\bI = \ksulk{\bI_1, \bI_2}$ be a partition of $\ksulk{1,2,\dots, m}$ such that $\bI_1 \neq \emptyset \neq \bI_2.$ 
Let $\seq{v} = (i)_{i \in \bI_1},$ $\seq{u} = (i)_{i \in \bI_2},$ $w$ be a multi-parameter $A_\infty$ weight on $\R^d$ and $b \in L^1_{loc}(\R^d).$ Then $b \in \BMOW{\seq{v}}{w}$ if and only if $b_{x_{\seq{u}}} \in \BMO^{\seq{v}}(w_{x_{\seq{u}}})$ uniformly on $x_{\seq{u}} \in \prod_{i \in \bI_2} \R^{d_i}$, that is,  we have
\begin{equation}\label{eq:uniformBMO}
|\pair{b_{x_{\seq{u}}}, f}| \leq C_b  \norm{S^{\seq{v}} f}{L^1(w_{x_{\seq{u}}})}
\end{equation}
for almost every $x_{\seq{u}} \in \prod_{i \in \bI_2} \R^{d_i}$ and for every nice $f$  defined on $\prod_{i \in \bI_1} \R^{d_i}.$ 
\end{prop}

\begin{proof}
First, suppose  that $b_{x_{\seq{u}}} \in \BMO^{\seq{v}}(w_{x_{\seq{u}}})$ uniformly on $x_{\seq{u}} \in \prod_{i \in \bI_2} \R^{d_i}.$
Let $\varphi$ be a nice function defined on $\R^d.$
Hence, we have
\begin{align*}
\sabs{\pair{b, \varphi}} &\leq \int_{\R^{d_{\seq{u}}}} \sabs{\pair{b_{x_{\seq{u}}}, \varphi_{x_{\seq{u}}}}} \dd x_{\seq{u}} \\
&\stackrel{\eqref{eq:uniformBMO}}{\leq} C_b \int_{\R^{d_{\seq{u}}}} \norm{S^{\seq{v}} \varphi_{x_\seq{u}}}{L^1(w_{x_\seq{u}})} \dd x_{\seq{u}} \\
&= C_b \norm{S^{\seq{v}} \varphi}{L^1(w)},
\end{align*}
as desired.

Suppose, conversely, that $b \in \BMOW{\seq{v}}{w}.$
Let $\varphi = \varphi_{1} \otimes \varphi_2$ be a function defined on $\prod_{i \in \bI_1} \R^{d_i} \times \prod_{i \in \bI_2} \R^{d_i}.$  Then we have
\begin{equation}\label{eq:propBMO}
|\pair{\pair{b, \varphi_1}_{\seq{v}}, \varphi_2}| = |\pair{b,\varphi}| \lesssim_{ \normBMO{b}{\seq{v}}{w}} \norm{S^{\seq{v}} \varphi}{L^1(w)} = \int_{\R^{d_{\seq{u}}}} \norm{S^{\seq{v}} \varphi_{1}}{L^1 (w_{y_\seq{u}})} \varphi_2(y_{\seq{u}}) \dd y_{\seq{u}}.
\end{equation}
Let $\varphi_1 = f \colon \prod_{i \in \bI_1} \R^{d_i} \rightarrow \C$ and for fixed $x_{\seq{u}} \in \prod_{i \in \bI_2} \R^{d_i}$  and $r > 0$ let
$$
\varphi_2 (y_\seq{u}) = \frac{1_{B(x_{\seq{u}}, r)} (y_\seq{u})}{|B(x_{\seq{u}}, r)|}.
$$  
By Lebesgue differentiation theorem, the left-hand side of \eqref{eq:propBMO} converges to $|\pair{b_{x_{\seq{u}}}, f}|$  for almost every $x_{\seq{u}}.$

 By same argument, the right-hand side of \eqref{eq:propBMO} converges to $\norm{S^{\seq{v}} f}{L^1(w_{x_{\seq{u}}})}$ for almost every $x_{\seq{u}}.$
Hence, we have
$$
|\pair{b_{x_{\seq{u}}}, f}| \leq \normBMO{b}{\seq{v}}{w} \norm{S^{\seq{v}} f}{L^1(w_{x_{\seq{u}}})}
$$
for almost every $x_{\seq{u}}$ and for every nice function $f$ defined on $\R^{d_\seq{v}}.$ 
\end{proof}

We record the following $\BMO$ embedding result and we use this fact implicitly later on.
\begin{lem}\label{lem:lBMO} Let $\seq{u}$ be a subsequence of $(1,2,\dots, m)$ and $\seq{v}$ be a subsequence of $\seq{u}.$  Suppose $b \in \BMO^{\seq{v}}(w).$ 
Then we have
$$
|\pair{b, f}| \lesssim_{[w]_{A_\infty}}  \normBMO{b}{\seq{v}}{w}\norm{S^{\seq{u}} f}{L^1(w)},
$$
that is, $\BMO^{\seq{v}}(w)  \subset \BMOW{\seq{u}}{w}.$ 
\end{lem}
\begin{proof}
The claim follows from the definition and Lemma \ref{lem:squareLower}, namely 
\begin{align*}
|\pair{b, f}| &\leq \normBMO{b}{\seq{v}}{w} \norm{S^{\seq{v}} f}{L^1{(w)}} \\
& \lesssim_{[w]_{A_\infty}} \normBMO{b}{\seq{v}}{w}\norm{S^{\seq{u}} f}{L^1(w)}.
\end{align*}
\end{proof}

\subsection{Singular integral operators}
We define multi-parameter SIOs. For brevity, we give an explicit definition only in bi-parameter. A general $m$-parameter definition can be found in Journ\'e \cite{Journe1985}, but in a different operator-valued language.
Our definition is as in \cite{Martikainen2012}, and is, in fact, equivalent to that given by Journ\'e as proved by Grau de la Herrán \cite{Grau2016}. An $m$-parameter definition using our partial kernel/full kernel language
is explicitly given in Ou \cite{Ou2017}.

Let $\alpha \in (0,1].$ We say that $T$ is a bi-parameter singular integral operator (SIO) if the kernel representations below are satisfied.

Furhermore, if, in addition to kernel representations, $T$ satisfies also some certain boundedness and cancellation assumptions, $T1$ assumptions, we say that $T$ is a Calderón-Zygmund operator (CZO). These boundedness and cancellation assumptions are equivalent with  $L^2$-boundedness of $T$ and its partial adjoint defined below.

For Calderón-Zygmund operators, we have the representation theorems \cite{Martikainen2012,Ou2017} using the dyadic model operators, namely paraproducts and shifts. The definitions of these model operators are presented later. We say that CZO $T$ is a paraproduct free Calderón-Zygmund operator if it can be represented using only the dyadic shifts.

\subsubsection{Full kernel representation}
If $f = f_1 \otimes f_2$ and $g = g_1 \otimes g_2$ with
$f_1, g_1 \colon\R^{d_1} \to \C,$ $f_2, g_2 \colon\R^{d_2} \to \C,$ $\spt f_1 \cap \spt g_1 = \emptyset$ and $\spt f_2 \cap \spt g_2 = \emptyset,$ then we have the
kernel representation
$$
\pair{Tf, g} = \int_{\R^{d_1 + d_2}} \int_{\R^{d_1 + d_2}}  K(x, y) f(y) g(x) \dd x \dd y.
$$
The so-called full kernel 
$$
K\colon (\R^{d_1 + d_2}\times \R^{d_1 +d_2}) \setminus \ksulk{(x,y) \in\R^{d_1 + d_2}\times \R^{d_1 +d_2}\colon x_1 = y_1 \text{ or } x_2 = y_2} \to \C 
$$
is assumed to satisfy the size condition
$$
|K(x,y)| \leq C \frac{1}{|x_1 - y_1|^{d_1}}\frac{1}{|x_2 - y_2|^{d_2}},
$$
 the Hölder condition
\begin{align*}
|K(x,y) - K(x, (y_1, y_2')) - K(x,(y_1',y_2)) -K(x,y')| \leq C \frac{|y_1 - y_1'|^\alpha}{|x_1 - y_1|^{d_1 + \alpha}}\frac{|y_2 - y_2'|^\alpha}{|x_2 - y_2|^{d_2 + \alpha}}
\end{align*}
whenever $|y_1 - y_1'| \leq |x_1 - y_1|/2$ and $|y_2 - y_2'| \leq |x_2 - y_2|/2,$
and the mixed Hölder and size condition
$$
|K(x,y) - K(x, (y_1', y_2))| \leq C \frac{|y_1 - y_1'|^\alpha}{|x_1 - y_1|^{d_1 + \alpha}}\frac{1}{|x_2 - y_2|^{d_2}}
$$
whenever $|y_1 - y_1'| \leq |x_1 - y_1|/2.$

Notice that this implies the kernel representation for $T^*, T^{1*}$ and $ T^{2*} = (T^{1*})^*,$ where $T^*$ is the usual adjoint and $T^{1*}$ is the partial adjoint defined by
$$
\pair{T^{1*}(f_1 \otimes f_2), g_1 \otimes g_2} = \pair{T(g_1 \otimes f_2), f_1 \otimes g_2}.
$$
Say that $K^*,K^{1*}$ and $ K^{2*}$ are the respective kernels of these, then we can write
 \begin{align*}
K^*((x_1,x_2),(y_1,y_2)) &= K((y_1,y_2),(x_1,x_2)) \\
 K^{1*}((x_1,x_2),(y_1,y_2))&= K((y_1,x_2), (x_1,y_2)) \\
  K^{2*}((x_1,x_2),(y_1,y_2))&= K((x_1,y_2), (y_1,x_2)).
 \end{align*}
We assume above size and Hölder conditions also  for $K^*,K^{1*}$ and $K^{2*}.$

\subsubsection{Partial kernel representation}
If $f = f_1 \otimes f_2$ and $g = g_1 \otimes g_2$ with $\spt f_1 \cap \spt g_1 = \emptyset,$ then we assume the kernel representation 
$$
\pair{Tf, g} = \int_{\R^{d_1}} \int_{\R^{d_1}} K_{f_2,g_2}(x_1,y_1) f_1(y_1) g_1(x_1)\dd x_1 \dd y_1.
$$
The kernel 
$$
K_{f_2,g_2}\colon \ksulk{(x_1,y_1) \in\R^{d_1}\times \R^{d_1}\colon x_1 \neq y_1} \to \C 
$$
is assumed to satisfy the size condition
$$
|K_{f_2,g_2}(x_1,y_1)| \leq C(f_2,g_2) \frac{1}{|x_1 - y_1|^{d_1}}
$$
and the Hölder conditions

$$
|K_{f_2,g_2}(x_1,y_1) - K_{f_2,g_2}(x_1',y_1)| \leq C(f_2,g_2) \frac{|x_1 - x_1'|^{\alpha}}{|x_1 - y_1|^{d_1 + \alpha}}
$$
whenever $|x_1 - x_1'| \leq |x_1 - y_1|/2$ and
$$
|K_{f_2,g_2}(x_1,y_1) - K_{f_2,g_2}(x_1,y_1')| \leq C(f_2,g_2) \frac{|y_1 - y_1'|^{\alpha}}{|x_1 - y_1|^{d_1 + \alpha}}
$$
whenever $|y_1 - y_1'| \leq |x_1 - y_1|/2.$ 
We require the following control on the constant $C(f_2,g_2).$ For every cube $I_2 \subset \R^{d_2}$ we assume that $C(1_{I_2}, 1_{I_2}) + C(\varphi_{I_2}, 1_{I_2}) + C(1_{I_2}, \varphi_{I_2}) \lesssim |I_2|,$ where $\varphi_{I_2}$ is supported on $I_2,$ $\int \varphi_{I_2} = 0$ and $|\varphi_{I_2}| \le 1$.

Analogously, we assume similar presentation and properties with $K_{f_1,g_1}$ whenever $\spt f_2 \cap \spt g_2 = \emptyset.$

\section{Paraproduct operators and martingale difference expansions of products}\label{sec:expansion}
We assume that operators in this section are defined in some fixed dyadic grids $\bD^{\seq{d}} = \prod_{i = 1}^m \bD^{d_i}.$

Define the one-parameter paraproduct operators 
$$
A_1^i (b,f) = \sum_{I_i \in \bD^{d_i}} \Delta_{I_i}^i b \Delta_{I_i}^i f, \quad A_2^i (b,f) = \sum_{I_i \in \bD^{d_i}} \Delta_{I_i}^i b E_{I_i}^i f, \quad A_3^i(b,f) =\sum_{I_i \in \bD^{d_i}}  E_{I_i}^i b \Delta_{I_i}^i f,
$$
where $E_{I_i} \varphi := \pair{\varphi}_{I_i,i} 1_{I_i}.$ We call the last term the ``illegal'' paraproduct.

Then we define the multi-parameter paraproduct operators as iterated one\hyp{}parameter paraproducts 
 -- e.g. for $\seq{v} = (v_i)_{i = 1}^n, n\leq m$ and $\seq{i}  = (i_1,i_2,\dots, i_n)$ with $i_1 =2 $ and $i_j  = 3$ for all $j \neq 1$ we have
\begin{align*}
A_{\seq{i}}^\seq{v} (b,f) &= A_{i_1}^{v_1} A_{i_2}^{v_2} \dots A_{i_n}^{v_n} (b,f) = \sum_{I_{v_1}} A_{i_2}^{v_2} \dots A_{i_n}^{v_n} (\Delta_{I_{v_1}}^{v_1} b, E_{I_{v_1}}^{v_1}f)  \\
&=\sum_{I_{v_1}, I_{v_2}} A_{i_3}^{v_3} \dots A_{i_n}^{v_n} (\Delta_{I_{v_1}}^{v_1} E_{I_{v_2}}^{v_2}b,  E_{I_{v_1}}^{v_1} \Delta_{I_{v_2}}^{v_2}f) \\
&= \dots = \sum_{I_{v_1}, I_{v_2}, \dots, I_{v_n}} \Delta_{I_{v_1}}^{v_1} E_{I_{\seq{v}'}}^{\seq{v}'} b E_{I_{v_1}}^{v_1} \Delta_{ I_{\seq{v}'}}^{\seq{v}'} f,
\end{align*}
where $\seq{v}' = (v_i)_{i \in \ksulk{2,3,\dots,n}}.$

We write
$$
bf  = \sum_{i = 1}^3 A_{i}^j (b,f),
$$
and we say that this is  the one-parameter expansion of the product $bf$ in the parameter $j.$
Then the multi-parameter expansion is obtained by iterating the previous one-parameter expansion -- e.g. let $\seq{v} = (v_i)_{i =1}^n$ be a subsequence of $\seqm.$ Then expansion in the parameters $\seq{v}$ is  
$$
bf = \sum_{\seq{i} \in \ksulku{1,2,3}^{n}} A_{\seq{i}}^{\seq{v}}(b,f) = \sum_{i_1,i_2,\dots,i_n = 1}^3 A_{i_1}^{v_1} A_{i_2}^{v_2} \dots A_{i_n}^{v_n} (b,f),
$$
where the ``illegal'' paraproduct is  the one with $\seq{i} = \ksulku{3}^n.$ We want to emphasize the paraproducts are directly bounded with some BMO assumption if $\seq{i} \neq \ksulku{3}^n,$ as we are going to next show, hence the name ``illegal''.

\begin{lem}\label{lem:paraBound}Let $\bJ$ be a subset of $\colM$ with  $n \leq m$ elements. Let  $\seq{v} = (v_j)_{v_j \in \bJ},$  $\seq{i} = (i_1,i_2, \dots, i_n) \in \ksulk{1,2,3}^{n} \setminus\ksulku{3}^n$ and  $\seq{u} = (u_j),$ $u_j \in  \{k \in \bJ: i_k \neq 3\}.$ 
 Also let $\nu = \mu^\pinv{p} \lambda^{-\pinv{p}},$ where $\mu, \lambda \in A_p(\R^{d_1} \times \dots \times \R^{d_m})$ and $1<p<\infty.$ Then
$$
\norm{A_{\seq{i}}^{\seq{v}}(b, \cdot)}{L^p{(\mu)} \rightarrow L^p(\lambda)} \lesssim_{[\mu]_{A_p}[\lambda]_{A_p}} \normBMO{b}{\seq{u}}{\nu}.
$$
\end{lem}

\begin{proof}
Since we have three different type of one-parameter paraproducts, let $\bI = \{\bI_i\}_{i = 1}^3$ be a partition of $\bJ$ such that $\bI_k = \ksulk{j \in \bJ: i_j = k}.$ Notice we require in the statement that $\bI_1 \cup \bI_2 \neq \emptyset.$ We set $\seq{v}^k = (v^k_j)_{v^k_j \in \bI_k}.$ If some set $\bI_k = \emptyset$ or $\bJ = \colM,$ it is fairly obvious what steps are not necessary and we omit the details. 
By the partition, we are considering the term
\begin{align*}
A^{\seq{v}}_{\seq{i}}(b,f) = \sum_{I_{\seq{v}^1}, I_{\seq{v}^2}, I_{\seq{v}^3}} &\Bpair{b, h_{I_{\seq{v}^1}} \otimes h_{I_{\seq{v}^2}} \otimes \avgfu{I_{\seq{v}^3}}}_{\seq{v}^1,\seq{v}^2,\seq{v}^3} \Bpair{f, h_{I_{\seq{v}^1}} \otimes \avgfu{I_{\seq{v}^2}} \otimes h_{I_{\seq{v}^3}}}_{\seq{v}^1,\seq{v}^2,\seq{v}^3} \\
&\otimes h_{I_{\seq{v}^1}} h_{I_{\seq{v}^1}} \otimes h_{I_{\seq{v}^2}} \otimes h_{I_{\seq{v}^3}}.
\end{align*}

Begin the estimation with the dual form 
\begin{align*}
\Babs{\int_{\R^{d_{\seq{v}'}}}\sum_{I_{\seq{v}^1}, I_{\seq{v}^2}, I_{\seq{v}^3}} &\Bpair{b, h_{I_{\seq{v}^1}} \otimes h_{I_{\seq{v}^2}} \otimes \avgfu{I_{\seq{v}^3}}}_{\seq{v}^1,\seq{v}^2,\seq{v}^3} \Bpair{f, h_{I_{\seq{v}^1}} \otimes \avgfu{I_{\seq{v}^2}} \otimes h_{I_{\seq{v}^3}}}_{\seq{v}^1,\seq{v}^2,\seq{v}^3} \\
&\times\pair{g, h_{I_{\seq{v}^1}} h_{I_{\seq{v}^1}} \otimes h_{I_{\seq{v}^2}} \otimes h_{I_{\seq{v}^3}}}_{\seq{v}^1, \seq{v}^2, \seq{v}^3} },
\end{align*}
where $\seq{v}' = (j)_{j \in \colM \setminus \bJ}.$ Then we fix the variable $x_{\seq{v}'}$ in $\R^{d_{\seq{v}'}}$ and consider the sum inside the integral. For now, in this proof, we do not write $x_{\seq{v}'}$ to the subscript of the functions, i.e. $b$ means $b_{x_{\seq{v}'}}$ and so on.
Thus we are estimating 
\begin{align*}
&\Babs{\sum_{I_{\seq{v}^1}, I_{\seq{v}^2}, I_{\seq{v}^3}} \Bpair{b, h_{I_{\seq{v}^1}} \otimes h_{I_{\seq{v}^2}} \otimes \avgfu{I_{\seq{v}^3}}} \Bpair{f, h_{I_{\seq{v}^1}} \otimes \avgfu{I_{\seq{v}^2}} \otimes h_{I_{\seq{v}^3}}}  \pair{g, h_{I_{\seq{v}^1}} h_{I_{\seq{v}^1}} \otimes h_{I_{\seq{v}^2}} \otimes h_{I_{\seq{v}^3}}}} \\
&= \Babs{\int_{\R^{d_{\seq{v}^3}}} \sum_{I_{\seq{v}^3}} \frac{1_{I_{\seq{v}^3}}(x_{\seq{v}^3})}{|I_{\seq{v}^3}|} \sum_{I_{\seq{v}^1}, I_{\seq{v}^2} } \pair{b_{x_{\seq{v}^3}}, h_{I_{\seq{v}^1}} \otimes h_{I_{\seq{v}^2}}} \Bpair{f, h_{I_{\seq{v}^1}} \otimes \avgfu{I_{\seq{v}^2}} \otimes h_{I_{\seq{v}^3}}} \\
&\hspace{12em}\times \pair{g, h_{I_{\seq{v}^1}} h_{I_{\seq{v}^1}} \otimes h_{I_{\seq{v}^2}} \otimes h_{I_{\seq{v}^3}}} \dd x_{\seq{v}^3}} \\
&= \Babs{\int_{\R^{d_{\seq{v}^3}}} \sum_{I_{\seq{v}^3}} \frac{1_{I_{\seq{v}^3}}(x_{\seq{v}^3})}{|I_{\seq{v}^3}|}  \pair{b_{x_{\seq{v}^3}}, A_{I_{\seq{v}^3}}^{\seq{v}^1, \seq{v}^2}(f,g)} \dd x_{\seq{v}^3}},
\end{align*}
where
$$
A_{I_{\seq{v}^3}}^{\seq{v}^1, \seq{v}^2}(f,g) := \sum_{I_{\seq{v}^1}, I_{\seq{v}^2}} \Bpair{f, h_{I_{\seq{v}^1}} \otimes \avgfu{I_{\seq{v}^2}} \otimes h_{I_{\seq{v}^3}}} \pair{g, h_{I_{\seq{v}^1}} h_{I_{\seq{v}^1}} \otimes h_{I_{\seq{v}^2}} \otimes h_{I_{\seq{v}^3}}} h_{I_{\seq{v}^1}} \otimes h_{I_{\seq{v}^2}}.
$$

By Proposition \ref{prop:uniformBMO} it is enough to show the boundedness of 
\begin{equation}\label{eq:paraprodproof}
\int_{\R^{\seq{v}^3}} \sum_{I_{\seq{v}^3}} \frac{1_{I_{\seq{v}^3}}(x_{\seq{v}^3})}{|I_{\seq{v}^3}|} \norm{ S^{\seq{v}^1, \seq{v}^2}( A^{\seq{v}^1, \seq{v}^2}_{I_{\seq{v}^3}}(f,g))}{L^{1}( \nu_{x_{\seq{v}^3}})} \dd x_{\seq{v}^3}.
\end{equation}
Note that $\nu_{x_{\seq{v}^3}}$ actually is $\nu_{x_{\seq{v}'},{x_{\seq{v}^3}}}.$

First, observe that 
\begin{align*}
&S^{\seq{v}^1, \seq{v}^2}( A_{I_{\seq{v}^3}}^{\seq{v}^1, \seq{v}^2}(f,g)) \\
&=  \Bsulk{\sum_{I_{\seq{v}^1}, I_{\seq{v}^2}}  \Babs{\Bpair{f, h_{I_{\seq{v}^1}} \otimes \avgfu{I_{\seq{v}^2}} \otimes h_{I_{\seq{v}^3}}}}^2 \Babs{\Bpair{g, h_{I_{\seq{v}^1}} h_{I_{\seq{v}^1}} \otimes h_{I_{\seq{v}^2}} \otimes h_{I_{\seq{v}^3}}}}^2 \avgfu{I_{\seq{v}^1}} \otimes \avgfu{I_{\seq{v}^2}}}^\pinv{2} \\
&\leq \Bsulk{\sum_{I_{\seq{v}^1}} \bbrac{M^{\seq{v}^2}\pair{f, h_{I_{\seq{v}^1}} \otimes h_{I_{\seq{v}^3}}}_{\seq{v}^1, \seq{v}^3}}^2 \otimes  \avgfu{I_{\seq{v}^1}} }^\pinv{2} \\
&\qquad \times \Bsulk{\sum_{I_{\seq{v}^2}}   \bbrac{M^{\seq{v}^1}\pair{g, h_{I_{\seq{v}^2}} \otimes h_{I_{\seq{v}^3}}}_{\seq{v}^2,\seq{v}^3}}^2  \otimes \avgfu{I_{\seq{v}^2}}}^\pinv{2}.
\end{align*}
Using the previous inequality we get
\begin{align*}
&\sum_{I_{\seq{v}^3}} \frac{1_{I_{\seq{v}^3}}(x_{\seq{v}^3})}{|I_{\seq{v}^3}|} \norm{ S^{\seq{v}^1, \seq{v}^2}( A_{I_{\seq{v}^3}}^{\seq{v}^1, \seq{v}^2}(f,g))}{L^{1}( \nu_{x_{\seq{v}^3}})} \\
&\leq \int_{\R^{d_{\seq{v}^1} + d_{\seq{v}^2}}}\sum_{I_{\seq{v}^3}} \frac{1_{I_{\seq{v}^3}}(x_{\seq{v}^3})}{|I_{\seq{v}^3}|}  \Bsulk{\sum_{I_{\seq{v}^1}} \bbrac{M^{\seq{v}^2}\pair{f, h_{I_{\seq{v}^1}} \otimes h_{I_{\seq{v}^3}}}_{\seq{v}^1, \seq{v}^3}}^2 \otimes  \avgfu{I_{\seq{v}^1}} }^\pinv{2} \\
&\hspace{6em} \times \Bsulk{\sum_{I_{\seq{v}^2}}   \bbrac{M^{\seq{v}^1}\pair{g, h_{I_{\seq{v}^2}} \otimes h_{I_{\seq{v}^3}}}_{\seq{v}^2,\seq{v}^3}}^2  \otimes \avgfu{I_{\seq{v}^2}}}^\pinv{2} \nu_{x_{\seq{v}^3}} \\
&\leq \int_{\R^{d_{\seq{v}^1} + d_{\seq{v}^2}}}  \Bsulk{\sum_{I_{\seq{v}^1}, I_{\seq{v}^3}} \bbrac{M^{\seq{v}^2}\pair{f, h_{I_{\seq{v}^1}} \otimes h_{I_{\seq{v}^3}}}_{\seq{v}^1, \seq{v}^3}}^2 \otimes  \avgfu{I_{\seq{v}^1}} \otimes  \frac{1_{I_{\seq{v}^3}}(x_{\seq{v}^3})}{|I_{\seq{v}^3}|} }^\pinv{2} \\
&\hspace{5em} \times \Bsulk{\sum_{I_{\seq{v}^2}, I_{\seq{v}^3}}   \bbrac{M^{\seq{v}^1}\pair{g, h_{I_{\seq{v}^2}} \otimes h_{I_{\seq{v}^3}}}_{\seq{v}^2,\seq{v}^3}}^2  \otimes \avgfu{I_{\seq{v}^2}} \otimes  \frac{1_{I_{\seq{v}^3}}(x_{\seq{v}^3})}{|I_{\seq{v}^3}|}}^\pinv{2} \nu_{x_{\seq{v}^3}} \\
&\leq \Bnorm{\Bsulk{\sum_{I_{\seq{v}^1}, I_{\seq{v}^3}} \bbrac{M^{\seq{v}^2}\pair{f, h_{I_{\seq{v}^1}} \otimes h_{I_{\seq{v}^3}}}_{\seq{v}^1, \seq{v}^3}}^2 \otimes  \avgfu{I_{\seq{v}^1}} \otimes  \frac{1_{I_{\seq{v}^3}}(x_{\seq{v}^3})}{|I_{\seq{v}^3}|} }^\pinv{2}}{L^{p}(\mu_{x_{\seq{v}^3}})} \\
&\qquad\times\Bnorm{\Bsulk{\sum_{I_{\seq{v}^2}, I_{\seq{v}^3}}   \bbrac{M^{\seq{v}^1}\pair{g, h_{I_{\seq{v}^2}} \otimes h_{I_{\seq{v}^3}}}_{\seq{v}^2,\seq{v}^3}}^2  \otimes \avgfu{I_{\seq{v}^2}} \otimes  \frac{1_{I_{\seq{v}^3}}(x_{\seq{v}^3})}{|I_{\seq{v}^3}|}}^\pinv{2}}{L^{p'}(\lambda^{1 - p'}_{x_{\seq{v}^3}})}.
\end{align*}
Putting this estimate back to \eqref{eq:paraprodproof} and applying Hölder's inequality we get
\begin{align*}
\eqref{eq:paraprodproof} &\leq \Bnorm{\Bsulk{\sum_{I_{\seq{v}^1}, I_{\seq{v}^3}} \bbrac{M^{\seq{v}^2}\pair{f, h_{I_{\seq{v}^1}} \otimes h_{I_{\seq{v}^3}}}_{\seq{v}^1, \seq{v}^3}}^2 \otimes  \avgfu{I_{\seq{v}^1}} \otimes  \avgfu{I_{\seq{v}^3}} }^\pinv{2}}{L^{p}(\mu)} \\
&\qquad\times\Bnorm{\Bsulk{\sum_{I_{\seq{v}^2}, I_{\seq{v}^3}}   \bbrac{M^{\seq{v}^1}\pair{g, h_{I_{\seq{v}^2}} \otimes h_{I_{\seq{v}^3}}}_{\seq{v}^2,\seq{v}^3}}^2  \otimes \avgfu{I_{\seq{v}^2}} \otimes  \avgfu{I_{\seq{v}^3}}}^\pinv{2}}{L^{p'}(\lambda^{1 - p'})} \\
&\lesssim \norm{f}{L^{p}(\mu)} \norm{g}{L^{p'}(\lambda^{1 - p'})},
\end{align*}
where in the last step we apply Lemma \ref{lem:FSgen}.

Lastly, recall that we fixed $x_{\seq{v}'} \in \R^{d_{\seq{v}'}}$ and the previous bound actually is
$$
 \norm{f_{x_{\seq{v}'}}}{L^{p}(\mu_{x_{\seq{v}'}})} \norm{g_{x_{\seq{v}'}}}{L^{p'}(\lambda_{{x_{\seq{v}'}}}^{1 - p'})}.
$$
However, by applying the Hölder's inequality once more to the integral over $\R^{d_{\seq{v}'}}$ we get
$$
\sabs{\pair{A^{\seq{v}}_{\seq{i}}(b,f), g}} \lesssim \normBMO{b}{\seq{u}}{\nu} \norm{f}{L^{p}(\mu)} \norm{g}{L^{p'}(\lambda^{1 - p'})},
$$
where $\seq{u} = (j)_{j \in \bI_1 \cup \bI_2}.$
\end{proof}
\section{Paraproduct free commutators}\label{sec:shift}
We assume that operators in this section are defined in some fixed dyadic grids $\bD^{\seq{d}} = \prod_{i = 1}^m \bD^{d_i}.$

 Let $\seq{v}$ be a subsequence of $\seqm.$ Define the multi-parameter shift 
$$
\bS^{\seq{v}} f = \bS^{\seq{v}, \seq{k}, \seq{l}}_{\bD^{d_{\seq{v}}}} f
= \sum_{K_\seq{v} \in \bD^{d_\seq{v}}} \sum_{ \substack{I_{\seq{v}} \in  \bD^{d_\seq{v}} \\I_{\seq{v}}^{(\sseq{k}{v})} = K_\seq{v}}} \sum_{\substack{J_\seq{v} \in  \bD^{d_\seq{v}} \\ J_{\seq{v}}^{(\sseq{l}{v})} = K_\seq{v}}} a_{K_\seq{v}, I_{\seq{v}}, J_\seq{v}} \pair{f, h_{I_{\seq{v}}}}_{\seq{v}}\otimes h_{J_{\seq{v}}} .
$$
Here $k_{v_i}, l_{v_i} \geq 0$  and only finitely many of 
the coefficients $a_{K_\seq{v}, I_{\seq{v}}, J_\seq{v}}$ are non-zero and 
$$
|a_{K_\seq{v}, I_{\seq{v}}, J_\seq{v}}| \leq\frac{|I_{\seq{v}}|^\pinv{2}|J_{\seq{v}}|^\pinv{2}}{|K_{\seq{v}}|}.
$$

First, we record here a standard equality as a lemma, since the notation in the multi-parameter setting needs some explaining.
\begin{lem}\label{lem:diffofavgs}
	Let $\varphi$ be a locally integrable function defined on $\R^{d_{\seq{v}}},$ where $\seq{v} = (v_i)_{i = 1}^n$ is a subsequence of $\seqm,$ and let $k_{v_i},l_{v_i}$ be non-negative integers for $i = 1,2,\dots n.$ Also let $I_\seq{v}, J_\seq{v}, K_\seq{v} \in \bD^{d_{\seq{v}}}$ such that $I_{v_i}^{(k_{v_i})} = K_{v_i} = J_{v_i}^{(l_{v_i})}$ for all $i = 1,2,\dots,n.$
	There holds
	\begin{equation}\label{eq:diffExpansion}
	\pair{\varphi}_{J_\seq{v}} - \pair{\varphi}_{I_{\seq{v}}} = \sum_{i = 1}^n \Bsulk{\sum_{t = 1}^{l_{v_i}} \pair{\Delta_{J_{v_i}^{(t)}}^{v_i} \varphi}_{J_{v_i} \times Q_{\seq{v}'_i}} - \sum_{s = 1}^{k_{v_i}} \pair{\Delta_{I_{v_i}^{(s)}}^{v_i} \varphi}_{I_{v_i} \times Q_{\seq{v}'_i}}},
	\end{equation}
	where $\seq{v}'_i = (v_j)_{j \in \{1,2,\dots, n \} \setminus i}, $  and
	 \begin{equation*}
	 Q_{\seq{v}'_i} = 
	 \begin{cases}
	 I_{v_1} \times \dots \times I_{v_{i-1}}  \times J_{v_{i+1}} \times \dots \times J_{v_n},  &\text{ if } 1 < i < n \\
	 I_{v_1} \times \dots \times I_{v_{n-1}},  &\text{ if }  i=  n \\
	 J_{v_{2}} \times \dots \times J_{v_n}, &\text{ if } i = 1.
	 \end{cases}
	 \end{equation*}
\end{lem}
\begin{proof}
The case $n=1$ follows easily from the telescoping nature of the sum.
The case $n=2$ follows from this as follows.
 For notational simplicity only, let $v_1 = 1$ and $v_2 = 2.$ Observe that
\begin{align*}
\pair{\varphi}_{J_{1} \times J_{2}} - \pair{\varphi}_{I_{1} \times I_{2}} &= \pair{\varphi}_{J_{1} \times J_{2}} - \pair{\varphi}_{I_{1} \times J_{2}} + \pair{\varphi}_{I_{1} \times J_{2}}  - \pair{\varphi}_{I_{1} \times I_{2}}.
\end{align*}
Since $K_{i}$ is some parent cube for both $I_{i}$ and $J_{i}$, we can use the one-parameter expansion  $\pair{\varphi}_{Q_i,i} - \pair{\varphi}_{K_i,i} = \sum_{q = 1}^{q_i} \pair{\Delta_{Q_i^{(q)}}^i \varphi}_{Q_i,i},$ where $Q_i^{(q_i)} = K_i.$ Thus, we have
\begin{align*}
\pair{\varphi}_{J_{1} \times J_{2}} - \pair{\varphi}_{I_{1} \times I_{2}} &= \sum_{t = 1}^{l_{1}} \pair{\Delta_{J_{1}^{(t)}}^{1} \varphi }_{J_{1} \times J_{2}} - \sum_{s = 1}^{k_1} \pair{\Delta_{I_1^{(s)}}^1 \varphi}_{I_1 \times J_2} \\
&\quad+ \sum_{t = 1}^{l_{2}} \pair{\Delta_{J_{2}^{(t)}}^{2} \varphi }_{I_{1} \times J_{2}} - \sum_{s = 1}^{k_2} \pair{\Delta_{I_2^{(s)}}^2 \varphi}_{I_1 \times I_2},
\end{align*}
as claimed.

We can continue as follows. If the claim holds for a fixed $n,$ we have
$$
	\pair{\varphi}_{J_\seq{v}} - \pair{\varphi}_{I_{\seq{v}}} = \sum_{i = 1}^{n+1} \Bsulk{\sum_{t = 1}^{l_{v_i}} \pair{\Delta_{J_{v_i}^{(t)}}^{v_i} \varphi}_{J_{v_i} \times Q_{\seq{v}'_i}} - \sum_{s = 1}^{k_{v_i}} \pair{\Delta_{I_{v_i}^{(s)}}^{v_i} \varphi}_{I_{v_i} \times Q_{\seq{v}'_i}}}
$$
for $\seq{v} = (v_i)_{i = 1}^{n+1}.$
Indeed, for notational simplicity let again $v_i = i$ for all $i = 1,2,\dots, n+1$, and notice that we may write
\begin{align*}
	\pair{\varphi}_{J_\seq{v}} - \pair{\varphi}_{I_{\seq{v}}} &= 	\pair{\varphi}_{J_1 \times J_{\seq{u}}} - \pair{\varphi}_{I_{1} \times J_{\seq{u}}} +  \pair{\varphi}_{I_{1} \times J_{\seq{u}}} - \pair{\varphi}_{I_1 \times I_{\seq{u}}} \\
	&=: A + B,
\end{align*} 
where $\seq{u} = (i)_{i=2}^{n+1}.$
For the term $A$ we use the one-parameter expansion and for the term $B$ we use the assumption that the claim holds for $n$ parameters. Hence, we get
\begin{align*}
\pair{\varphi}_{J_\seq{v}} - \pair{\varphi}_{I_{\seq{v}}} &= \sum_{t = 1}^{l_{1}} \pair{\Delta_{J_{1}^{(t)}}^{1} \varphi}_{J_{1} \times J_{\seq{u}}} - \sum_{s = 1}^{k_{1}} \pair{\Delta_{I_{1}^{(s)}}^{1} \varphi}_{I_{1} \times J_{\seq{u}}} \\
&\quad + \sum_{i = 2}^{n+1} \Bsulk{\sum_{t = 1}^{l_{i}} \pair{\Delta_{J_{i}^{(t)}}^{i} \varphi}_{I_1 \times J_i \times  Q_{\seq{u}'_i}} - \sum_{s = 1}^{k_{i}} \pair{\Delta_{I_{i}^{(s)}}^{i} \varphi}_{ I_1 \times I_i \times  Q_{\seq{u}'_i}}}\\
&= \sum_{i = 1}^{n+1} \Bsulk{\sum_{t = 1}^{l_{i}} \pair{\Delta_{J_{i}^{(t)}}^{i} \varphi}_{J_{i} \times Q_{\seq{v}'_i}} - \sum_{s = 1}^{k_{i}} \pair{\Delta_{I_{i}^{(s)}}^{i} \varphi}_{I_{i} \times Q_{\seq{v}'_i}}},
\end{align*}
where $Q_{\seq{u}'_i}$ and $Q_{\seq{v}'_i}$ are defined as in the statement.
 \end{proof}
 
 The main result of this section is to show the boundedness of the commutators with paraproduct free Calder\'on-Zygmund operators. By the representation theorem, it is enough to consider the commutators of dyadic shifts, Theorem \ref{thm:iteratedShift}. The strategy is to expand the commutator using martingale differences. This leaves us with terms that are compositions of shifts and paraproducts, legal or illegal ones. In the case of illegal paraproducts, we combine some terms together and apply Lemma \ref{lem:diffofavgs}. This leads to terms, which all fall under the general term \eqref{eq:generalterm}. Before we show in detail how to expand the commutators, we present the general term and show its boundedness.

Assume that $\bI  = \ksulku{\bI_i: i = 1,2,3,4,5}$ is a partition of $\ksulku{1,2,\dots, m}.$  We set $\seq{v}^i = (v^i_j)_{v_j^i \in \mathcal{I}_i}.$
The general term is defined as 
\begin{align}\label{eq:generalterm}
 &\sum_{\substack{K_{\seq{v}^i} \in \bD^{d_{\seq{v}^i}} \\ K_{\seq{v}^5} \in \bD^{d_{\seq{v}^5}}}}  \sum_{ \substack{I_{\seq{v}^i}, J_{\seq{v}^i} \in \bD^{d_{\seq{v}^i}} \\ I_{\seq{v}^i}^{\sulku{\sseqi{k}{v}{i}}}= K_{\seq{v}^i}\\J_{\seq{v}^i}^{\sulku{\sseqi{l}{v}{i}}}=K_{\seq{v}^i}}} \alpha_{K_{\seq{v}^i}, I_{\seq{v}^i}, J_{\seq{v}^i}} \beta_{ I_{\seq{v}^1}, J_{\seq{v}^3}}  \Bpair{b,  h_{I_{\seq{v}^1}^{(\seq{s})}} \otimes  \frac{1_{I_{\seq{v}^2}}}{|I_{\seq{v}^2}|}   \otimes h_{J_{\seq{v}^3}^{(\seq{t})}}\otimes  \frac{1_{J_{\seq{v}^4}}}{|J_{\seq{v}^4}|} \otimes h_{K_{\seq{v}^5}} }\\
&\hspace{4em}\times \Bpair{f,  h_{I_{\seq{v}^1}} \otimes h_{I_{\seq{v}^2}}  \otimes h_{I_{\seq{v}^3}}  \otimes h_{I_{\seq{v}^4}} \otimes \frac{1_{K_{\seq{v}^5}}}{|K_{\seq{v}^5}|}}   \wt{h}_{J_{\seq{v}^1}} \otimes h_{J_{\seq{v}^2}} \otimes \wt{h}_{J_{\seq{v}^3}}  \otimes h_{J_{\seq{v}^4}}\otimes h_{K_{\seq{v}^5}}, \nonumber
\end{align}
where $i =1,2,3,4,$  
$k_{v^i_j},l_{v^i_j} \geq 0$ for $j \in \bI_i,$ $s_j \leq k_{v^1_j}$ for $j \in \bI_1,$ $t_j \leq l_{v^3_j}$ for $j \in \bI_3,$ 
and 
$$ 
|\alpha_{K_{\seq{v}^i}, I_{\seq{v}^i}, J_{\seq{v}^i}}|\le \prod_{j = 1}^4 \frac{|I_{\seq{v}^{j}}|^\pinv{2}|J_{\seq{v}^{j}}|^\pinv{2}}{|K_{\seq{v}^{j}}|} ,
$$
$$
|\beta_{ I_{\seq{v}^1}, J_{\seq{v}^3}}|\le \babs{I_{\seq{v}^1}^{(\seq{s})} }^{-\pinv{2}}  \babs{J_{\seq{v}^3}^{( \seq{t} )}}^{-\pinv{2}}, $$
$$
  |\wt{h}_Q|\le |Q|^{-\pinv 2}1_Q .$$ 
Moreover, if, e.g.  $\mathcal{I}_j =\emptyset$, then the related terms are understood as $1$ and we require that $\bigcup_{j = 1,3,5}\mathcal{I}_j \neq \emptyset.$  

In addition, similar to the proof of Lemma \ref{lem:paraBound}, we omit the details if some $\bI_j =\emptyset.$

\begin{lem}\label{lem:generalEstimation}
There holds
$$
\norm{\varphi}{L^{p}(\lambda)} \lesssim_{[\nu]_{A_p}} \normBMO{b}{\seq{u}}{\nu}\norm{f}{L^p{(\mu)} },
$$
where  $\varphi$ is the term \eqref{eq:generalterm} defined above and $\seq{u} = (u_j), u_j \in \bigcup_{i = 1,3,5} \mathcal{I}_i.$ 

\end{lem}
\begin{proof}
We begin by using the size conditions of $\alpha$ and $\beta$ for the dual form. Hence, we have
  \begin{align*}
&\Babs{\sum_{\substack{K_{\seq{v}^i},K_{\seq{v}^5}}}  \sum_{ \substack{I_{\seq{v}^i}^{\sulku{\sseqi{k}{v}{i}}}= K_{\seq{v}^i}\\J_{\seq{v}^i}^{\sulku{\sseqi{l}{v}{i}}}=K_{\seq{v}^i}}} \alpha_{K_{\seq{v}^i}, I_{\seq{v}^i}, J_{\seq{v}^i}} \beta_{ I_{\seq{v}^1}, J_{\seq{v}^3}}  \Bpair{b,  h_{I_{\seq{v}^1}^{(\seq{s})} \times J_{\seq{v}^3}^{(\seq{t})}} \otimes  \frac{1_{I_{\seq{v}^2} \times J_{\seq{v}^4}}}{|I_{\seq{v}^2} \times J_{\seq{v}^4}|}  \otimes h_{K_{\seq{v}^5}} }\\
&\hspace{4em}\times \Bpair{f,  h_{I_{\seq{v}^1}\times I_{\seq{v}^2} \times I_{\seq{v}^3}  \times I_{\seq{v}^4}} \otimes \frac{1_{K_{\seq{v}^5}}}{|K_{\seq{v}^5}|}} \pair{g,   \wt{h}_{J_{\seq{v}^1}} \otimes \wt{h}_{J_{\seq{v}^3}} \otimes h_{J_{\seq{v}^2} \times J_{\seq{v}^4} \times K_{\seq{v}^5}}}}\\
 &\leq \sum_{\substack{K_{\seq{v}^i},K_{\seq{v}^5}}}   \sum_{ \substack{I_{\seq{v}^i}^{\sulku{\sseqi{k}{v}{i}}}= K_{\seq{v}^i}\\J_{\seq{v}^i}^{\sulku{\sseqi{l}{v}{i}}}=K_{\seq{v}^i}}} \prod_{j = 1}^4 \frac{|I_{\seq{v}^j}|^\pinv{2}|J_{\seq{v}^j}|^\pinv{2}}{|K_{\seq{v}^j}|} \babs{I_{ \seq{v}^1}^{(\seq{s})} }^{-\pinv{2}}  \babs{J_{\seq{v}^3}^{(\seq t)} }^{-\pinv{2}} \babs{I_{\seq{v}^2} \times J_{\seq{v}^4}}^{-1} \\
  &\hspace{3em}\int_{I_{\seq{v}^2} \times J_{\seq{v}^4}}\Bbrac{ \babs{\pair{b,  h_{I_{\seq{v}^1}^{(\seq{s})} \times J_{\seq{v}^3}^{(\seq{t})}}  \otimes h_{K_{\seq{v}^5}} }_{\seq{v}^1, \seq{v}^3, \seq{v}^5}}\Babs{\Bpair{f,  h_{I_{\seq{v}^1}\times I_{\seq{v}^2} \times I_{\seq{v}^3}  \times I_{\seq{v}^4}} \otimes \frac{1_{K_{\seq{v}^5}}}{|K_{\seq{v}^5}|}}} \\
 &\hspace{18em}\babs{\pair{g,   \wt{h}_{J_{\seq{v}^1}} \otimes \wt{h}_{J_{\seq{v}^3}} \otimes h_{J_{\seq{v}^2} \times J_{\seq{v}^4} \times K_{\seq{v}^5}}}}} \\
 &= \sum_{K_{\seq{v}^2},K_{\seq{v}^4}} \sum_{ \substack{I_{\seq{v}^2}^{\sulku{\sseqi{k}{v}{2}}}= K_{\seq{v}^2}}}\sum_{ \substack{J_{\seq{v}^4}^{(\sseqi{l}{v}{4})}=K_{\seq{v}^4}}} \int_{I_{\seq{v}^2} \times J_{\seq{v}^4}} A^b(x_{\seq{v}^2, \seq{v}^4}) \dd {x_{\seq{v}^2, \seq{v}^4}},
  \end{align*}
where 
\begin{align*}
A^b(x_{\seq{v}^2, \seq{v}^4}) &:=  \sum_{K_{\seq{v}^1}, K_{\seq{v}^3}, K_{\seq{v}^5}}  \sum_{ \substack{P_{\seq{v}^1}^{(\sseqi{k}{v}{1} - \seq{s})}= K_{\seq{v}^1} \\ Q_{\seq{v}^3}^{(\sseqi{l}{v}{3} - \seq{t})} = K_{\seq{v}^3} }}  \sum_{\substack{I_{\seq{v}^1}^{(\seq{s})} = P_{\seq{v}^1} \\ J_{\seq{v}^3}^{(\seq{t})} = Q_{\seq{v}^3}}}\sum_{\substack{I_{\seq{v}^3}^{(k_{\seq{v}^3})} = K_{\seq{v}^3} \\ J_{\seq{v}^1}^{(l_{\seq{v}^1})} = K_{\seq{v}^1}}} \sum_{\substack{I_{\seq{v}^4}^{(k_{\seq{v}^4})} = K_{\seq{v}^4} \\ J_{\seq{v}^2}^{(l_{\seq{v}^2})} = K_{\seq{v}^2}}}\\
 &\times \prod_{j = 1}^4 \frac{|I_{\seq{v}^j}|^\pinv{2}|J_{\seq{v}^j}|^\pinv{2}}{|K_{\seq{v}^j}|} \babs{P_{ \seq{v}^1} }^{-\pinv{2}}  \babs{Q_{\seq{v}^3}}^{-\pinv{2}} \babs{I_{\seq{v}^2} \times J_{\seq{v}^4}}^{-1}\babs{\pair{b_{x_{\seq{v}^2,\seq{v}^4}},  h_{P_{\seq{v}^1}\times Q_{\seq{v}^3}}  \otimes h_{K_{\seq{v}^5}} }}\\
 &\times \Babs{\Bpair{ f,  h_{I_{\seq{v}^1}\times I_{\seq{v}^2} \times I_{\seq{v}^3}  \times I_{\seq{v}^4}} \otimes \frac{1_{K_{\seq{v}^5}}}{|K_{\seq{v}^5}|}}}\babs{\pair{g,   \wt{h}_{J_{\seq{v}^1}} \otimes \wt{h}_{J_{\seq{v}^3}} \otimes h_{J_{\seq{v}^2} \times J_{\seq{v}^4} \times K_{\seq{v}^5}}}}.
\end{align*}

Then we proceed by replacing $f$ and $g$ with suitable multi-parameter martingale blocks, i.e.
\begin{align*}
&A^b(x_{\seq{v}^2, \seq{v}^4}) \\
&=  \sum_{K_{\seq{v}^1}, K_{\seq{v}^3}, K_{\seq{v}^5}}  \sum_{ \substack{P_{\seq{v}^1}^{(\sseqi{k}{v}{1} - \seq{s})}= K_{\seq{v}^1} \\ Q_{\seq{v}^3}^{(\sseqi{l}{v}{3} - \seq{t})} = K_{\seq{v}^3} }}  \sum_{\substack{I_{\seq{v}^1}^{(\seq{s})} = P_{\seq{v}^1} \\ J_{\seq{v}^3}^{(\seq{t})} = Q_{\seq{v}^3}}}\sum_{\substack{I_{\seq{v}^3}^{(k_{\seq{v}^3})} = K_{\seq{v}^3} \\ J_{\seq{v}^1}^{(l_{\seq{v}^1})} = K_{\seq{v}^1}}} \sum_{\substack{I_{\seq{v}^4}^{(k_{\seq{v}^4})} = K_{\seq{v}^4} \\ J_{\seq{v}^2}^{(l_{\seq{v}^2})} = K_{\seq{v}^2}}}\\
 &\qquad\times \prod_{j = 1}^4 \frac{|I_{\seq{v}^j}|^\pinv{2}|J_{\seq{v}^j}|^\pinv{2}}{|K_{\seq{v}^j}|} \babs{P_{ \seq{v}^1} }^{-\pinv{2}}  \babs{Q_{\seq{v}^3}}^{-\pinv{2}} \babs{I_{\seq{v}^2} \times J_{\seq{v}^4}}^{-1}\babs{\pair{b_{x_{\seq{v}^2,\seq{v}^4}},  h_{P_{\seq{v}^1}\times Q_{\seq{v}^3}}  \otimes h_{K_{\seq{v}^5}} }}\\
 &\qquad\times \Babs{\Bpair{\Delta_{K_{\seq{v}^1}\times K_{\seq{v}^2}\times K_{\seq{v}^3}  \times K_{\seq{v}^4}, (\sseqi{k}{v}{j})_{j =1,2,3,4}}^{{\seq{v}^1,  \seq{v}^2, \seq{v}^3, \seq{v}^4}} f,  h_{I_{\seq{v}^1}\times I_{\seq{v}^2} \times I_{\seq{v}^3}  \times I_{\seq{v}^4}} \otimes \frac{1_{K_{\seq{v}^5}}}{|K_{\seq{v}^5}|}}} \\
 &\qquad\times\babs{\pair{\Delta_{  K_{\seq{v}^2}\times K_{\seq{v}^4} \times K_{\seq{v}^5}, (\sseqi{l}{v}{2}, \sseqi{l}{v}{4}, \seq{0})}^{{ \seq{v}^2, \seq{v}^4}, \seq{v}^5}g,   \wt{h}_{J_{\seq{v}^1}} \otimes \wt{h}_{J_{\seq{v}^3}} \otimes h_{J_{\seq{v}^2} \times J_{\seq{v}^4} \times K_{\seq{v}^5}}}}\\
 &\leq  \sum_{K_{\seq{v}^1}, K_{\seq{v}^3}, K_{\seq{v}^5}}  \sum_{ \substack{P_{\seq{v}^1}^{(\sseqi{k}{v}{1} - \seq{s})}= K_{\seq{v}^1}}} \sum_{ \substack{Q_{\seq{v}^3}^{(\sseqi{l}{v}{3} - \seq{t})}=K_{\seq{v}^3}}} |P_{\seq{v}^1}|^{\pinv{2}} |Q_{\seq{v}^3}|^{\pinv{2}} |K_{\seq{v}^5}|^\pinv{2} \\
&\qquad\times \babs{\pair{b_{x_{\seq{v}^2, \seq{v}^4}}, h_{P_{\seq{v}^1} \times Q_{\seq{v}^3}} \otimes  h_{K_{\seq{v}^5}} }}\\ 
&\qquad\times \pair{|\Delta_{K_{\seq{v}^1}\times K_{\seq{v}^2}\times K_{\seq{v}^3}  \times K_{\seq{v}^4}, (\sseqi{k}{v}{j})_{j =1,2,3,4}}^{{\seq{v}^1,  \seq{v}^2, \seq{v}^3, \seq{v}^4}} f|}_{P_{\seq{v}^1}    \times I_{\seq{v}^2}\times K_{\seq{v}^3} \times K_{\seq{v}^4}\times K_{\seq{v}^5}} 
 \\ &\qquad\times \pair{|\Delta_{  K_{\seq{v}^2}\times K_{\seq{v}^4} \times K_{\seq{v}^5}, (\sseqi{l}{v}{2}, \sseqi{l}{v}{4}, \seq{0})}^{{ \seq{v}^2, \seq{v}^4}, \seq{v}^5} g|}_{ K_{\seq{v}^1} \times K_{\seq{v}^2} \times Q_{\seq{v}^3}  \times J_{\seq{v}^4}\times K_{\seq{v}^5}},
\end{align*}
where we summed up rectangles of levels $\seq{s},\seq{t},k_{\seq{v}^3},l_{\seq{v}^1},k_{\seq{v}^4},l_{\seq{v}^2} $   after modulus is taken inside of the pairings of martingale blocks of $f$ and $g.$

Now, using Proposition \ref{prop:uniformBMO} for $A^b$ with fixed $x_{\seq{v}^2, \seq{v}^4}$ we get
\begin{align*}
&A^b(x_{\seq{v}^2, \seq{v}^4}) \lesssim \normBMO{b}{\seq{u}}{\nu} \int_{\R^{d_{\seq{v}^1} +d_{\seq{v}^3} + d_{\seq{v}^5} }} \\
 &\hspace{2em}\Bsulk{ \sum_{K_{\seq{v}^1}, K_{\seq{v}^3}}\bbrac{M^{\seq{v}^1, \seq{v}^3, \seq{v}^5} \pair{|\Delta_{K_{\seq{v}^1} \times K_{\seq{v}^2} \times K_{\seq{v}^3} \times K_{\seq{v}^4}, (\sseqi{k}{v}{j})_{j =1,2,3,4}}^{{\seq{v}^1,  \seq{v}^2,\seq{v}^3, \seq{v}^4}} f|}_{ I_{\seq{v}^2} \times K_{\seq{v}^4}, \seq{v}^2, \seq{v}^4}}^2 1_{K_{\seq{v}^1}  K_{\seq{v}^3}} }^\pinv{2}
 \\  &\hspace{2em} \Bsulk{\sum_{K_{\seq{v}^5}}   \bbrac{M^{\seq{v}^1, \seq{v}^3, \seq{v}^5} \pair{|\Delta_{  K_{\seq{v}^2}\times K_{\seq{v}^4} \times K_{\seq{v}^5}, (\sseqi{l}{v}{2}, \sseqi{l}{v}{4}, \seq{0})}^{{ \seq{v}^2, \seq{v}^4}, \seq{v}^5} g|}_{ K_{\seq{v}^2}  \times J_{\seq{v}^4}, \seq{v}^2, \seq{v}^4}}^2   1_{ K_{\seq{v}^5}}}^\pinv{2}\nu_{x_{\seq{v}^2, \seq{v}^4}},
\end{align*}
where again we summed over the rectangles $P_{\seq{v}^1}, Q_{\seq{v}^3}$ and $\seq{u} = (u_j), u_j\in \bigcup_{i = 1,3,5}\mathcal{I}_i.$
Notice that here we needed the requirement that  $\bigcup_{i = 1,3,5}\mathcal{I}_i \neq \emptyset.$

Hence, we can conclude that
\begin{align*}
&\sabs{\pair{\varphi,g}} \lesssim \normBMO{b}{\seq{u}}{\nu} \int \sum_{K_{\seq{v}^2},K_{\seq{v}^4}} \sum_{ \substack{I_{\seq{v}^2}^{\sulku{\sseqi{k}{v}{2}}}= K_{\seq{v}^2}}}\sum_{ \substack{J_{\seq{v}^4}^{(\sseqi{l}{v}{4})}=K_{\seq{v}^4}}}  \\ 
&\times \Bsulk{\sum_{K_{\seq{v}^1}, K_{\seq{v}^3}} \bbrac{M^{\seq{v}^1, \seq{v}^3, \seq{v}^5} \pair{|\Delta_{K_{\seq{v}^1} \times K_{\seq{v}^2} \times K_{\seq{v}^3}\times K_{\seq{v}^4}, (\sseqi{k}{v}{j})_{j =1,2,3,4}}^{{\seq{v}^1,  \seq{v}^2,\seq{v}^3, \seq{v}^4}} f|}_{ I_{\seq{v}^2} \times K_{\seq{v}^4}, \seq{v}^2, \seq{v}^4} }^2   1_{K_{\seq{v}^1} \times K_{\seq{v}^3}}}^\pinv{2}
 \\  &\times \Bsulk{\sum_{K_{\seq{v}^5}} \bbrac{M^{\seq{v}^1, \seq{v}^3, \seq{v}^5} \pair{|\Delta_{  K_{\seq{v}^2}\times K_{\seq{v}^4} \times K_{\seq{v}^5}, (\sseqi{l}{v}{2}, \sseqi{l}{v}{4}, \seq{0})}^{{ \seq{v}^2, \seq{v}^4}, \seq{v}^5} g|}_{ K_{\seq{v}^2}  \times J_{\seq{v}^4}, \seq{v}^2, \seq{v}^4}}^2  1_{K_{\seq{v}^5}}  }^\pinv{2}   1_{I_{\seq{v}^2} \times J_{\seq{v}^4}}  \nu \\
 &=: \normBMO{b}{\seq{u}}{\nu} \times  \text{I}.
\end{align*}
Using standard estimates we get
\begin{align*}
\text{I} &\lesssim \int  \Bsulk{\sum_{\substack{ K_{\seq{v}^1} K_{\seq{v}^2} \\ K_{\seq{v}^3}, K_{\seq{v}^4}}} \bbrac{M^{\seqm} \Delta_{K_{\seq{v}^1} \times K_{\seq{v}^2} \times K_{\seq{v}^3}\times K_{\seq{v}^4}, (\sseqi{k}{v}{j})_{j =1,2,3,4}}^{{\seq{v}^1,  \seq{v}^2,\seq{v}^3, \seq{v}^4}} f }^2  1_{K_{\seq{v}^1}  \times K_{\seq{v}^2}\times K_{\seq{v}^3} \times K_{\seq{v}^4}}}^\pinv{2}
 \\  &\qquad\times \Bsulk{\sum_{K_{\seq{v}^2}, K_{\seq{v}^4}, K_{\seq{v}^5}} \bbrac{M^{\seqm}\Delta_{  K_{\seq{v}^2}\times K_{\seq{v}^4} \times K_{\seq{v}^5}, (\sseqi{l}{v}{2}, \sseqi{l}{v}{4}, \seq{0})}^{{ \seq{v}^2, \seq{v}^4}, \seq{v}^5} g}^2   1_{K_{\seq{v}^2} \times K_{\seq{v}^4} \times K_{\seq{v}^5}} }^\pinv{2}   \nu \\
 &\lesssim \norm{f}{L^p(\mu)} \norm{g}{L^{p'}(\lambda^{1- p'})}.
\end{align*}

\end{proof}

For simplicity, we begin with the case of two iterations.
\begin{thm}
Let $\bI = \ksulku{\bI_1, \bI_2}$ be some partition of $\colM,$ such that $\bI_1 \neq \emptyset \neq \bI_2.$
Let $\nu = \mu^\pinv{p} \lambda^{-\pinv{p}},$ where $\mu, \lambda \in A_p(\R^{d_1} \times \dots \times \R^{d_m}),$
and $\seq{u} = (u_i)_{i \in \bI_1}$ and $\seq{v} = (v_i)_{i \in \bI_2}.$ It holds
$$
\norm{[\bS^{\seq{u}}, [b,\bS^{\seq{v}}]]}{L^p(\mu) \rightarrow L^{p}(\lambda)}
\lesssim \normbmo{b}{\bI}{\nu}
$$
for $1<p<\infty.$
\end{thm}
Before the proof, we make a small remark. 
We can begin with commutator $[b, \bS^\seqm].$ However, in this case, $b$ is in the little BMO space. In \cite{Li2018Bloom} this is  proved for bi-parameter operators, i.e. the case $m = 2.$ The method used there can be applied to the multi-parameter case. Hence, the result of \cite{Li2018Bloom} regarding the first order shift case extends to the multi-parameter framework.  We omit the details.
\begin{proof}
We say that the number of parameters in $\bI_1$ is $n.$
We begin by expanding appearing products in all of the parameters. Hence, we have
\begin{align*}
[\bS^{\seq{u}}[b, \bS^{\seq{v}}]]f &=\bS^{\seq{u}}(b \bS^{\seq{v}}f) - \bS^{\seq{u}}  \bS^{\seq{v}}(bf) \\
&\quad-  b \bS^{\seq{u}}  \bS^{\seq{v}}f + \bS^{\seq{v}} (b \bS^{\seq{u}}  f) \\
 &= \sum_{\substack{i_{\seq{u}} \in \ksulku{1,2,3}^{n}\\ i_{\seq{v}} \in \ksulku{1,2,3}^{m -n}}} \bS^{\seq{u}}(A_{i_{\seq{u}}}^{\seq{u}} A_{i_{\seq{v}}}^{\seq{v}} (b, \bS^{\seq{v}}f))  -\bS^{\seq{u}}  \bS^{\seq{v}}(A_{i_{\seq{u}}}^{\seq{u}} A_{i_{\seq{v}}}^{\seq{v}} (b,f)) \\ 
&\hspace{4em} - A_{i_{\seq{u}}}^{\seq{u}} A_{i_{\seq{v}}}^{\seq{v}}(b, \bS^{\seq{u}}  \bS^{\seq{v}}f) + \bS^{\seq{v}} (A_{i_{\seq{u}}}^{\seq{u}} A_{i_{\seq{v}}}^{\seq{v}}(b, \bS^{\seq{u}}  f)).
\end{align*}

Now, if $i_{\seq{u}} \neq \ksulku{3}^n$ and $i_{\seq{v}} \neq \ksulku{3}^{m-n},$ then each individual term of is bounded by combining boundedness of the multi-parameter shifts with Lemma \ref{lem:paraBound}.
Hence, it is enough to consider terms in the following sums 
\begin{equation}\label{eq:shiftIllegals}
\sum_{\substack{i_{\seq{u}} = \ksulku{3}^{n}\\ i_{\seq{v}} \in \ksulku{1,2,3}^{m -n} \setminus \ksulku{3}^{m-n}}} + \sum_{\substack{i_{\seq{u}} \in \ksulku{1,2,3}^{n} \setminus \ksulku{3}^{n} \\ i_{\seq{v}} = \ksulku{3}^{m-n}}} + \sum_{\substack{i_{\seq{u}} = \ksulku{3}^{n}\\ i_{\seq{v}} = \ksulku{3}^{m-n}}}.
\end{equation}

The terms in the first two sums are similar. Hence, considering the first sum, we are essentially handling the second one simultaneously and we choose to deal with the first one.

Fix $i_{\seq{u}} = \ksulku{3}^n,$  $i_{\seq{v}} \in \ksulku{1,2,3}^{m -n} \setminus \ksulku{3}^{m-n}.$  $\, $ We pair  $\bS^{\seq{u}}(A_{i_{\seq{u}}}^{\seq{u}} A_{i_{\seq{v}}}^{\seq{v}} (b, \bS^{\seq{v}}f))$ with
$A_{i_{\seq{u}}}^{\seq{u}} A_{i_{\seq{v}}}^{\seq{v}}(b, \bS^{\seq{u}}  \bS^{\seq{v}}f)$ and   $\bS^{\seq{v}} (A_{i_{\seq{u}}}^{\seq{u}} A_{i_{\seq{v}}}^{\seq{v}}(b, \bS^{\seq{u}}  f))$ with $\bS^{\seq{u}} \bS^{\seq{v}}(A_{i_{\seq{u}}}^{\seq{u}} A_{i_{\seq{v}}}^{\seq{v}} (b,f)).$ It is enough to study $$A_{i_{\seq{u}}}^{\seq{u}} A_{i_{\seq{v}}}^{\seq{v}} (b, \bS^{\seq{u}}f) -  \bS^{\seq{u}}(A_{i_{\seq{u}}}^{\seq{u}} A_{i_{\seq{v}}}^{\seq{v}} (b,f)),$$ since $\bS^{\seq{u}} \bS^{\seq{v}} = \bS^{\seq{v}} \bS^{\seq{u}} $ and $\bS^{\seq{v}}$ is bounded.

We remark that when considering the second sum in \eqref{eq:shiftIllegals}, the terms need to be paired in the other order and then $\bS^{\seq{u}}$ can be left out by similar argument. Generally, we pair the terms so that we get rid of the shifts on the parameters where the paraproduct operator is legal one.

Let us recall the definition of paraproduct operator $A_{i_{\seq{v}}}^\seq{v}.$ 
Let  $\bJ_s = \ksulku{j \leq m-n: i_{v_j}  = s},$ $s = 1,2,3$ and 
 $\seq{v}^1 = (v_j)_{j \in \bJ_1},$  $\seq{v}^2  = (v_j)_{j \in \bJ_2}$ and  $\seq{v}^3 = (v_j)_{j \in \bJ_3}.$ Thus
 \begin{align*}
 A_{i_{\seq{v}}}^\seq{v}(b,f) = \sum_{K_{\seq{v}^{1}},K_{\seq{v}^{2}}, K_{\seq{v}^3}} \Bpair{b, h_{K_{\seq{v}^{1,2}}} \otimes \frac{1_{K_{\seq{v}^3}}}{|K_{\seq{v}^3}|}}_{\seq{v}^1,\seq{v}^2,\seq{v}^3}\Bpair{f,{h}_{K_{\seq{v}^1}} \otimes \avgfu{K_{\seq{v}^2}} \otimes h_{K_{\seq{v}^3}}}_{\seq{v}^1,\seq{v}^2,\seq{v}^3}\\
  \otimes h_{K_{\seq{v}^1}}h_{K_{\seq{v}^1}}  \otimes h_{K_{\seq{v}^2}} \otimes h_{K_{\seq{v}^3}},
 \end{align*}
where $K_{\seq{v}^{1,2}}$ denotes the rectangle $K_{\seq{v}^1}\times K_{\seq{v}^2}.$ Also recall the definition of multi-parameter shift 
\begin{align*}
& \bS^{\seq{u}}f =  \sum_{K_{\seq{u}}}\sum_{I_{\seq{u}}^{(\sseq{k}{u})} = K_{\seq{u}}} \sum_{J_{\seq{u}}^{(\sseq{l}{u})} = K_{\seq{u}}} a_{K_{\seq{u}}, I_{\seq{u}}, J_{\seq{u}}} \pair{f, h_{I_\seq{u}}}_{\seq{u}} \otimes h_{J_{\seq{u}}}.
\end{align*}
Hence, we have
\begin{align}\label{eq:oneparamAddition}
&A_{i_{\seq{u}}}^{\seq{u}} A_{i_{\seq{v}}}^{\seq{v}} (b, \bS^{\seq{u}}f) -  \bS^{\seq{u}}(A_{i_{\seq{u}}}^{\seq{u}} A_{i_{\seq{v}}}^{\seq{v}} (b,f))\nonumber \\
&=\sum_{\substack{K_{\seq{v}^{1}},K_{\seq{v}^{2}}, K_{\seq{v}^3} \\
 K_\seq{u}}} \sum_{I_{\seq{u}}^{(\sseq{k}{u})} = K_{\seq{u}}} \sum_{J_{\seq{u}}^{(\sseq{l}{u})} = K_{\seq{u}}} a_{K_{\seq{u}}, I_{\seq{u}}, J_{\seq{u}}}\nonumber  \\
 &\qquad\times\bigbrac{\Bpair{b, h_{K_{\seq{v}^{1,2}}} \otimes \frac{1_{K_{\seq{v}^3}}}{|K_{\seq{v}^3}|} \otimes \frac{1_{J_{\seq{u}}}}{|J_{\seq{u}}|}} - \Bpair{b, h_{K_{\seq{v}^{1,2}}} \otimes \frac{1_{K_{\seq{v}^3}}}{|K_{\seq{v}^3}|} \otimes \frac{1_{I_{\seq{u}}}}{|I_{\seq{u}}|}}} \\
 &\qquad\times \Bpair{f,{h}_{K_{\seq{v}^1}} \otimes \avgfu{K_{\seq{v}^2}} \otimes h_{K_{\seq{v}^3}} \otimes h_{I_{\seq{u}}}}  {h}_{K_{\seq{v}^1}}{h}_{K_{\seq{v}^1}}  \otimes h_{K_{\seq{v}^2}} \otimes h_{K_{\seq{v}^3}} \otimes h_{J_{\seq{u}}}, \nonumber
\end{align}
where  $\seq{v}^1,$  $\seq{v}^2,$  $\seq{v}^3,$ and $K_{\seq{v}^{1,2}}$ is defined as above.

By Lemma \ref{lem:diffofavgs}, we can write \eqref{eq:oneparamAddition} as
\begin{align*}
&\sum_{\substack{K_{\seq{v}^{1}},K_{\seq{v}^{2}}, K_{\seq{v}^3} \\
 K_\seq{u}}} \sum_{I_{\seq{u}}^{(\sseq{k}{u})} = K_{\seq{u}}} \sum_{J_{\seq{u}}^{(\sseq{l}{u})} = K_{\seq{u}}} a_{K_{\seq{u}}, I_{\seq{u}}, J_{\seq{u}}}\\
  &\qquad\times\Bbrac{\sum_{j = 1}^n \Bsulk{ \sum_{t = 1}^{l_{u_j}} \pair{\Delta_{J_{u_j}^{(t)}}^{u_j} b_{K_{\seq{v}}}}_{J_{u_j} \times Q_{\seq{u}'_j}} 
 - \sum_{s = 1}^{k_{u_j}} \pair{\Delta_{I_{u_j}^{(s)}}^{u_j} b_{K_{\seq{v}}}}_{I_{u_j} \times Q_{\seq{u}'_j}}}}\\
 &\qquad\times \Bpair{f,{h}_{K_{\seq{v}^1}} \otimes \avgfu{K_{\seq{v}^2}} \otimes h_{K_{\seq{v}^3}} \otimes h_{I_{\seq{u}}}}  {h}_{K_{\seq{v}^1}}{h}_{K_{\seq{v}^1}}  \otimes h_{K_{\seq{v}^2}} \otimes h_{K_{\seq{v}^3}} \otimes h_{J_{\seq{u}}},
\end{align*}
where $b_{K_{\seq{v}}} = \pair{b, h_{K_{\seq{v}^{1,2}}} \otimes 1_{K_{\seq{v}^3}}/|K_{\seq{v}^3}|}_{\seq{v}^{1}, \seq{v}^2, \seq{v}^3},$ and	 $Q_{\seq{u}'_j} = I_{(u_i)_{1 \leq i < j}} \times J_{(u_i)_{j < i \leq n }}.$ 

 Now these terms are expanded to a desired form, i.e.  there is a cancellative Haar function on some parameter in $\bI_1$ and $\bI_2$ paired with the function $b.$  For example,  the first term in the above pair with $j = 2$ equals to 
\begin{align*}\label{eq:shiftOneParamExpanded}
&\sum_{t = 1}^{l_{u_2}}\sum_{\substack{K_{\seq{v}^{1}},K_{\seq{v}^{2}}, K_{\seq{v}^3} \nonumber\\
 K_\seq{u}}} \sum_{I_{\seq{u}}^{(\sseq{k}{u})} = K_{\seq{u}}} \sum_{J_{\seq{u}}^{(\sseq{l}{u})} = K_{\seq{u}}} a_{K_{\seq{u}}, I_{\seq{u}}, J_{\seq{u}}} \pair{h_{J_{u_2}^{(t)}}}_{J_{u_2}} \\
 &\hspace{5em}\times\Bpair{b, h_{K_{\seq{v}^{1,2}}} \otimes \frac{1_{K_{\seq{v}^3}}}{|K_{\seq{v}^3}|} \otimes \avgfu{I_{u_1}}\otimes h_{J_{u_2}^{(t)}} \otimes  \avgfu{J_{\seq{u}'}} } \\
 &\hspace{5em}\times \Bpair{f, {h}_{K_{\seq{v}^1}} \otimes \avgfu{K_{\seq{v}^2}} \otimes h_{K_{\seq{v}^3}} \otimes h_{I_{\seq{u}}}} {h}_{K_{\seq{v}^1}}{h}_{K_{\seq{v}^1}}  \otimes h_{K_{\seq{v}^2}} \otimes h_{K_{\seq{v}^3}} \otimes h_{J_{\seq{u}}}, \nonumber
\end{align*}
where $\seq{u}' = (u_j)_{j=3}^{n}$ and $|\pair{h_{J_{u_2}^{(t)}}}_{J_{u_2}}| = |J_{u_2}^{(t)}|^{-\pinv{2}}.$ Hence, these terms are bounded by Lemma \ref{lem:generalEstimation}.

Let us also expand the last term of \eqref{eq:shiftIllegals}. Here we can not do any reductions and we sum everything together. Hence, the term equals to
\begin{equation}\label{eq:shiftFinalTerm}
\begin{split}
&\sum_{K_{\seq{u}}, K_\seq{v}}  \sum_{\substack{I_{\seq{u}}^{(\sseq{k}{u})} = K_{\seq{u}} \\ J_{\seq{u}}^{(\sseq{l}{u})} = K_{\seq{u}}}} \sum_{\substack{I_{\seq{v}}^{(\sseq{k}{v})} = K_{\seq{v}} \\ J_{\seq{v}}^{(\sseq{l}{v})} = K_{\seq{v}}}} a_{K_{\seq{v}}, I_{\seq{v}}, J_{\seq{v}}}
a_{K_{\seq{u}}, I_{\seq{u}}, J_{\seq{u}}} \bbrac{\pair{b}_{I_\seq{u} \times J_\seq{v}} - \pair{b}_{I_\seq{u} \times I_\seq{v}} - \pair{b}_{J_\seq{u} \times J_\seq{v}} + \pair{b}_{J_\seq{u} \times I_\seq{v}}} \\
 &\hspace{7em}\times \pair{f, h_{I_{\seq{u}}} \otimes h_{I_{\seq{v}}}} h_{J_{\seq{u}}} \otimes h_{J_{\seq{v}}}.
\end{split}
\end{equation}
Here we proceed similarly, as with the previous terms, but now expanding on the both parameters sets  $\bI_1$ and $\bI_2.$ 

First, we apply Lemma \ref{lem:diffofavgs} in the parameters $\bI_1.$ Hence, we have 
\begin{align}\label{eq:SLfirstsplit}
&\pair{b}_{I_\seq{u} \times J_\seq{v}} - \pair{b}_{I_\seq{u} \times I_\seq{v}} - \pair{b}_{J_\seq{u} \times J_\seq{v}} + \pair{b}_{J_\seq{u} \times I_\seq{v}} \nonumber \\
&=\sulku{\pair{b}_{J_\seq{u} \times I_\seq{v}} - \pair{b}_{I_\seq{u} \times I_\seq{v}}} - \sulku{ \pair{b}_{J_\seq{u} \times J_\seq{v}} - \pair{b}_{I_\seq{u} \times J_\seq{v}} } \nonumber \\
&=\sum_{j_1 = 1}^n\Bsulk{{ \sum_{t_1 = 1}^{l_{u_{j_1}}} \pair{\Delta_{J_{u_{j_1}}^{(t_1)}}^{u_{j_1}} b_{I_{\seq{v}}}}_{J_{u_{j_1}} \times Q_{\seq{u}'_{j_1}}} 
 - \sum_{s_1 = 1}^{k_{u_{j_1}}} \pair{\Delta_{I_{u_{j_1}}^{(s_1)}}^{u_{j_1}} b_{I_{\seq{v}}} }_{I_{u_{j_1}} \times Q_{\seq{u}'_{j_1}}}}} \\
 &\quad - \sum_{{j_1} = 1}^n\Bsulk{ \sum_{t_1 = 1}^{l_{u_{j_1}}} \pair{\Delta_{J_{u_{j_1}}^{(t_1)}}^{u_{j_1}} b_{J_{\seq{v}}}}_{J_{u_{j_1}} \times Q_{\seq{u}'_{j_1}}} 
 - \sum_{s_1 = 1}^{k_{u_{j_1}}} \pair{\Delta_{I_{u_{j_1}}^{(s_1)}}^{u_{j_1}} b_{J_{\seq{v}}}}_{I_{u_{j_1}} \times Q_{\seq{u}'_{j_1}}}}, \nonumber
\end{align}
where $b_{I_\seq{v}} = \pair{b}_{I_\seq{v}, \seq{v}},$ $b_{J_\seq{v}} = \pair{b}_{J_\seq{v}, \seq{v}}$ and $Q_{\seq{u}_{j_1}'} = I_{(u_j)_{1 \leq j < j_1}} \times J_{(u_j)_{j_1 <  j \leq n}}.$
Then we pair these terms in the other order and apply Lemma \ref{lem:diffofavgs} in the parameters $\bI_2.$ Then, for example, the pair of the first and third term on the right-hand side of \eqref{eq:SLfirstsplit} equals to 
\begin{align}
&\sum_{j_1 = 1}^n\Bsulk{ \sum_{t_1 = 1}^{l_{u_{j_1}}} \pair{\Delta_{J_{u_{j_1}}^{(t_1)}}^{u_{j_1}} b_{I_{\seq{v}}}}_{J_{u_{j_1}} \times Q_{\seq{u}'_{j_1}}} 
 - \pair{\Delta_{J_{u_{j_1}}^{(t_1)}}^{u_{j_1}} b_{J_{\seq{v}}}}_{J_{u_{j_1}} \times Q_{\seq{u}'_{j_1}}} } \nonumber \\
 &= \sum_{{j_1} = 1}^n \sum_{t_1 = 1}^{l_{u_{j_1}}} \pair{\Delta_{J_{u_{j_1}}^{(t_1)}}^{u_{j_1}} (\pair{b}_{I_{\seq{v}}, \seq{v}} - \pair{b}_{J_{\seq{v}}, \seq{v}} )}_{J_{u_{j_1}} \times Q_{\seq{u}'_{j_1}}} \nonumber \\
 & =  \sum_{j_1 = 1}^n \sum_{t_{1} = 1}^{l_{u_{j_1}}} \sum_{j_2 = 1}^{m - n} \Bsulk{\sum_{s_2 = 1}^{k_{v_{j_2}}} \pair{\Delta_{J_{u_{j_1}}^{(t_{1})} \times I_{v_{j_2}}^{(s_2)}}^{u_{j_1}, v_{j_2}} b}_{J_{u_{j_1}} \times I_{v_{j_2}} \times  Q_{\seq{u}'_{j_1}} \times Q_{\seq{v}'_{j_2}}} \label{eq:SLE}\\
 &\hspace{8em}-\sum_{t_{2} = 1}^{l_{v_{j_2}}} \pair{\Delta_{J_{u_{j_1}}^{(t_{1})} \times J_{v_{j_2}}^{(t_{2})}}^{u_{j_1}, v_{j_2}} b}_{J_{u_{j_1}} \times J_{v_{j_2}}  \times  Q_{\seq{u}'_{j_1}} \times Q_{\seq{v}'_{j_2}}} }, \nonumber
\end{align}
where $Q_{\seq{v}'_{j_2}} = I_{(v_j)_{1 \leq j < j_2}} \times J_{(v_j)_{j_2 <  j \leq m - n}}.$

Thus, for example,  the first term on the right-hand side of  \eqref{eq:SLE} with $j_1 = 2$ and $j_2 = 2$ related to \eqref{eq:shiftFinalTerm} equals to
\begin{align}\label{eq:shiftLastTerm}
\sum_{t_1 = 1}^{l_{u_2}} \sum_{s_2 = 1}^{k_{v_2}} \sum_{K_{\seq{u}}, K_\seq{v}}  \sum_{\substack{I_{\seq{u}}^{(\sseq{k}{u})} = K_{\seq{u}} \\ J_{\seq{u}}^{(\sseq{l}{u})} = K_{\seq{u}}}} \sum_{\substack{I_{\seq{v}}^{(\sseq{k}{v})} = K_{\seq{v}} \\ J_{\seq{v}}^{(\sseq{l}{v})} = K_{\seq{v}}}} &a_{K_{\seq{v}}, I_{\seq{v}}, J_{\seq{v}}} a_{K_{\seq{u}}, I_{\seq{u}}, J_{\seq{u}}}  \pair{h_{J_{u_{2}}^{(t_1)}}}_{J_{u_{2}}} \pair{h_{I_{v_{2}}^{(s_2)}}}_{I_{v_{2}}} \nonumber \\
&\times\Bpair{b, \avgfu{I_{u_1}} \otimes h_{J_{u_{2}}^{(t_1)}} \otimes \avgfu{J_{\seq{u}'}} \otimes\avgfu{I_{v_1}} \otimes h_{I_{v_{2}}^{(s_2)}} \otimes \avgfu{J_{\seq{v}'} } } \\
&\times \pair{f, h_{I_{\seq{u}}} \otimes h_{I_{\seq{v}}}} h_{J_{\seq{u}}} \otimes h_{J_{\seq{v}}},\nonumber
\end{align}
where $\seq{u}' = (u_j)_{j = 3}^n$ and $\seq{v}' = (v_j)_{j = 3}^{m-n}.$
Hence, the boundedness of these terms follows from Lemma \ref{lem:generalEstimation}.
\end{proof}

\begin{thm}\label{thm:iteratedShift}
Let $k \leq m$ be an integer and let $\nu = \mu^\pinv{p} \lambda^{-\pinv{p}},$ where $\mu, \lambda \in A_p(\R^{d_1} \times \dots \times \R^{d_m}),$ and $1<p<\infty.$
For all partitions  $\bI = \ksulku{\bI_i: i\leq k}$  of $\colM$ we have
$$
\norm{[\bS^{\seq{v}^1},[\bS^{\seq{v}^2},\dots [b,\bS^{\seq{v}^k}]]]}{L^p(\mu) \rightarrow L^{p}(\lambda)}
\lesssim \normbmo{b}{\bI}{\nu},
$$
where $\seq{v}^i = (v^i_j)_{v^i_j\in \bI_i}.$
\end{thm}

\begin{proof}
We consider here the case $k=3.$ The aim is to show that the strategy and techniques used in the case $k = 2$ work here also. The general case follows similarly. 

Let $n_1 + n_2 + n_3 = m$ and say the number of parameters in $\seq{v}^1$ is $n_1$ and so on.
By definition 
\begin{align*}
&[\bS^{\seq{v}^1},[\bS^{\seq{v}^2},[b,\bS^{\seq{v}^3}]]]f \\
&= \bS^{\seq{v}^1}\bS^{\seq{v}^2} (b \bS^{\seq{v}^3} f) - \bS^{\seq{v}^1}\bS^{\seq{v}^2} \bS^{\seq{v}^3} (bf)- \bS^{\seq{v}^1} (b\bS^{\seq{v}^3}\bS^{\seq{v}^2}f) + \bS^{\seq{v}^1} \bS^{\seq{v}^3} (b \bS^{\seq{v}^2} f) \\
&\qquad- \bS^{\seq{v}^2}(b\bS^{\seq{v}^3}\bS^{\seq{v}^1}f) + \bS^{\seq{v}^2}\bS^{\seq{v}^3}(b\bS^{\seq{v}^1}f) + b\bS^{\seq{v}^3}\bS^{\seq{v}^2}\bS^{\seq{v}^1}f -\bS^{\seq{v}^3}(b\bS^{\seq{v}^2}\bS^{\seq{v}^1}f)
\end{align*}
As in the case $k=2,$ we expand in all of the parameters
$$
[\bS^{\seq{v}^1},[\bS^{\seq{v}^2},[b,\bS^{\seq{v}^3}]]]f = \sum_{\substack{\seq{i}_1 \in \ksulku{1,2,3}^{n_1} \\ \seq{i}_2 \in \ksulku{1,2,3}^{n_2} \\ \seq{i}_3 \in \ksulku{1,2,3}^{n_3}}} \sum_{j = 1}^8 \sigma_{j, \seq{i}_1, \seq{i}_2, \seq{i}_3},
$$
where e.g.
$$
\sigma_{1, \seq{i}_1, \seq{i}_2, \seq{i}_3} = \bS^{\seq{v}^1}\bS^{\seq{v}^2} (A^{\seq{v}^1}_{\seq{i}_1} A^{\seq{v}^2}_{\seq{i}_2} A^{\seq{v}^3}_{\seq{i}_3} (b, \bS^{\seq{v}^3} f)).
$$
  Again, Lemma \ref{lem:paraBound} combined with boundedness of the shifts yields that each individual term of  $$\sum_{\substack{\seq{i}_1 \in \ksulku{1,2,3}^{n_1} \setminus \ksulku{3}^{n_1} \\ \seq{i}_2 \in \ksulku{1,2,3}^{n_2} \setminus \ksulku{3}^{n_2} \\ \seq{i}_3 \in \ksulku{1,2,3}^{n_3}  \setminus \ksulku{3}^{n_3}}} \sum_{j = 1}^8 \sigma_{j, \seq{i}_1, \seq{i}_2, \seq{i}_3}$$
is directly bounded. Hence, we need to consider terms in the sum 
\begin{align}\label{eq:iterSums}
&\sum_{\substack{\seq{i}_1 =  \ksulku{3}^{n_1} \\ \seq{i}_2 \in \ksulku{1,2,3}^{n_2} \setminus \ksulku{3}^{n_2} \\ \seq{i}_3 \in \ksulku{1,2,3}^{n_3}  \setminus \ksulku{3}^{n_3}}} + \sum_{\substack{\seq{i}_1 \in \ksulku{1,2,3}^{n_1} \setminus \ksulku{3}^{n_1} \\ \seq{i}_2 = \ksulku{3}^{n_2} \\ \seq{i}_3 \in \ksulku{1,2,3}^{n_3}  \setminus \ksulku{3}^{n_3}}} + \sum_{\substack{\seq{i}_1 \in \ksulku{1,2,3}^{n_1} \setminus \ksulku{3}^{n_1} \\ \seq{i}_2 \in \ksulku{1,2,3}^{n_2} \setminus \ksulku{3}^{n_2} \\ \seq{i}_3 =  \ksulku{3}^{n_3}}} \\
&+\sum_{\substack{\seq{i}_1 =  \ksulku{3}^{n_1} \\ \seq{i}_2 = \ksulku{3}^{n_2} \\ \seq{i}_3 \in \ksulku{1,2,3}^{n_3}  \setminus \ksulku{3}^{n_3}}}  + \sum_{\substack{\seq{i}_1 = \ksulku{3}^{n_1} \\ \seq{i}_2 \in \ksulku{1,2,3}^{n_2} \setminus \ksulku{3}^{n_2} \\ \seq{i}_3 =  \ksulku{3}^{n_3}}} + \sum_{\substack{\seq{i}_1 \in \ksulku{1,2,3}^{n_1} \setminus \ksulku{3}^{n_1} \\ \seq{i}_2 = \ksulku{3}^{n_2} \\ \seq{i}_3 = \ksulku{3}^{n_3}}} +  \sum_{\substack{\seq{i}_1 = \ksulku{3}^{n_1} \\ \seq{i}_2 = \ksulku{3}^{n_2} \\ \seq{i}_3 =\ksulku{3}^{n_3}}}.\nonumber
\end{align}
In the first three sums we can reduce to cases of one shift by pairing two terms.  
For example,  let  $\seq{i}_1 =  \ksulku{3}^{n_1},$ $\seq{i}_2 \in \ksulku{1,2,3}^{n_2} \setminus \ksulku{3}^{n_2},$ and $\seq{i}_3 \in  \ksulku{1,2,3}^{n_3}  \setminus \ksulku{3}^{n_3}$, then the first pair of terms is 
$$
\bS^{\seq{v}^1}\bS^{\seq{v}^2} (A^{\seq{v}^1}_{\seq{i}_1} A^{\seq{v}^2}_{\seq{i}_2} A^{\seq{v}^3}_{\seq{i}_3}(b, \bS^{\seq{v}^3} f))- \bS^{\seq{v}^2}(A^{\seq{v}^1}_{\seq{i}_1} A^{\seq{v}^2}_{\seq{i}_2} A^{\seq{v}^3}_{\seq{i}_3}(b,\bS^{\seq{v}^3}\bS^{\seq{v}^1}f)).
$$
Here it is enough to study
$$
\bS^{\seq{v}^1} (A^{\seq{v}^1}_{\seq{i}_1} A^{\seq{v}^2}_{\seq{i}_2} A^{\seq{v}^3}_{\seq{i}_3}(b,  f))- (A^{\seq{v}^1}_{\seq{i}_1} A^{\seq{v}^2}_{\seq{i}_2} A^{\seq{v}^3}_{\seq{i}_3}(b,\bS^{\seq{v}^1}f))
$$
since $\bS^{\seq{v}^2}$ and $\bS^{\seq{v}^3}$ are bounded and order of the shifts is interchangeable.

Then the following three sums reduces to cases of two shifts by summing four terms. For example, let $\seq{i}_1 =  \ksulku{3}^{n_1},$ $\seq{i}_2 = \ksulku{3}^{n_2},$ and $\seq{i}_3 \in  \ksulku{1,2,3}^{n_3}  \setminus \ksulku{3}^{n_3}$, then the first sum of four terms is
\begin{align*}
\bS^{\seq{v}^1}\bS^{\seq{v}^2} (A^{\seq{v}^1}_{\seq{i}_1} A^{\seq{v}^2}_{\seq{i}_2} A^{\seq{v}^3}_{\seq{i}_3}(b, \bS^{\seq{v}^3} f)) - \bS^{\seq{v}^1} (A^{\seq{v}^1}_{\seq{i}_1} A^{\seq{v}^2}_{\seq{i}_2} A^{\seq{v}^3}_{\seq{i}_3}(b, \bS^{\seq{v}^3}\bS^{\seq{v}^2}f)) \\
- \bS^{\seq{v}^2}(A^{\seq{v}^1}_{\seq{i}_1} A^{\seq{v}^2}_{\seq{i}_2} A^{\seq{v}^3}_{\seq{i}_3}(b,\bS^{\seq{v}^3}\bS^{\seq{v}^1}f)) + A^{\seq{v}^1}_{\seq{i}_1} A^{\seq{v}^2}_{\seq{i}_2} A^{\seq{v}^3}_{\seq{i}_3}(b,\bS^{\seq{v}^3}\bS^{\seq{v}^2}\bS^{\seq{v}^1}f).
\end{align*}
By similar arguments as previously, it is enough to consider 
\begin{align*}
\bS^{\seq{v}^1}\bS^{\seq{v}^2} (A^{\seq{v}^1}_{\seq{i}_1} A^{\seq{v}^2}_{\seq{i}_2} A^{\seq{v}^3}_{\seq{i}_3}(b , f)) - \bS^{\seq{v}^1} (A^{\seq{v}^1}_{\seq{i}_1} A^{\seq{v}^2}_{\seq{i}_2} A^{\seq{v}^3}_{\seq{i}_3}(b,\bS^{\seq{v}^2}f)) \\
- \bS^{\seq{v}^2}(A^{\seq{v}^1}_{\seq{i}_1} A^{\seq{v}^2}_{\seq{i}_2} A^{\seq{v}^3}_{\seq{i}_3}(b,\bS^{\seq{v}^1}f)) + A^{\seq{v}^1}_{\seq{i}_1} A^{\seq{v}^2}_{\seq{i}_2} A^{\seq{v}^3}_{\seq{i}_3}(b, \bS^{\seq{v}^2}\bS^{\seq{v}^1}f).
\end{align*}

Notice that these types of terms are similar to terms of the case $k=2.$ In the latter example,  there is an additional legal paraproducts in the parameters  $\bI_{3}$ compared to the last term in the previous proof. However, we have already taken this account in the general term \eqref{eq:generalterm} and the boundedness follows by similar expansion as in the case $k=2.$

In the last term in \eqref{eq:iterSums} we need to expand
\begin{align}\label{eq:SIdiffavg}
 &\pair{b}_{I_{\seq{v}^1} \times I_{\seq{v}^2} \times J_{\seq{v}^3}} - \pair{b}_{I_{\seq{v}^1} \times I_{\seq{v}^2} \times I_{\seq{v}^3}} - \pair{b}_{I_{\seq{v}^1} \times J_{\seq{v}^2} \times J_{\seq{v}^3}} + \pair{b}_{I_{\seq{v}^1} \times J_{\seq{v}^2} \times I_{\seq{v}^3}} \nonumber\\
& - \pair{b}_{J_{\seq{v}^1} \times I_{\seq{v}^2} \times J_{\seq{v}^3}} +\pair{b}_{J_{\seq{v}^1} \times I_{\seq{v}^2} \times I_{\seq{v}^3}}  + \pair{b}_{J_{\seq{v}^1} \times J_{\seq{v}^2} \times J_{\seq{v}^3}} -  \pair{b}_{J_{\seq{v}^1} \times J_{\seq{v}^2} \times I_{\seq{v}^3}}
\end{align}
Now, we apply Lemma \ref{lem:diffofavgs} three times. First, we apply the lemma in parameters $\bI_1$ -- e.g. the sum of the first and the fifth terms  equals to
$$
\sum_{j_1 = 1}^{n_1} \Bsulk{\sum_{s_1 = 1}^{k_{v_{j_1}^1}}\pair{\Delta_{I_{v^1_{j_1}}^{(s_1)}} b}_{I_{v^1_{j_1}} \times Q_{\seq{v}^{ 1\prime}_{j_1}} \times I_{\seq{v}^2} \times J_{\seq{v}^3}} -\sum_{t_1 = 1}^{l_{v_{j_1}^1}} \pair{\Delta_{J_{v^1_{j_1}}^{(t_1)}} b}_{J_{v^1_{j_1}} \times Q_{\seq{v}^{ 1\prime}_{j_1}} \times I_{\seq{v}^2} \times J_{\seq{v}^3}} },
$$ where $Q_{\seq{v}^{ 1\prime}_{j_1}} = I_{(v^1_u)_{1\leq u < j_1}} \times J_{(v^1_u)_{j_1< u \leq n_1}}.$ 
Then we switch pairs such that we can  apply the lemma in parameters $\bI_2.$  For example, the sum of related terms of the first and the third terms in \eqref{eq:SIdiffavg} equals to
\begin{align*}
\sum_{j_1  =1}^{n_1}\sum_{j_2  =1}^{n_2}\sum_{s_1 = 1}^{k_{v_{j_1}^1}}\Bsulk{ &\sum_{s_2 = 1}^{k_{v_{j_2}^2}}
\pair{\Delta_{I_{v^1_{j_1}}^{(s_1)}} \Delta_{I_{v^2_{j_2}}^{(s_2)}} b}_{I_{v^1_{j_1}}\times I_{v^2_{j_2}} \times Q_{\seq{v}^{ 1\prime}_{j_1}} \times Q_{\seq{v}^{2\prime }_{j_2}} \times J_{\seq{v}^3}} \\
&- \sum_{t_2 = 1}^{l_{v_{j_2}^2}}\pair{\Delta_{I_{v^1_{j_1}}^{(s_1)}} \Delta_{J_{v^2_{j_2}}^{(t_2)}} b}_{I_{v^1_{j_1}}\times J_{v^2_{j_2}} \times Q_{\seq{v}^{ 1\prime}_{j_1}} \times Q_{\seq{v}^{ 2\prime}_{j_2}} \times J_{\seq{v}^3}}}.
\end{align*}

Finally, we pair terms such that we apply Lemma \ref{lem:diffofavgs} in parameters $\bI_3$. Hence, for example, the sum of the related terms of the first and the second terms in \eqref{eq:SIdiffavg} equals to
\begin{align}\label{eq:SILT}
\sum_{j_1  =1}^{n_1}\sum_{j_2  =1}^{n_2}\sum_{j_3  =1}^{n_3}\sum_{s_1 = 1}^{k_{v_{j_1}^1}}\sum_{s_2 = 1}^{k_{v_{j_2}^2}}\Bsulk{&\sum_{t_3 = 1}^{l_{v_{j_3}^3}}\pair{\Delta_{I_{v^1_{j_1}}^{(s_1)}} \Delta_{I_{v^2_{j_2}}^{(s_2)}}\Delta_{J_{v^3_{j_3}}^{(t_3)}} b}_{I_{v^1_{j_1}}\times I_{v^2_{j_2}}\times J_{v^3_{j_3}}  \times Q_{\seq{v}^{1\prime }_{j_1}} \times Q_{\seq{v}^{ 2\prime}_{j_2}} \times Q_{\seq{v}^{3\prime }_{j_3}}} \\
-&\sum_{s_3 = 1}^{k_{v_{j_3}^3}} \pair{\Delta_{I_{v^1_{j_1}}^{(s_1)}}  \Delta_{I_{v^2_{j_2}}^{(s_2)}} \Delta_{I_{v^3_{j_3}}^{(s_3)}} b}_{I_{v^1_{j_1}}\times I_{v^2_{j_2}}\times I_{v^3_{j_3}} \times Q_{\seq{v}^{1 \prime}_{j_1}} \times Q_{\seq{v}^{ 2 \prime}_{j_2}} \times Q_{\seq{v}^{ 3\prime}_{j_3}}}}. \nonumber
\end{align}
Now, each appearing term is fully expanded, for example, for fixed  $s_1 , s_2, t_3$ and  $j_1 = 1, j_2= 1, j_3 =1$ the term to be estimated related to the first term in \eqref{eq:SILT} equals to
\begin{align*}
&\sum_{K_{\seq{v}^1},K_{\seq{v}^2},K_{\seq{v}^3}}  \sum_{\substack{I_{\seq{v}^1}^{(k_{\seq{v}^1})} = K_{\seq{v}^1} \\ J_{\seq{v}^1}^{(l_{\seq{v}^1})} = K_{\seq{v}^1}}}\sum_{\substack{I_{\seq{v}^2}^{(k_{\seq{v}^2})} = K_{\seq{v}^2} \\ J_{\seq{v}^2}^{(l_{\seq{v}^2})} = K_{\seq{v}^2}}}\sum_{\substack{I_{\seq{v}^3}^{(k_{\seq{v}^3})} = K_{\seq{v}^3} \\ J_{\seq{v}^3}^{(l_{\seq{v}^3})} = K_{\seq{v}^3}}}  a_{K_{\seq{v}^1}, I_{\seq{v}^1}, J_{\seq{v}^1}} a_{K_{\seq{v}^2}, I_{\seq{v}^2}, J_{\seq{v}^2}}a_{K_{\seq{v}^3}, I_{\seq{v}^3}, J_{\seq{v}^3}}\\
 &\qquad\times\pair{h_{I_{v_{1}^1}^{(s_{1})}}}_{I_{v_{1}^1}} \pair{h_{I_{v_{1}^2}^{(s_{2})}}}_{I_{v_{1}^2}} \pair{
 h_{J_{v_{1}^3}^{(t_{3})}}}_{J_{v_{1}^3}} \Bpair{b, h_{I_{v_{1}^1}^{(s_{1})}  \times I_{v_{1}^2}^{(s_2)} \times J_{v^{3}_{1}}^{(t_3)}} \otimes \avgfu{J_{\seq{v}_1^{1\prime }} \times J_{\seq{v}_1^{2\prime }} \times J_{\seq{v}_1^{ 3 \prime}} }} \\
 &\qquad\times\pair{f, h_{I_{\seq{v}^1} \times I_{\seq{v}^2} \times I_{\seq{v}^3}} }h_{J_{\seq{v}^1} \times J_{\seq{v}^2} \times J_{\seq{v}^3}},
\end{align*}
where ${\seq{v}_1^{j \prime}} = (v_u^j)_{u = 2}^{n_j}$ for $j =1,2,3.$ 
It is easy see that these terms have the form of the general term. Hence, by Lemma \ref{lem:generalEstimation} these terms are bounded with $b \in \bmo^{\bI}$ condition.
\end{proof}

 By the representation theorem of the multi-parameter singular integrals \cite{Ou2017} we get the following result:
\begin{cor}
Let $k \leq m$ be an integer and let $\nu = \mu^\pinv{p} \lambda^{-\pinv{p}},$ where $\mu, \lambda \in A_p(\R^{d_1} \times \dots \times \R^{d_m})$ and  $1<p<\infty.$
For all partitions $\bI = \ksulku{\bI_i: i\leq k}$ of $\colM$ we have
$$
\norm{[T^{\seq{v}^1},[T^{\seq{v}^2},\dots [b,T^{\seq{v}^k}]]]}{L^p(\mu) \rightarrow L^{p}(\lambda)}
\lesssim \normbmo{b}{\bI}{\nu},
$$
where $\seq{v}^i = (v^i_j)_{v^i_j\in \bI_i}$ and $T^{\seq{v}^i}$s are paraproduct free multi-parameter Calderón\hyp{}Zygmund operators.
\end{cor}
\section{Commutators involving paraproducts}\label{sec:journe}

In this section, we consider the space $\R^d = \R^{d_1 + d_2 +d_3 + d_4}$ and  operators, which are defined in some  fixed grids $\bD^{d_i}.$

Next, we define the other two bi-parameter dyadic model operators: partial and full paraproducts.

\subsubsection*{Partial paraproduct} Let $k_1, l_1 \geq 0.$
We define
$$
 P^{1,2} g = P^{1,2, (k_1, l_1)}_{\bD^{d_1}, \bD^{d_2}} g =\sum_{\substack{K_1 \in \bD^{d_1} \\ K_2\in \bD^{d_2}}}\sum_{\substack{I_1,J_1 \in \bD^{d_1} \\
 I_1^{(k_1)} = K_1\\ J_1^{(l_1)}  = K_1}} \pair{a_{K_1, I_1,J_1}, h_{K_2}}   h_{J_1} \otimes \avgfu{K_2}\otimes\pair{ g, h_{I_1}\otimes h_{K_2}}_{1,2}.
$$
Here only finitely many of 
the coefficients $\pair{a_{K_1, I_1,J_1}, h_{K_2}}$ are non-zero, and 
$$
\norm{a_{K_1, I_1,J_1} }{\BMO(\bD^{d_2})} \leq \frac{|I_1|^{\pinv{2}}|J_1|^{\pinv{2}}}{|K_1|}.
$$
Also we have partial paraproduct of the form
$$
 P^{1,2} g = P^{1,2, (k_2,l_2)}_{\bD^{d_1}, \bD^{d_2}} g =
\sum_{\substack{K_1 \in \bD^{d_1} \\ K_2\in \bD^{d_2}}} \sum_{\substack{I_2, J_2 \in \bD^{d_2}\\I_2^{(k_2)} = K_2\\ J_2^{(l_2)}  = K_2}} \pair{a_{K_2, I_2,J_2}, h_{K_1}}  \avgfu{K_1} \otimes h_{J_2}\otimes \pair{ g, h_{K_1}\otimes h_{I_2}}_{1,2},
$$
where $k_2, l_2 \geq 0$ and 
$$
\norm{a_{K_2, I_2,J_2} }{\BMO(\bD^{d_1})} \leq \frac{|I_2|^{\pinv{2}}|J_2|^{\pinv{2}}}{|K_2|}.
$$

\subsubsection*{Full paraproduct}
We define
\begin{align*}
\Pi^{1,2} f = \Pi^{1,2}_{\bD^{d_1}, \bD^{d_2}} f = \sum_{\substack {K_1 \in \bD^{d_1} \\K_2 \in \bD^{d_2}}} \pair{a, h_{K_1 \times  K_2}}   \avgfu{K_1}\otimes \avgfu{K_2} \otimes\pair{ f, h_{K_1}\otimes h_{K_2}}_{1,2}.
\end{align*}
Here only finitely many of 
the coefficients $\pair{a, h_{K_1 \times  K_2}}$ are non-zero, and 
$$
\norm{a}{\BMO(\bD^{d_1} \times \bD^{d_2})} \leq 1.
$$
Also $\Pi^*, \Pi^{1*}, \Pi^{2*}$ are full paraproducts, where  e.g
 \begin{align*}
\Pi^{1*} f = \sum_{\substack {K_1 \in \bD^{d_1} \\K_2 \in \bD^{d_2}}} \pair{a, h_{K_1 \times  K_2}}   h_{K_1}\otimes \avgfu{K_2}\otimes \Bpair{ f,\avgfu{K_1} \otimes h_{K_2}}_{1,2}
\end{align*}
is the partial adjoint in the first parameter of above $\Pi^{1,2}$.
Later on, we abbreviate $ \pair{a, h_{K_1 \times {K_2}}}$ by $a_{K_1, K_2}.$

\begin{thm}
Let $\nu = \mu^\pinv{p} \lambda^{-\pinv{p}},$ where $\mu, \lambda \in A_p(\R^{d_1} \times \dots \times \R^{d_4})$ and  $1<p<\infty.$ There holds
$$
\norm{[T^{1,2}, [b,T^{3,4}]]}{L^p(\mu) \rightarrow L^{p}(\lambda)}
\lesssim \normbmo{b}{\ksulku{\ksulku{1,2} \ksulku{3,4}}}{\nu},
$$
where $T^{1,2}$ and $T^{3,4}$ are 
bi-parameter Calderón-Zygmund operators in $\R^{d_1 + d_2}$ and $\R^{d_3 + d_4},$ respectively.
\end{thm}

\begin{proof}

By the representation theorem \cite{Martikainen2012},  we are considering the following collection of commutators:
\begingroup
\setlength{\tabcolsep}{10pt} 
\renewcommand{\arraystretch}{1.5} 
\begin{center}
    \begin{tabular}{| l | l | l |}
    \hline
    $[\bS^{1,2}, [b, \bS^{3,4}]]$ & $[\bS^{1,2}, [b, P^{3,4}]]$ & $[\bS^{1,2}, [b, \Pi^{3,4}]]$  \\ \hline
    $[P^{1,2}, [b, \bS^{3,4}]]$ & $[P^{1,2}, [b, P^{3,4}]]$ & $[P^{1,2}, [b, \Pi^{3,4}]]$ \\ \hline
    $[\Pi^{1,2}, [b, \bS^{3,4}]]$ & $[\Pi^{1,2}, [b, P^{3,4}]]$ & $[\Pi^{1,2}, [b, \Pi^{3,4}]]$\\ \hline
    \end{tabular}
\end{center}
\endgroup 

By definition, for all model operators we have
\begin{align*}
[U^{1,2}, [b, V^{3,4}]]f
&= U^{1,2}bV^{3,4}f-U^{1,2}V^{3,4}bf-bV^{3,4}U^{1,2} f+ V^{3,4}bU^{1,2}f\\
&=:I-II-III+IV.
\end{align*}
Now, the forms of $U$ and $V$ determines how we expand the terms. We expand the products in the parameters, where a cancellative Haar function is paired with  $b.$ In the parameters where $b$ is paired with a non-cancellative Haar function we do not expand at all.

As explained earlier, by Lemma \ref{lem:paraBound} the terms, where $b$ is paired with the cancellative Haar functions on parameters 1 or 2, and 3 or 4, are directly bounded with the correct $\BMO$ condition. For the other terms, we need to pair terms depending on the expansion.

We only demonstrate the general strategy with a case involving both full and partial paraproducts.

We are considering model operators of the following form
\begin{align*}
\Pi^{1,2} f = \sum_{K_1,K_2} \pair{a ,h_{K_1 \times {K_2}}}   \frac{1_{K_1}}{|K_1|}\otimes h_{K_2} \otimes\Bpair{ f, h_{K_1}\otimes \avgfu{K_2}}_{1,2}, \\
 P^{3,4} g =\sum_{K_3,K_4}\sum_{\substack{I_3^{(k_3)} = K_3\\ J_3^{(l_3)}  = K_3}} \pair{a_{K_3, I_3,J_3}, h_{K_4}}  \pair{ g, h_{I_3}\otimes h_{K_4}}_{3,4}\otimes h_{J_3} \otimes \frac{1_{K_4}}{|K_4|}.
\end{align*}
Here only finitely many of 
the coefficients $\pair{a,h_{K_1 \times {K_2}}}$ and $\pair{a_{K_3, I_3,J_3}, h_{K_4}}$ are non-zero, and these coefficients have the following bounds
$$
\norm{a}{\BMO(\bD^{d_1} \times \bD^{d_2})} \leq 1,
$$
$$
\norm{a_{K_3, I_3,J_3} }{\BMO(\bD^{d_4})} \leq \frac{|I_3|^{\pinv{2}}|J_3|^{\pinv{2}}}{|K_3|}.
$$

As explained earlier, we expand the appearing terms in the following way:
\begin{align*}
I &= \sum_{i_1, i_3=1}^3 \Pi^{1,2} A_{i_1}^1 A_{i_3}^3 (b, P^{3,4}f),  &II &= \sum_{i_1,i_3,i_4 = 1}^3 \Pi^{1,2} P^{3,4} A_{i_1}^1 A_{i_3,i_4}^{3,4} (b, f), \\
III &= \sum_{i_2, i_3 = 1}^3 A_{i_2}^{2} A_{i_3}^3 (b, \Pi^{1,2} P^{3,4} f), &IV &=  \sum_{i_2, i_3, i_4 = 1}^3 P^{3,4} A_{i_2}^{2} A_{i_3,i_4}^{3,4}(b, \Pi^{1,2} f).
\end{align*} 

Boundedness of the model operators combined with Lemma \ref{lem:paraBound} implies that each term is directly bounded whenever we do not have the ``illegal'' paraproducts in the parameters 1 or 2 and 3 or 4.
 For the rest of the terms, we group as follows 
\begin{align}\label{eq:Parapairing}
&\sum_{i_3=1}^2 \bbrac{\Pi^{1,2} A_{3}^1 A_{i_3}^3 (b, P^{3,4}f) - A_{3}^{2} A_{i_3}^3 (b, \Pi^{1,2} P^{3,4} f)}\nonumber  \\
&+ \sum_{i_1=1}^2 \bbrac{ \Pi^{1,2} A_{i_1}^1 A_{3}^3 (b, P^{3,4}f) - \Pi^{1,2} P^{3,4} A_{i_1}^1 A_{3,3}^{3,4} (b, f)} \nonumber \\
&+ \sum_{(i_3,i_4) \in \ksulku{1,2,3}^2 \setminus \ksulku{3}^2}  \bbrac{P^{3,4} A_{3}^{2} A_{i_3,i_4}^{3,4}(b, \Pi^{1,2} f) -  \Pi^{1,2} P^{3,4} A_{3}^1 A_{i_3,i_4}^{3,4} (b, f)} \\
&+ \sum_{i_2=1}^2 \bbrac{ P^{3,4} A_{i_2}^{2} A_{3,3}^{3,4}(b, \Pi^{1,2} f) - A_{i_2}^{2} A_{3}^3 (b, \Pi^{1,2} P^{3,4} f)}\nonumber \\
&+\bbrac{ \Pi^{1,2} A_{3}^1 A_{3}^3 (b, P^{3,4}f) -  \Pi^{1,2} P^{3,4} A_{3}^1 A_{3,3}^{3,4} (b, f) \nonumber \\
&\qquad- A_{3}^{2} A_{3}^3 (b, \Pi^{1,2} P^{3,4} f) + P^{3,4} A_{3}^{2} A_{3,3}^{3,4}(b, \Pi^{1,2} f)}. \nonumber 
\end{align}

We begin with the first pair. Since $P^{3,4}$ is bounded, it is enough to consider the boundedness of $ \Pi^{1,2} A_{3}^1 A_{i_3}^3 (b, f) - A_{3}^{2} A_{i_3}^3 (b, \Pi^{1,2} f).
$
We show the case $i_3 = 1,$ the other case can be handled similarly. 

First, notice that by the one-parameter expansion $\sum_{I_1 \subset K_1} \Delta_{I_1}^1 \varphi = (\varphi  - \pair{\varphi}_{K_1,1})1_{K_1},$  we have
\begin{align*}
(\pair{\varphi}_{K_1,1} &- \pair{\varphi}_{K_2,2})1_{K_1 \times K_2} \\
&= (\pair{\varphi}_{K_1,1} - \pair{\varphi}_{K_1 \times K_2,1,2} + \pair{\varphi}_{K_1 \times K_2,1,2} - \pair{\varphi}_{K_2,2})1_{K_1 \times K_2}  \\
&=\sum_{I_2 \subset K_2} \Bpair{\varphi, \avgfu{K_1}\otimes h_{I_2} }_{1,2} \otimes h_{I_2} - \sum_{I_1 \subset K_1} \Bpair{\varphi, h_{I_1} \otimes \avgfu{K_2}}_{1,2} \otimes h_{I_1}.
\end{align*}
Using the previous observation, we have
\begin{align*}
 &\Pi^{1,2} A_{3}^1 A_{1}^3 (b, f) - A_{3}^{2} A_{1}^3 (b, \Pi^{1,2} f)\\
&=\sum_{K_1, K_2, K_3} a_{K_1, K_2} \Bbrac{\pair{\pair{b_{K_3}}_{K_1, 1} \pair{f, h_{K_1}\otimes h_{K_3}}_{1,3}}_{K_{2},2} \\
&\hspace{4em}- \pair{b_{K_3}}_{K_2, 2} \Bpair{f, h_{K_1}\otimes \frac{1_{K_2}}{|K_2|} \otimes h_{K_3}}_{1,2,3}}\frac{1_{K_1}}{|K_1|} \otimes h_{K_2} \otimes h_{K_3}h_{K_3} \\
&=\sum_{K_1, K_2, K_3} a_{K_1, K_2} \Bpair{\sulku{\pair{b_{K_3}}_{K_1, 1}- \pair{b_{K_3}}_{K_2,2}}1_{K_1 \times K_2} \pair{f, h_{K_1}\otimes h_{K_3}}_{1,3}, \avgfu{K_2}}_{2} \\
&\hspace{4em}\times  \frac{1_{K_1}}{|K_1|} \otimes h_{K_2} \otimes h_{K_3}h_{K_3} \\
&= \sum_{K_1, K_2, K_3} a_{K_1, K_2} \bigbrac{\sum_{Q_2 \subset K_2} |K_2|^{-1}\Bpair{b, \frac{1_{K_1}}{|K_1|} \otimes h_{Q_2} \otimes h_{K_3}}_{1,2,3}\pair{f, h_{K_1} \otimes h_{Q_2} \otimes h_{K_3}}_{1,2,3} \\
&\qquad- \sum_{Q_1 \subset K_1} \Bpair{b, h_{Q_1} \otimes \frac{1_{K_2}}{|K_2|} \otimes h_{K_3}}_{1,2,3}\pair{f, h_{K_1} \otimes \frac{1_{K_2}}{|K_2|} \otimes h_{K_3}}_{1,2,3} \otimes h_{Q_1}}\\
&\hspace{4em}\times \frac{1_{K_1}}{|K_1|} \otimes h_{K_2} \otimes h_{K_3}h_{K_3},
\end{align*}
where $b_{K_3} = \pair{b, 1_{K_3}/|K_3|}_3.$
These terms are similar to handle and we deal with the first one. By the $L^p$ duality,  we have
\begin{align*}
\Babs{\int_{\R^{d_4}} \sum_{K_1, K_2, K_3} a_{K_1, K_2}|K_2|^{-1}   &\sum_{Q_2 \subset K_2} \Bpair{b_{x_4}, \frac{1_{K_1}}{|K_1|} \otimes h_{Q_2} \otimes h_{K_3}}_{1,2,3} \\
&\times \pair{f_{x_4}, h_{K_1} \otimes h_{Q_2} \otimes h_{K_3}}_{1,2,3}  \\
&\times\Bpair{g_{x_4},\frac{1_{K_1}}{|K_1|} \otimes h_{K_2} \otimes h_{K_3}h_{K_3}}_{1,2,3} \dd x_4}.
\end{align*}
Fix $x_4 \in \R^{d_4}.$  By Proposition \ref{prop:uniformBMO} we need to show the boundedness of 
\begin{align*}
&\int_{\R^{d_1+d_2+d_3}} \sum_{K_1,K_2} |a_{K_1, K_2}| \avgfu{K_1} \\
&\qquad\times\Bsulk{\sum_{Q_2, K_3}\pair{f_{x_4}, h_{K_1} \otimes h_{Q_2} \otimes h_{K_3}}^2 \pair{\abs{\pair{g_{x_4}, h_{K_2}}_2}}_{K_1 \times K_3}^2 \frac{1_{Q_2}}{|Q_2|} \otimes \frac{1_{K_3}}{|K_3|}}^\pinv{2} \frac{1_{K_2}}{|K_2|} \nu_{x_4} \\
&\leq \int_{\R^{d_1+d_2+d_3}} \sum_{K_1} S^{2,3} \pair{f_{x_4}, h_{K_1}}_1 \sum_{ K_2} |a_{K_1, K_2}|  \frac{1_{K_2}}{|K_2|} \otimes \pair{M^3 \pair{g_{x_4}, h_{K_2}}_2}_{K_1,1} \nu_{x_4}.
\end{align*}
We begin by writing
$$
G = \sum_{I_1, K_2} |a_{I_1, K_2}| h_{I_1} \otimes \frac{1_{K_2}}{|K_2|} \otimes \pair{M^3 \pair{g_{x_4}, h_{K_2}}_2}_{I_1,1}.
$$ 

Hence, by standard estimates we get 
\begin{align*}
&\int_{\R^{d_1+d_2+d_3}} \sum_{K_1} S^{2,3} \pair{f_{x_4}, h_{K_1}}_1 \pair{G, h_{K_1}}_1 \frac{1_{K_1}}{|K_1|} \nu_{x_4} \\
&\leq \int_{\R^{d_1+d_2+d_3}} \sqsumm{S^{2,3}\pair{f_{x_4},h_{K_1}}_1}{K_1} \sqsumm{\pair{G, h_{K_1}}_1}{K_1} \nu_{x_4} \\
&\leq \norm{S^{1,2,3} f_{x_4}}{L^{p}(\mu_{x_4})} \norm{S^{1} G}{L^{p'}(\lambda_{x_4}^{1 - p'})} \\
&\lesssim_{[\mu]_{A_p}, [\lambda]_{A_p}} \norm{f}{L^{p}(\mu_{x_4})}\norm{g}{L^{p'}(\lambda_{x_4}^{1 - p'})},
\end{align*}
and applying Hölder's inequality once more to the integral on $\R^{d_4}$ we have the desired bound.

Next, we deal with the second term in \eqref{eq:Parapairing} with $i_1 = 2.$ More precisely, it is enough to consider the term 
\begin{align*}
&\sum_{K_1, K_3, K_4} \sum_{\substack{I_3^{(k_3)} = K_3 \\ J_3^{(l_3)} = K_3}} \pair{a_{K_3, I_3, J_3}, h_{K_4}} h_{K_1} \otimes \Bbrac{\pair{b_{K_1}}_{J_3 ,3} -\pair{ b_{K_1}}_{I_3 \times K_4,3,4}} \\
&\hspace{5em}\times \Bpair{f, \frac{1_{K_1}}{|K_1|} \otimes h_{I_3} \otimes h_{K_4}}_{1,3,4}  \otimes h_{J_3} \otimes \frac{1_{K_4}}{|K_4|} \\
&= \sum_{K_1, K_3, K_4} \sum_{\substack{I_3^{(k_3)} = K_3 \\ J_3^{(l_3)} = K_3}} \pair{a_{K_3, I_3, J_3}, h_{K_4}} h_{K_1} \otimes \\
&\qquad\bigbrac{ \sum_{Q_4 \subset K_4} \Delta_{Q_4}^4\pair{ b_{K_1}}_{J_3,3} + \sum_{t_3 = 1}^{l_3 } \bpair{\Delta_{J_3^{(t_3)}}^3 b_{K_1} }_{J_3 \times K_4, 3,4} - \sum_{s_3 = 1}^{k_3 } \bpair{\Delta_{I_3^{(s_3)}}^3 b_{K_1}}_{I_3 \times K_4, 3,4}} \\
&\hspace{3em}\times\Bpair{f, \frac{1_{K_1}}{|K_1|} \otimes h_{I_3} \otimes h_{K_4}}_{1,3,4}   \otimes h_{J_3} \otimes \frac{1_{K_4}}{|K_4|}
=: \sigma_1 + \sigma_2 - \sigma_3,
\end{align*}
where $b_{K_1} := \pair{b, h_{K_1}}_{1}.$
Terms $\sigma_2$ and $\sigma_3$ are handled similarly.  Therefore, we show the boundedness of $\sigma_1$ and $\sigma_2.$ We begin with the dual form of the first one
\begin{align}
\Babs{\int_{\R^{d_2}}\sum_{K_3, K_4} \sum_{\substack{I_3^{(k_3)} = K_3 \\ J_3^{(l_3)} = K_3}} \pair{a_{K_3, I_3, J_3}, h_{K_4}}|K_4|^{-1} \sum_{\substack{K_1\\Q_4 \subset K_4}}\pair{\pair{ b, h_{K_1} \otimes h_{Q_4} }_{1,4}}_{J_3,3} \nonumber \\
\Bpair{f, \frac{1_{K_1}}{|K_1|} \otimes h_{I_3} \otimes h_{K_4}}_{1,3,4} \pair{g, h_{K_1}  \otimes h_{J_3} \otimes h_{Q_4}}_{1,3,4}}.
\end{align}
Fix $x_2\in \R^{d_2}.$ By Proposition \ref{prop:uniformBMO} it is enough to estimate
\begin{align*}
&\int_{\R^{d_1 + d_3 + d_4}}\sum_{ K_3, K_4} \sum_{\substack{I_3^{(k_3)} = K_3 \\ J_3^{(l_3)} = K_3}} \abs{\pair{a_{K_3, I_3, J_3}, h_{K_4}}} \biggsulk{\sum_{\substack{K_1 ,Q_4 }}\Bpair{f_{x_2}, \frac{1_{K_1}}{|K_1|} \otimes h_{I_3} \otimes h_{K_4}}^2  \\
&\hspace{10em}\times\pair{g_{x_2}, h_{K_1}  \otimes h_{J_3} \otimes h_{Q_4}}^2 \avgfu{K_1} \otimes \avgfu{Q_4}}^\pinv{2} \avgfu{J_3} \otimes \avgfu{K_4} \nu_{x_2}.
\end{align*}
First, we estimate and write
\begin{align*}
\Babs{\Bpair{f_{x_2}, \frac{1_{K_1}}{|K_1|} \otimes h_{I_3} \otimes h_{K_4}}} &\leq  \pair{ M^1 \Delta_{K_3 \times K_4, (k_3,0)}^{3,4} f_{x_2}}_{I_3 \times K_4,3,4} |I_3|^\pinv{2} |K_4|^\pinv{2}\\
&=: \pair{F_{x_2, K_3, K_4}}_{I_3 \times K_4,3,4}|I_3|^\pinv{2} |K_4|^\pinv{2}.
\end{align*}
Hence, we can estimate
\begin{align*}
&\biggsulk{\sum_{\substack{K_1 ,Q_4 }} \pair{ g_{x_2}, h_{K_1} \otimes h_{J_3} \otimes h_{Q_4}}^2 \avgfu{K_1} \otimes \avgfu{Q_4}}^\pinv{2}\\
&\qquad\leq \biggsulk{\sum_{\substack{K_1 ,Q_4 }} \pair{M^{3} \Delta_{K_1 \times K_3 \times Q_4, (0,l_3,0)}g_{x_2}}_{K_1 \times Q_4}^2 1_{K_1} \otimes 1_{Q_4}}^\pinv{2}|J_3|^\pinv{2} =: G_{x_2, K_3} |J_3|^\pinv{2}.
\end{align*}
Using these estimates we get
\begin{align*}
&\int_{\R^{d_1 + d_3 + d_4}}\sum_{ K_3, K_4} \sum_{\substack{I_3^{(k_3)} = K_3 \\ J_3^{(l_3)} = K_3}} \abs{\pair{a_{K_3, I_3, J_3}, h_{K_4}}} \biggsulk{\sum_{\substack{K_1 ,Q_4 }}\Bpair{f_{x_2}, \frac{1_{K_1}}{|K_1|} \otimes h_{I_3} \otimes h_{K_4}}^2  \\
&\hspace{10em}\times\pair{g_{x_2}, h_{K_1}  \otimes h_{J_3} \otimes h_{Q_4}}^2 \avgfu{K_1} \otimes \avgfu{Q_4}}^\pinv{2} \avgfu{J_3} \otimes \avgfu{K_4} \nu_{x_2} \\
&\leq \int_{\R^{d_1 + d_3 }}\sum_{ K_3} \sum_{\substack{I_3^{(k_3)} = K_3 \\ J_3^{(l_3)} = K_3}}|I_3|^\pinv{2} \avgfr{J_3} \\
&\hspace{6em}\times\sum_{K_4} \abs{\pair{a_{K_3, I_3, J_3}, h_{K_4}}}  \pair{F_{x_2, K_3, K_4}}_{I_3 \times K_4,3,4} \pair{G_{x_2, K_3}\nu_{x_2}}_{K_4} |K_4|^{\pinv{2}} \\ 
&\leq   \int_{\R^{d_1 + d_3+d_4 }}\sum_{ K_3 } |K_3|^{-1}  M^{4}(G_{x_2, K_3}\nu_{x_2})
\\
&\hspace{6em}\times\sum_{\substack{I_3^{(k_3)} = K_3 \\ J_3^{(l_3)} = K_3}}  \Bsulk{\sum_{K_4} |I_3|^{2} \pair{F_{x_2, K_3, K_4}}_{I_3 \times K_4,3,4}^2 \otimes 1_{K_4}}^\pinv{2}\otimes 1_{J_3}  \\
&\stackrel{(*)}{\lesssim} \Bnorm{\Bsulk{\sum_{K_3} \Bpair{\Bsulk{\sum_{K_4}\pair{ F_{x_2, K_3, K_4}}_{K_4,4}^2 \otimes 1_{K_4}}^\pinv{2}}_{K_3}^2 \otimes 1_{K_3}}^\pinv{2}}{L^{p}(\lambda_{x_2})} \\
&\qquad \times\Bnorm{\Bsulk{\sum_{K_3}  M^{4}(G_{x_2, K_3}\nu_{x_2})^2 }^\pinv{2}}{L^{p'}({\lambda_{x_2}^{1-p'}})} \\
&\lesssim_{[\mu]_{A_p}[\lambda]_{A_p}}\norm{f_{x_2}}{L^p(\mu_{x_2})}\norm{g_{x_2}}{L^{p'}(\lambda_{x_2}^{1-p'})},
\end{align*}
where in the step $(*)$ along with obvious Hölder's inequalities we used the following application of Kahane-Khintchine's inequality
\begin{align*}
|K|^{-1} \sum_{I^{(k)} = K} \Bsulk{\sum_{i} |I|^2 \pair{\varphi_i}_{I}^2}^\pinv{2} &\sim |K|^{-1}\sum_{I^{(k)} = K} \E\Babs{\sum_{i} \epsilon_i |I| \pair{\varphi_{i}}_{I}} \\
&= |K|^{-1}\sum_{I^{(k)} = K} \E\Babs{\sum_{i} \epsilon_i \int_{I} \varphi_{i}(x) \dd x} \\
&\leq |K|^{-1} \sum_{I^{(k)} = K} \E \int_{I}\Babs{\sum_{i} \epsilon_i  \varphi_{i}(x) }\dd x \\
&\sim \Bpair{\Bsulk{\sum_{i} \abs{\varphi_i}^2 }^\pinv{2} }_{K}.
\end{align*}
After applying Hölder's inequality to the integral on $\R^{d_2},$ we get the desired bound for $\sigma_1.$

 Then take $\sigma_2$  with fixed $t_3 \in [1,l_3].$
By duality and Proposition \ref{prop:uniformBMO}, the term that we are estimating equals to
\begin{align*}
&\int_{\R^{d}} \sum_{K_3} \biggsulk{\sum_{\substack{K_1\\ P_3^{(l_3 - t_3)} = K_3}} \bigbrac{\sum_{\substack{I_3^{(k_3)} = K_3 \\ J_3^{(t_3)} = P_3}} \sum_{K_4} |\pair{a_{K_3, I_3, J_3}, h_{K_4}}| |P_{3}|^{-\pinv{2}}\Babs{\Bpair{f,\frac{1_{K_1}}{|K_1|} \otimes h_{I_3} \otimes h_{K_4}}_{1,3,4}} \\
&\hspace{10em}\times \Babs{\Bpair{g, h_{K_1} \otimes h_{J_3} \otimes \frac{1_{K_4}}{|K_4|}}_{1,3,4}} \otimes\avgfu{K_4}}^2  \otimes \avgfu{K_1} \otimes \avgfu{P_3} }^\pinv{2} \nu. 
\end{align*}

Begin by fixing $x_2 \in \R^{d_2}.$   Then by sparse domination of bilinear paraproducts (see e.g.  Lemma 6.7 in \cite{Li2018toAp}) we can deduce
\begin{align*}
\int_{\R^{d_4}}& \sum_{K_4} |\pair{a_{K_3, I_3, J_3}, h_{K_4}}|\Babs{\Bpair{f_{x_2},\frac{1_{K_1}}{|K_1|} \otimes h_{I_3} \otimes h_{K_4}}}  \Babs{\Bpair{g_{x_2}, h_{K_1} \otimes h_{J_3} \otimes \frac{1_{K_4}}{|K_4|}}} \avgfu{K_4} \nu_{x_2} \\
&\lesssim_{[\nu]_{A_\infty}} \frac{|I_3|^\pinv{2}|J_3|^\pinv{2}}{|K_3|}\int_{\R^{d_4}} M^{4}\Bpair{f_{x_2},\frac{1_{K_1}}{|K_1|} \otimes h_{I_3}}_{1,3} M^{4} \pair{g_{x_2}, h_{K_1} \otimes h_{J_3}}_{1,3}  \nu_{x_2} \\
&\leq \frac{|I_3||J_3||K_1|^\pinv{2}}{|K_3|} \int_{\R^{d_4}} \pair{M^4 \Delta_{K_3, k_3}^3 f_{x_2}}_{K_1 \times I_3,1,3}   \pair{M^{4} \Delta_{K_1, K_3, (0,l_3)}^{1,3}g_{x_2} }_{K_1\times J_3, 1,3} \nu_{x_2}.
\end{align*}
Using the previous estimates and $A_\infty$-extrapolation we get that our term is bounded by 
\begin{align*}
&\sum_{K_3} \int_{\R^{d_1+  d_3 +d_4}}\Bsulk{\sum_{\substack{K_1\\ P_3^{(l_3 - t_3)} = K_3}} \pair{M^4 \Delta_{K_3, k_3}^3 f_{x_2}}_{K_1 \times K_3,1,3}^2 \\
&\hspace{10em}\times \pair{M^{4} \Delta_{K_1, K_3, (0,l_3)}^{1,3}g_{x_2} }_{K_1\times P_3, 1,3}^2\otimes 1_{K_1} \otimes 1_{P_3}}^\pinv{2} \nu_{x_2} \\
&\leq \int_{\R^{d_1+  d_3 +d_4}} \Bsulk{\sum_{K_3} \pair{M^{1,4} \Delta_{K_3, k_3}^3 f_{x_2}}_{ K_3,3}^2 \otimes 1_{K_3}}^\pinv{2} \\
&\hspace{6em} \Bsulk{\sum_{K_1, K_3} \pair{M^{3,4} \Delta_{K_1, K_3, (0,l_3)}^{1,3}g_{x_2} }_{K_1 , 1,3}^2 \otimes 1_{K_1} }^\pinv{2} \nu_{x_2} \\
&\lesssim_{[\mu]_{A_p}, [\lambda]_{A_p}} \norm{f}{L^{p}(\mu_{x_2})}\norm{g}{L^{p'}(\lambda_{x_2}^{1 - p'})}.
\end{align*}
Apply Hölder's inequality to the integral on $\R^{d_2}$ to get the desired boundedness.

The third and fourth terms in \eqref{eq:Parapairing} are similar. Thus, we only need to take care of the last term 
\begin{align*}
&\sum_{\substack{K_1,K_2\\K_3,K_4}}\sum_{\substack{I_3^{(k_3)} = K_3\\ J_3^{(l_3)}  = K_3}} a_{K_1, K_2}\pair{a_{K_3, I_3,J_3}, h_{K_4}} 
A^{f,b}_{K_1, K_2, I_3, J_3, K_4} \avgfu{K_1} \otimes h_{K_2} \otimes h_{J_3} \otimes \frac{1_{K_4}}{|K_4|},
\end{align*}
where
$$
A^{f,b}_{K_1, K_2, I_3, J_3, K_4}  = \pair{B \pair{f, h_{K_1} \otimes h_{I_3} \otimes h_{K_4}}_{1,3,4}}_{K_2, 2}
$$
and 
\begin{align*}
B = \pair{b}_{K_1 \times J_3,1,3} - \pair{b}_{K_1 \times I_3 \times K_4,1,3,4} - \pair{b}_{K_2 \times J_3,2,3} + \pair{b}_{K_2 \times I_3 \times K_4,2,3,4}.
\end{align*}

We write
\begin{align*}
B &=\sum_{\substack{Q_2\subset K_2 \\ Q_4 \subset K_4}}\Delta_{Q_2 \times Q_4} \pair{b}_{K_1 \times J_3,1,3} + \sum_{t_3 = 1}^{l_3} \sum_{Q_2\subset K_2} \pair{\Delta_{J_3^{(t_3)}\times Q_2}\pair{b}_{K_1 \times  K_4,1,4}}_{J_3,3} \\
&\qquad- \sum_{s_3 = 1}^{k_3} \sum_{Q_2\subset K_2} \pair{\Delta_{I_3^{(s_3)} \times Q_2}\pair{b}_{K_1 \times  K_4,1,4}}_{I_3,3}  - \sum_{\substack{Q_1\subset K_1 \\ Q_4 \subset K_4}}\Delta_{Q_1 \times Q_4} \pair{b}_{K_2 \times J_3,2,3} \\
&\qquad - \sum_{t_3 = 1}^{l_3} \sum_{Q_1\subset K_1} \pair{\Delta_{J_3^{(t_3)}\times Q_1}\pair{b}_{K_2 \times  K_4,2,4}}_{J_3,3}  + \sum_{s_3 = 1}^{k_3} \sum_{Q_1\subset K_1} \pair{\Delta_{I_3^{(s_3)} \times Q_1}\pair{b}_{K_2 \times  K_4,2,4}}_{I_3,3} \\
&=: \sum_{i = 1}^6 B_{i}.
\end{align*}
We consider the terms associated to $B_1$ and $B_2$ since the rest can be estimated similarly.

The dual form of the term associated to $B_1$ equals to
\begin{align*}
&\Big|\sum_{\substack{K_1,K_2\\K_3,K_4}}\sum_{\substack{I_3^{(k_3)} = K_3\\ J_3^{(l_3)}  = K_3}} a_{K_1, K_2}\pair{a_{K_3, I_3,J_3}, h_{K_4}}|K_2|^{-1} |K_4|^{-1}  \\
&\qquad\times \sum_{\substack{Q_2\subset K_2 \\ Q_4 \subset K_4}}  \Bpair{b, \avgfu{K_1} \otimes h_{Q_2} \otimes \avgfu{J_3} \otimes h_{Q_4} }\pair{f, h_{K_1} \otimes h_{Q_2} \otimes h_{I_3} \otimes h_{K_4}}\\
&\qquad\hspace{4em}\times\Bpair{g, \avgfu{K_1} \otimes h_{K_2} \otimes h_{J_3} \otimes h_{Q_4}}\Big|.
\end{align*}
By Proposition \ref{prop:uniformBMO} it is enough to estimate the following term
\begin{align*}
&\int_{\R^{d}} \Bsulk{\sum_{\substack{K_1,K_2\\K_3,K_4}}\sum_{\substack{I_3^{(k_3)} = K_3\\ J_3^{(l_3)}  = K_3}} |a_{K_1, K_2}||\pair{a_{K_3, I_3,J_3}, h_{K_4}}| \avgfu{K_1} \otimes  \avgfu{K_2}S^{2} \pair{f, h_{K_1} \otimes h_{I_3} \otimes h_{K_4}}_{1,3,4}\\
&\hspace{4em} \otimes \avgfu{J_3} \otimes \avgfu{K_4}S^{4}\Bpair{g, \avgfu{K_1} \otimes h_{K_2} \otimes h_{J_3} }_{1,2,3} } \nu.
\end{align*}
First, we write
$$
G := \sum_{Q_1, K_2, Q_3} |a_{Q_1, K_2}| h_{Q_1}\otimes \avgfu{K_2} \otimes h_{Q_3} \otimes  S^4 \Bpair{g, \avgfu{Q_1} \otimes h_{K_2} \otimes h_{Q_3} }_{1,2,3}.
$$
Thus  we get
\begin{align*}
&\int_{\R^{d_1 +d_2 +d_3}}\sum_{K_1,K_3}\sum_{\substack{I_3^{(k_3)} = K_3\\ J_3^{(l_3)}  = K_3}} \avgfu{K_1} \otimes \avgfu{J_3} \\
&\hspace{4em}\times\sum_{K_4}|\pair{a_{K_3, I_3,J_3}, h_{K_4}}|    S^{2} \pair{f, h_{K_1} \otimes h_{I_3} \otimes h_{K_4}}_{1,3,4} \pair{\pair{G, h_{K_1} \otimes h_{J_3}}_{1,3} \nu }_{K_4, 4} \\
&\lesssim  \int_{\R^d} \sum_{K_1,K_3}\sum_{\substack{I_3^{(k_3)} = K_3\\ J_3^{(l_3)}  = K_3}} \avgfu{K_1} \otimes \avgfu{J_3} \frac{|I_3|^\pinv{2} |J_3|^\pinv{2}}{|K_3|} \\
&\hspace{4em} \times \Bsulk{\sum_{K_4} S^{2} \pair{f, h_{K_1} \otimes h_{I_3} \otimes h_{K_4}}_{1,3,4}^2 \pair{\pair{G, h_{K_1} \otimes h_{J_3}}_{1,3} \nu }_{K_4, 4}^2 \avgfu{K_4}}^\pinv{2}.
\end{align*}

Next, we estimate 
\begin{align*}
|\pair{\pair{G, h_{K_1} \otimes h_{J_3}}_{1,3} \nu }_{K_4, 4}| 
&\leq M^{4} \bbrac{ \pair{M^3\Delta_{K_1 \times K_3, (0,l_3)}^{1,3} G}_{K_1,1} \nu}|J_3|^\pinv{2}|K_1|^\pinv{2}
\end{align*}
and 
\begin{align*}
\Bsulk{\sum_{K_4} S^{2} \pair{f, h_{K_1} \otimes h_{I_3} \otimes h_{K_4}}_{1,3,4}^2 \otimes \avgfu{K_4}}^\pinv{2} \lesssim \pair{S^{2,4} \Delta_{K_1 \times K_3, (0, k_3)}^{1,3} f}_{K_1 \times I_3, 1,3}|I_3|^\pinv{2}|K_1|^\pinv{2}.
\end{align*}

Hence, we have 
\begin{align*}
&  \int_{\R^d} \sum_{K_1,K_3} M^{4} \bbrac{M^3 \pair{\Delta_{K_1 \times K_3, (0,l_3)}^{1,3} G}_{K_1,1} \nu} \\
&\hspace{4em}\times\pair{S^{2,4} \Delta_{K_1 \times K_3, (0, k_3)}^{1,3} f}_{K_1 \times K_3, 1,3} \otimes 1_{K_1} \otimes 1_{K_3}  \\
&\lesssim \Bnorm{\Bsulk{\sum_{K_1,K_3} M^{4} \bbrac{M^3\pair{\Delta_{K_1 \times K_3, (0,l_3)}^{1,3} G}_{K_1,1} \nu}^2 1_{K_1} \otimes 1_{K_3}}^\pinv{2}}{L^{p'}(\mu^{1-p'})}\\
&\qquad \times\Bnorm{\Bsulk{\sum_{K_1,K_3} \pair{ S^{2,4} \Delta_{K_1 \times K_3, (0, k_3)}^{1,3} f}_{K_1 \times K_3, 1,3}^2 \otimes 1_{K_1} \otimes 1_{K_3}}^\pinv{2}}{L^{p}(\mu)} \\
&\lesssim_{[\mu]_{A_p},[\lambda]_{A_p}} \norm{g}{L^{p'}(\lambda^{1-p'})} \norm{f}{L^p(\mu)}.
\end{align*}

Next, fix the integer $t_3 \in [1,l_3]$ and consider the dual form of the  term $B_2$
\begin{align*}
&\Big|\sum_{\substack{K_1,K_2\\K_3,K_4}}\sum_{\substack{I_3^{(k_3)} = K_3\\ J_3^{(l_3)}  = K_3}} a_{K_1, K_2}\pair{a_{K_3, I_3,J_3}, h_{K_4}} \\
&\qquad\times \sum_{Q_2\subset K_2} |K_2|^{-1} |J_3^{(t_3)}|^{-\pinv{2}}\Bpair{b, \avgfu{K_1} \otimes h_{Q_2} \otimes h_{J_3^{(t_3)}} \otimes \avgfu{K_4}}  \\
&\qquad\times \pair{f, h_{K_1} \otimes h_{Q_2} \otimes h_{I_3} \otimes h_{K_4}} \Bpair{g, \avgfu{K_1} \otimes h_{K_2} \otimes h_{J_3} \otimes \avgfu{K_4}} \Big|. 
\end{align*}
Thus, by Proposition \ref{prop:uniformBMO} we are estimating the following term
\begin{align*}
&\int_{\R^{d}} \sum_{\substack{K_1,K_2\\K_3}}  |a_{K_1, K_2}| \avgfu{K_1} \otimes \avgfu{K_2}  \\
&\hspace{3em}\times\biggsulk{\sum_{\substack{Q_2 \\ P_3^{(l_3 - t_3)} = K_3}} \Bsulk{\sum_{\substack{I_3^{(k_3)} = K_3 \\ J_3^{(t_3)} = P_3}} |P_3|^{-\pinv{2}} \pi^4(f_{K_1, Q_2, I_3}, g_{K_1, K_2, J_3}) }^2 \avgfu{Q_2} \otimes \avgfu{P_3}  }^\pinv{2}\nu,
\end{align*}
where 
\begin{align*}
&\pi^4(f_{K_1, Q_2, I_3}, g_{K_1, K_2, J_3}) \\
 &:= \pi_{K_3,I_3,J_3}^4\Bsulk{\pair{f,h_{K_1} \otimes h_{Q_2} \otimes h_{I_3}}_{1,2,3}, \Bpair{g,\avgfu{K_1}\otimes h_{K_2} \otimes h_{J_3}}_{1,2,3}}
\end{align*}
 is a bilinear one-parameter paraproduct such that $a_{K_4}\pair{\varphi, h_{K_4}}\pair{\phi}_{K_4}$ is replaced by\\ $\abs{a_{K_4}\pair{\varphi, h_{K_4}}\pair{\phi}_{K_4}}$. 
As previously, by sparse domination and $A_\infty$-extrapolation we get
\begin{align*}
&\int_{\R^{d}} \sum_{\substack{K_1,K_2\\K_3}}  |a_{K_1, K_2}| \avgfu{K_1} \otimes \avgfu{K_2}  \\
&\hspace{3em}\times\Bsulk{\sum_{\substack{Q_2 \\ P_3^{(l_3 - t_3)} = K_3}} \Bsulk{\sum_{\substack{I_3^{(k_3)} = K_3 \\ J_3^{(t_3)} = P_3}} |P_3|^{-\pinv{2}} \pi^4(f_{K_1, Q_2, I_3}, g_{K_1, K_2, J_3}) }^2 \avgfu{Q_2} \otimes \avgfu{P_3}  }^\pinv{2}\nu \\
&\lesssim_{[\nu]_{A_\infty}} \int_{\R^{d}} \sum_{\substack{K_1,K_2\\K_3}}  |a_{K_1, K_2}| \frac{1_{K_1}}{|K_1|^\pinv{2}}\otimes \frac{1_{K_2}}{|K_2|^\pinv{2}} \otimes 1_{K_3}   \pair{M^{3, 4} \Delta_{K_2 \times K_3, (0, l_3)}^{2,3} g}_{K_1 \times K_2, 1,2} \\
&\hspace{4em}\times\Bsulk{\sum_{Q_2} \pair{M^4 \Delta_{K_1 \times Q_2 \times K_3,(0,0,k_3)}^{1,2,3} f}_{K_1 \times Q_2 \times K_3,1,2,3}^2 \otimes 1_{Q_2} }^\pinv{2}\nu \\
&\lesssim \int_{\R^{d}} \sum_{K_3} 1_{K_3} \Bsulk{\sum_{K_1, K_2} \pair{M^{3, 4} \Delta_{K_2 \times K_3, (0, l_3)}^{2,3} g}_{K_1 \times K_2 , 1,2}^2 \\
&\quad\times \Bpair{\Bsulk{\sum_{Q_2} \pair{M^4 \Delta_{K_1 \times Q_2 \times K_3,(0,0,k_3)}^{1,2,3} f}_{K_1 \times Q_2 \times K_3,1,2,3}^2 \otimes 1_{Q_2} }^\pinv{2}\nu}_{K_1 \times K_2, 1,2}^2 1_{K_1} \otimes 1_{K_2}}^\pinv{2} \\
&\leq \int_{\R^d}\Bsulk{\sum_{ K_2, K_3} \pair{M^{1,3,4} \Delta_{K_2 \times K_3, (0, l_3)}^{2,3} g}_{ K_2, 2}^2  \otimes 1_{K_2}  }^\pinv{2} \\
&\quad \biggsulk{\sum_{K_1,K_3}  M^2 \Bpair{\Bsulk{\sum_{Q_2} \pair{M^4 \Delta_{K_1 \times Q_2 \times K_3,(0,0,k_3)}^{1,2,3} f}_{K_1 \times Q_2 \times K_3, 1,2,3}^2 \otimes 1_{Q_2} }^\pinv{2} \nu}_{K_1, 1}^2 1_{K_1 \times K_3}}^\pinv{2} \\
&\lesssim_{[\mu]_{A_p},[\lambda]_{A_p}}  \norm{g}{L^{p'}(\lambda^{1-p'})} \norm{f}{L^p(\mu)}.
\end{align*}
\end{proof}

We return to consider the space $\R^d = \R^{d_1} \times \R^{d_2} \times \dots \times \R^{d_m}.$ 

\begin{thm}
Let $\nu = \mu^\pinv{p} \lambda^{-\pinv{p}},$ where $\mu, \lambda \in A_p(\R^{d_1} \times \cdots \times \R^{d_m})$ in $\R^d$) and $1<p<\infty.$
In addition, let $\bI  = \ksulku{\bI_i}_{i = 1}^k$ be a partition of $\colM.$ For given CZO $T^{\seq{v}_i},$ where $\seq{v}_i = (j)_{j \in \bI_i},$ suppose that at least one of the following conditions holds:
\begin{enumerate}
\item the CZO $T_i$ is paraproduct free, or
\item $\#\bI_i \le 2$
\end{enumerate}
for all $i = 1, \ldots, k.$
Then we have
$$
\norm{[T^{\seq{v}_1},[T^{\seq{v}_2}, \dots, [b,T^{\seq{v}_k}]]]}{L^p(\mu) \rightarrow L^{p}(\lambda)}
\lesssim \normbmo{b}{\bI}{\nu}.
$$
\end{thm}

Here even with the case $k=3$ with bi-parameter operators, we have a collection 27 commutators. Actually, even more, when counting different forms of paraproducts. We can use the same strategy as in the case $k=2$ also here and essentially nothing really changes. Clearly, the number of paraproduct coefficients increase but techniques used in the case $k=2$ also apply to these situations. The previous theorem is not stated for paraproduct free CZOs. However, if we combine techniques of Theorem \ref{thm:iteratedShift}, we can allow paraproduct free CZOs of arbitrary parameters. We omit the details.
Furthermore, we remark that the case $k=1$ is proven in \cite{Li2018Bloom} for the bi-parameter CZOs.

\bibliographystyle{habbrv}

\end{document}